\documentclass[19pt]{article}

\usepackage[left=1in,top=1in,right=1in,bottom=1in,letterpaper]{geometry}

\usepackage{graphicx} % Required for inserting images

\usepackage{amssymb,amsmath,amsthm}
\usepackage{enumerate}
\usepackage{latexsym,amsfonts,amscd,amsxtra,amstext}
\usepackage{algorithm}
\usepackage{algorithmic}
\usepackage{xspace}
\usepackage{mathtools}
\usepackage{booktabs}
\usepackage{threeparttable}
\usepackage{color}
\mathtoolsset{showonlyrefs=true}
\usepackage{dsfont}

\usepackage[round]{natbib}

\usepackage[colorlinks,linkcolor=red,filecolor=blue,citecolor=blue,urlcolor=blue]{hyperref}

\usepackage{varwidth}

\newtheorem{lemma}{Lemma}
\newtheorem{assumption}{Assumption}
\newtheorem{remark}{Remark}
\newtheorem{example}{Example}

\usepackage{xcolor,pifont}
\newcommand{\cmark}{\textcolor{green}{\ding{52}}}
\newcommand{\xmark}{\textcolor{red}{\ding{55}}}

\newcommand{\eg}{{\em e.g.\xspace}}
\newcommand{\ie}{{\em i.e.\xspace}}
\newcommand{\sgn}{\mathrm{sign}}

\newcommand{\EE}{\mathbb{E}}
\newcommand{\RR}{\mathbb{R}}

\newcommand{\cC}{\mathcal{C}}
\newcommand{\cD}{\mathcal{D}}

\newcommand{\cF}{\mathcal{F}}
\newcommand{\cI}{\mathcal{I}}

\newcommand{\cS}{\mathcal{S}}

\newcommand{\one}{\mathds{1}}

\usepackage{colortbl}
\definecolor{Ocean}{RGB}{129,194,234}
\definecolor{BPink}{rgb}{0.96, 0.76, 0.76}

\newtheorem{theorem}{Theorem}
\newtheorem{definition}{Definition}
\newtheorem{corollary}{Corollary}

\newcommand{\fedavg}{{\sc FedAvg}\xspace}
\newcommand{\scafcom}{{\sc SCAFCOM}\xspace}
\newcommand{\scaffold}{{\sc SCAFFOLD}\xspace}
\newcommand{\scallion}{{\sc SCALLION}\xspace}

\newcommand{\fedcom}{{\sc FedCOMGATE}\xspace}
\newcommand{\vrlsgd}{{\sc VRL-SGD}\xspace}
\newcommand{\fedpaq}{{\sc FedPAQ}\xspace}
\newcommand{\fedef}{{\sc Fed-EF}\xspace}

\title{\bf Stochastic Controlled Averaging for Federated Learning with Communication Compression}
 \author{Xinmeng Huang, Ping Li, Xiaoyun Li}
\date{}

\begin{document}
\maketitle

\begin{abstract}
\noindent\footnote{The work is conducted at LinkedIn --- Bellevue, 98004 WA, USA. Xinmeng Huang is currently a Ph.D. student in the Graduate Group of Applied Mathematics and Computational Science at the University of Pennsylvania.}
Communication compression, a technique aiming to reduce the information volume to be transmitted over the air, has gained great interest in Federated Learning (FL) for the potential of alleviating its communication overhead. However, communication compression brings
forth new challenges in FL due to the interplay of compression-incurred information distortion and inherent characteristics of FL such as partial participation and data heterogeneity.
     Despite the recent development,  
     the performance of compressed FL approaches has not been fully exploited.    
     The existing approaches either cannot accommodate 
     arbitrary data heterogeneity or partial
participation, or require stringent conditions on compression.

\vspace{0.2in}
\noindent In this paper, we revisit the seminal stochastic controlled averaging method by proposing an equivalent but more efficient/simplified formulation with halved uplink communication costs. Building upon this implementation, we
propose two compressed FL algorithms, \scallion and  \scafcom, to support unbiased and biased compression, respectively.
Both the proposed methods outperform the existing compressed FL methods in terms of communication and computation complexities.
Moreover, \scallion and \scafcom accommodate arbitrary data heterogeneity and do not make any additional assumptions on compression errors.
Experiments show that \scallion and  \scafcom can match the performance of corresponding full-precision FL approaches with substantially reduced uplink communication, and outperform recent compressed FL methods under the same communication budget. 
\end{abstract}

\newpage

\section{Introduction}
Federated learning (FL) is a powerful paradigm for large-scale machine learning~\citep{konevcny2016federated,mcmahan2017communication}. In situations where data and computational resources are dispersed among diverse clients such as phones, tablets, sensors, banks, hospitals, and other devices and agents, federated learning facilitates local data processing and collaboration among these clients~\citep{kairouz2021advances}. FL enjoys the advantage of distributed optimization on the efficiency of computational resources as the local clients conduct computations simultaneously. Moreover, since the centralized model is trained without transmitting decentralized data from clients directly to servers, FL provides the first layer of protection of data privacy as the local data never leaves the local device. 

Due to its nature and application scenarios, federated learning encounters several significant challenges in algorithmic development and theory~\citep{yang2020federated,li2023analysis}:
\begin{itemize}
\item \textbf{Severe data heterogeneity.} 
Unlike in classic distributed training, the local data distribution in FL can vary significantly (\ie, non-iid clients), reflecting practical scenarios where local data held by clients is highly personalized~\citep{zhao2018federated,kairouz2021advances,yuan2023removing,li2022federated}. When multiple local training steps are taken, the local models become ``biased'' toward minimizing the local losses instead of the global loss, hindering the convergence quality of the  global model~\citep{mohri2019agnostic,li2020convergence,li2020federated}.

\item \textbf{Partial client participation.} 
Another practical issue in FL systems is partial participation, where not all clients can always join the training, \eg, due to unstable connections or active selection~\citep{li2020federated}. Consequently, only a fraction of clients are involved in each FL training round to interact with the central server. This slows down the convergence of the global model because of less accessible data/information per round~\citep{charles2021large,chen2022optimal,li2023analysis}.

\item \textbf{Heavy communication workload.} 
The cost of model transmission can be a major challenge in FL systems with limited bandwidth (\eg, portable wireless devices), especially when training large models with millions or billions of model parameters. Therefore, communication compression, a technique that aims to reduce the volume of information transmitted, has gained growing research interests in FL~\citep{basu2019qsparse,reisizadeh2020fedpaq,haddadpour2021federated,li2023analysis}.
\end{itemize}

The classic FL approach, {\sc FedAvg}~\citep{konevcny2016federated,mcmahan2017communication,stich2019local,yu2019linear,lin2020don,wang2021cooperative}, performs multiple gradient-descent steps within each accessible client before communicating with the central server. While showing success in certain scenarios, {\sc FedAvg} is notably hampered by data heterogeneity  and partial client participation~\citep{karimireddy2020scaffold, li2020convergence,yang2021achieving} due to the ``client drift'' effect. Furthermore, when communication compression is employed, the adverse effect of data heterogeneity can be amplified due to the interplay of client drift and inaccurate message aggregation caused by compression~\citep{basu2019qsparse,reisizadeh2020fedpaq,haddadpour2021federated,gao2021convergence,malekijoo2021fedzip,li2023analysis}; see  Figure~\ref{fig:fedavg-c} for illustration. The inaccurate aggregation incurred by compression imposes more obstacles to obtaining stable and robust performance in FL systems, particularly when data heterogeneity is severe, and compressors are biased (\ie, the compressed output is a biased estimate of input)~\citep{li2023analysis,gao2021convergence,basu2019qsparse}.

\begin{figure}
    \centering
    \includegraphics[width=0.85\textwidth]{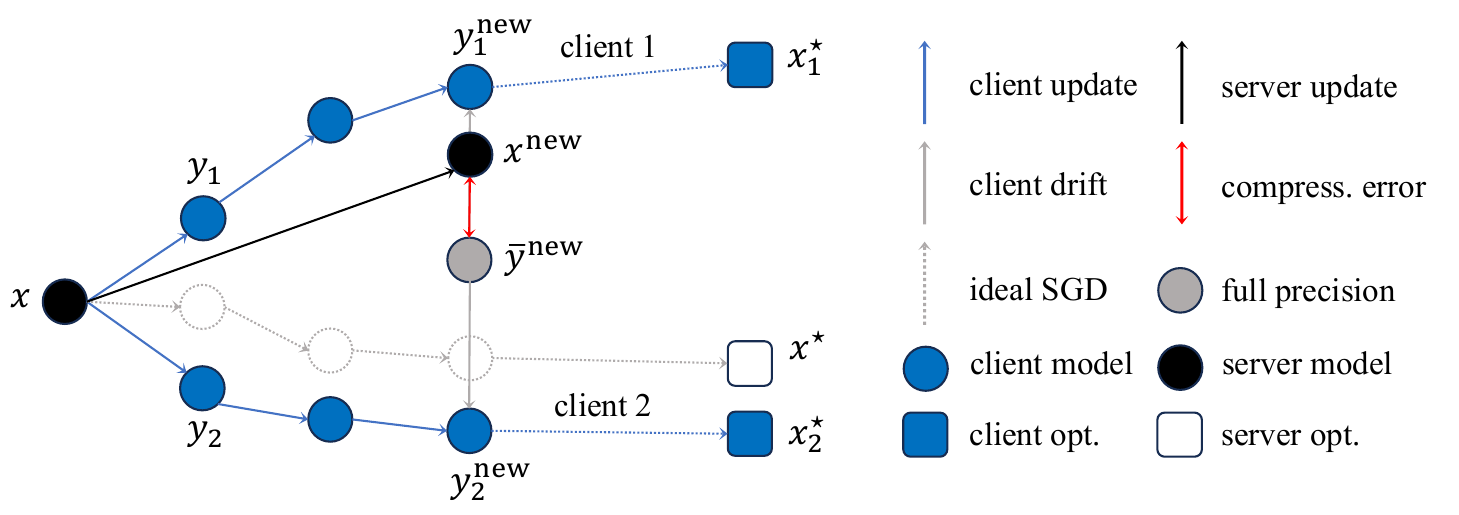}
    \caption{\small Interplay of client-drift and inaccurate message aggregation incurred by communication compression in {\sc FedAvg} is illustrated for $2$ clients with $3$ local steps (\ie, $S=N=2$, $K=3$). The client updates $y_i$ (blue circle)
move towards the individual client optima $x_i^\star$
(blue square).
The server updates (black circle) move towards a {\em distorted proxy}, depending on the degree of compression, of the full-precision averaged model  $\frac{1}{N}\sum_{i=1}^N x_i^\star$ (grey circle),
instead of 
the true optimum  $x^\star$
(white square).}
    \label{fig:fedavg-c}
\end{figure}

While having the potential to reduce communication costs in FL, communication compression brings forth new challenges in addition to FL's inherent characteristics like local updates and partial participation. This naturally raises the question regarding the utility of compressed FL approaches:
\begin{center}
    {\em Can we design FL approaches that accommodate  arbitrary data heterogeneity, local updates, \\ and partial participation, as well as support communication
compression?}
\end{center}

Despite several attempts, none of the existing algorithms have successfully achieved this goal for non-convex FL, to the best of our knowledge. For instance,  {\sc FedPAQ}~\citep{reisizadeh2020fedpaq}, {\sc FedCOM}~\citep{haddadpour2021federated}, {\sc QSPARSE-SGD}~\citep{basu2019qsparse}, {\sc Local-SGD-C}~\citep{gao2021convergence} consider compressed FL algorithms under homogeneous data (\ie, iid clients). {\sc FedCOMGATE}~\citep{haddadpour2021federated}, designed for unbiased compressors, does not support biased compressors, and their analysis 
does not validate the utility under partial client participation.
% {\sc Fed-EF}~\citep{li2023analysis} focuses on biased communication compression in FL with error feedback~\citep{seide20141,karimireddy2019error}, 
% and considers both SGD and adaptive optimizer (AMSGrad~\citep{reddi2019convergence}) for central server update, 
{\sc Fed-EF}~\citep{li2023analysis} focuses on biased communication compression in FL with error feedback~\citep{seide20141,karimireddy2019error} 
and partial client participation. However, the convergence analysis requires the assumption of bounded gradient dissimilarity on the data heterogeneity and shows an extra slow-down factor in the convergence rate under partial client participation, suggesting a theoretical limitation of error feedback in FL. Moreover, both~\cite{haddadpour2021federated} and~\cite{li2023analysis} impose stringent conditions on compression errors (see Remark~\ref{rmk:cond} for more details). 
% These limitations motivate us to develop new compressed FL approaches that are easy to implement, robust to data heterogeneity and partial participation, support both biased and unbiased compressors, and exhibit superior theoretical convergence.

Given these limitations, the motivation of this work is to develop new compressed FL approaches that are practical to implement, robust to data heterogeneity and partial participation, support both biased and unbiased compressors, and exhibit superior theoretical convergence.

\subsection{Main Results \& Contributions}
In this paper, we propose two algorithms, \scallion and \scafcom, which cover unbiased and biased compression and offer enhanced communication efficiency, faster convergence rates, and robustness to arbitrary data heterogeneity and partial participation. Table~\ref{tab:full_results} presents a comprehensive comparison of communication and computation complexities and associated restrictions
of existing algorithms, as well as our newly proposed approaches.

It is worth emphasizing that, the theoretical analysis in our paper only requires the smoothness of local objectives and bounded variance of stochastic gradients, without any additional assumptions on data heterogeneity or compression errors (see Remark~\ref{rmk:cond} for more details), as opposed to all prior related works. The keys to this significant improvement are our new formulation of stochastic controlled averaging and the introduction of momentum. The main contributions are:

\begin{itemize}
    \item We revisit the \scaffold method~\citep{karimireddy2020scaffold} by proposing a simplified and more communication-efficient formulation. The new implementation reduces the uplink communication cost by half, requiring each client to transmit only one increment variable (of the same size as the model) when participating in a training round, instead of two variables in the original implementation~\citep{karimireddy2020scaffold}.

    \item Building upon our new formulation of \scaffold, we propose the \scallion method that employs unbiased compressors for the communication of increment variables. We establish its convergence result for non-convex objectives. \scallion obtains the state-of-the-art  communication and computation complexities for FL under unbiased compressors and supports partial client participation.

    \item We further develop \scafcom which enables biased compressors for broader applications. Local momentum is applied to guarantee fast convergence and improve empirical performance. The communication and computation complexities of \scafcom improve prior results by significant margins, particularly when compression is aggressive.

    % \item We further incorporate  momentum to  enable biased compression, proposing the \scafcom method with established convergence for non-convex objectives. \scafcom 
    % accommodates both unbiased and biased compression and outperforms existing compressed FL methods by margins depending on compression ratios.

    \item We conduct experiments to illustrate the effectiveness of \scallion and \scafcom and support our theories. Our empirical results show that the proposed methods achieve comparable performance to full-precision FL methods with substantially reduced communication costs, and outperform recent compressed FL methods under the same communication budget.
\end{itemize}

\begin{table}[t!]
    % \caption{\footnotesize Full participation}
    \caption{\small 
The comparison of compressed FL algorithms when {\bf all clients} participate in training. Notation $N$ is the number of clients,
% $K$ is the number of  local steps per round, 
$\epsilon$ is a target for the stationary condition such that $\EE[\|\nabla f(\hat x)\|^2]\leq \epsilon$, $\omega$ and $q$ are compression-related parameters (see Definition~\ref{def:unbiased} and Definition~\ref{def:contract}).
% $\zeta^2$ is the uniform bound of data heterogeneity $(1/N)\sum_{1\leq i\leq N}\|\nabla f_i(x)-\nabla f(x)\|^2$. 
Parameters such as smoothness $L$, the variance of stochastic gradients $\sigma^2$, and additional ones introduced in compared works (see Remark~\ref{rmk:cond}) are omitted for clarity.  Here, \#{\bf A. Comm.} denotes the  total number of communication rounds  when $\sigma\to0$ asymptotically while \#{\bf A. Comp.} denotes the total number of gradient evaluations required per client when $\epsilon\to 0$ asymptotically (see the discussion in Section~\ref{sec:scallion-conv});  
{\bf P.P.} denotes allowing partial client participation; 
{\bf D.H.}  denotes allowing arbitrary data heterogeneity;
{\bf S.C.} denotes only requiring standard compressibilities (\ie, Definition~\ref{def:unbiased} and Definition~\ref{def:contract}).
% Constants such as smoothness parameter $L$ of objective functions and variance of gradient noises $\sigma^2$, and the initialization gap $f(x^0)-\min_x f(x)$ are omitted for clarity. Note that~\citep{haddadpour2021federated,li2023analysis,basu2019qsparse} rely on additional assumptions on data heterogeneity, gradient norm, or averaged compression (see detailed discussion in Sec.~{add ref here}).
}\label{tab:full_results}
    \vspace{2mm}
    \centering{
    % \small
    % \scalebox{0.95}
    {
    \begin{threeparttable}
    \begin{tabular}{ccccccc}
    \toprule
    {\bf Algorithm}  & \#{\bf A. Comm.} & \#{\bf A. Comp.}  &  {\bf P.P.} & {\bf D.H.} & {\bf S.C.}\\ 
    \midrule 
    {\sc Unbiased Compression}\vspace{2mm}\\
    \quad {\sc FedPAQ}~\citep{reisizadeh2020fedpaq} & $\frac{1+\omega /N}{\epsilon}^\natural$ & $\frac{1+\omega}{N\epsilon^2}$ &  \cmark & \xmark & \cmark\vspace{1.5mm}\\
    \quad {\sc FedCOM}~\citep{haddadpour2021federated} & $\frac{1+\omega/N}{\epsilon}^\natural$ & $\frac{1+\omega}{N\epsilon^2}$ &  \xmark & \xmark & \cmark \vspace{1.5mm}\\
    \quad {\sc FedCOMGATE}~\citep{haddadpour2021federated} & $\frac{1+\omega}{\epsilon}$ & $\frac{1+\omega}{N\epsilon^2}$ &  \xmark & \cmark & \xmark \vspace{1.5mm}\\
    % \quad {\sc Fed-EF}~\citep{li2023analysis} & $\frac{(1+\omega)^2}{\epsilon}$ & $\frac{(1+\omega)^2}{N\epsilon^2}$ &  \cmark & \xmark & \xmark \vspace{1.5mm}\\
    \quad {\bf SCALLION (Theorem~\ref{thm:scallion})}  & $\boldsymbol{\frac{1+\omega}{\epsilon}}$ &  $\boldsymbol{\frac{1+\omega}{N\epsilon^2}}$  &  \cmark & \cmark & \cmark \vspace{1.5mm}\\
    \quad {\bf SCAFCOM\tnote{$\dagger$}\;\;(Corollary~\ref{thm:scafcom})} & $\boldsymbol{\frac{1+\omega}{\epsilon}}$ & $\boldsymbol{\frac{1}{N\epsilon^2}}$   & \cmark & \cmark & \cmark
    \vspace{2mm}\\
    \midrule 
    {\sc Biased Compression}\vspace{2mm}\\
    \quad {\sc QSPARSE-SGD}~\citep{basu2019qsparse} & $\frac{1}{(1-q)^2\epsilon}$ & $\frac{1}{N\epsilon^2}$ & \xmark & \xmark & \cmark \vspace{1.5mm}\\
    \quad {\sc Local-SGD-C}~\citep{gao2021convergence} & $\frac{K}{(1-q)^2\epsilon}$ & $\frac{1}{N\epsilon^2}$  & \xmark & \xmark  & \cmark \vspace{1.5mm}\\ 
    \quad {\sc Fed-EF}~\citep{li2023analysis} & $\frac{1}{(1-q)^2\epsilon}$ & $\frac{1}{N(1-q)^2\epsilon^2}$  & \cmark & \xmark & \xmark \vspace{1.5mm}\\
    \quad {\bf SCAFCOM (Theorem~\ref{thm:scafcom})} & $\boldsymbol{\frac{1}{(1-q)\epsilon}}$ & $\boldsymbol{\frac{1}{N\epsilon^2}}$  & \cmark & \cmark & \cmark \vspace{1mm}\\
    \bottomrule
    \end{tabular}
    \begin{tablenotes}
        \footnotesize
        \item[$\natural$] The communication complexity requires homogeneous (iid) clients, though slightly better than ours.
        \item[$\dagger$] The results are obtained by transforming unbiased compressors into biased compressors through scaling.
    \end{tablenotes}
    \end{threeparttable}
    }
    }
\end{table}

% Table~\ref{tab:full_results} presents a comprehensive comparison of communication and computation complexities and associated restrictions
% of existing algorithms, as well as our newly proposed approaches.
% The results demonstrate that \scallion and \scafcom achieve state-of-the-art convergence or even surpass prior methods. Notably, \scallion and \scafcom are the {\em first} stochastic FL methods, to the best of our knowledge, that exhibit robustness to arbitrary data heterogeneity, partial participation, and local
% updates, and  also support communication compression relying solely on standard unbiased or contractive compressibility.

\section{Related Work}
\paragraph{Communication compression \& error feedback.}
Two popular approaches are commonly employed to compress communication in distributed systems: quantization and sparsification. Quantization involves mapping input vectors to a set of grid values, and the output can be either unbiased (random dithering) or biased (deterministic dithering) of the input value. Notable examples include Sign-SGD~\citep{seide20141,bernstein2018signsgd}, low-bit fixed rounding~\citep{dettmers20168}, Q-SGD~\citep{alistarh2017qsgd}, TurnGrad~\citep{wen2017terngrad}, and natural compression~\citep{horvoth2022natural}. On the other hand, sparsification operators only transmit a small subset of entries from the input vector, which can also be unbiased or biased~\citep{wangni2018gradient,stich2018sparsified}. Theoretical analyses of biased compressors often impose stringent assumptions, such as bounded gradients~\citep{karimireddy2019error,zhao2019global,beznosikov2020biased} due to the challenges incurred by biasedness. A more detailed summary of unbiased and biased compressors can be found in~\citet{huang2022lower,safaryan2022uncertainty,he2023unbiased}, among others.

Usually, in distributed training, unbiased compressors can be applied in place of the full-precision gradients to get reasonable theoretical rates and empirical performance. However, directly using biased compressors may slow down convergence or even lead to divergence~\citep{beznosikov2020biased,li2023analysis}. To alleviate the information distortion caused by compression, the technique of error feedback (EF) was first proposed in \cite{seide20141}. Error feedback has proven particularly effective in addressing biased compressors~\citep{stich2018sparsified,karimireddy2019error}, and it has inspired numerous subsequent distributed approaches~\citep[\eg,][]{wu2018error,alistarh2018convergence,li2022distributed}.
% ~\citet{li2022distributed} shows that EF is also effective for Adam-type adaptive optimizers.
Moreover, a variant scheme of error feedback called EF21 was introduced recently~\citep{richtarik2021ef21}. EF21 compresses increments of deterministic gradients and offers superior theoretical guarantees compared to vanilla error feedback. 
% Consequently, extensions of EF21 accounting for gradient stochasticity, gradient clipping, and variance reduction have emerged ~\citep{fatkhullin2021ef21,khirirat2023clip21,gruntkowska2023ef21,fatkhullin2023momentum}.

\paragraph{Federated learning with compression.}
Federated learning has gained great prominence since the introduction of {\sc FedAvg}, proposed by~\citet{mcmahan2017communication} to improve the communication efficiency of classic distributed training. Subsequent studies have explored its theoretical convergence and empirical performance, revealing its susceptibility to data heterogeneity (\ie, non-iid clients) due to the ``client-drift'' effect, particularly when not all clients participate in training~\citep{stich2019local,yu2019parallel,wang2021cooperative,lin2020don,wang2020slowmo,li2020convergence,yang2021achieving}.
Substantial efforts have been made to address client heterogeneity in FL~\citep{liang2019variance,li2020fedprox,li2020federated,wang2020tackling,zhang2021fedpd,haddadpour2021federated,yuan2022convergence,alghunaim2023local,cheng2024momentum}, and develop other FL protocols involving variance reduction techniques or adaptive optimizers~\citep{karimireddy2020scaffold,reddi2020adaptive,chen2020toward,karimi2023fedlamb}.
Notably, \scaffold introduced by~\citet{karimireddy2020scaffold} leverages control variables to mitigate the impact of data herogeneity and partial client participation.

To further reduce communication costs, communication compression has been integrated into federated learning algorithms, leading to methods such as {\sc FedPAQ}~\citep{reisizadeh2020fedpaq}, {\sc FedCOMGATE}~\citep{haddadpour2021federated}, {\sc Fed-EF}~\citep{li2023analysis}, etc. However, due to the information distortion incurred by compression, the existing communication-compressed FL methods either lack the robustness to arbitrary client heterogeneity and partial participation or rely on stringent conditions of compressors used by clients, going beyond standard unbiased/contractive compressibility. In contrast, our proposed algorithms work under minimal assumptions, which accommodate arbitrary client heterogeneity, partial participation and standard compressibilities while outperforming previous methods theoretically and empirically.

\paragraph{Federated learning with momentum.}
The utilization of momentum in optimization traces back to Nesterov's acceleration~\citep{Nesterov2004Intro}  and  the heavy-ball method~\citep{polyak1964some} in deterministic settings, which has been extended to the stochastic scenario~\citep{yan2018unified,yu2019linear,liu2020improved} and other domains~\citep{yuan2021decentlam,he2023unbiased,he2023lower,chen2023optimal}.  A recent work \citep{fatkhullin2023momentum} suggests the benefits of momentum in error feedback for distributed optimization.
In the context of federated learning, momentum has been widely incorporated and empirically shown to enhance performance 
\citep{wang2020slowmo,karimireddy2020mime,khanduri2021stem,das2022faster}. For FL, \cite{cheng2024momentum} demonstrates that momentum can mitigate the client drift phenomenon in {\sc FedAvg} under full client participation.  
It is important to note that the algorithms and analysis in this paper are different from these prior works because of the unique challenges posed by the interplay of local updates, partial client participation, and communication compression.

% \subsection{Notations}

\section{Problem Setup}
Formally, in federated learning, we aim to minimize the following objective:
\begin{equation*}
    \min_{x\in\RR^d}\quad f(x):= \frac{1}{N}\sum_{i=1}^N f_i(x)\quad \text{where}\quad  f_i(x):=\EE_{\xi_i\sim \mathcal{D}_i}[F(x;\xi_i)],
\end{equation*}
where $\xi_i$ represents a local data sample of client $i$, $F(x;\xi_i)$ represents the loss function evaluated at model $x$ and sample $\xi_i$,  and
$f_i(x)$ is the  local objective w.r.t. data distribution $\mathcal{D}_i$ at client $i$. Since finding the optimum of non-convex objectives is generally intractable, we devote to finding a stationary point of $f$.

In practice, the data distributions $\mathcal{D}_i$ across clients may vary significantly, resulting in the inequality $f_i(x) \neq f_j(x)$ for different clients $i$ and $j$. Consequently, a globally stationary model $x^\star$ with $\nabla f(x^\star)=0$ may not be a stationary point of the local objectives, leading to large values of $\|\nabla f_i(x^\star)\|$.
This phenomenon is widely referred to as {\em  data heterogeneity}. If all local clients were homogeneous, meaning that the local data samples of different clients follow a common distribution $\mathcal{D}$, we would have $f_1(x) = \cdots=f_N(x)$ and each globally stationary  model would also be  stationary for each client.

In this paper, we use $\|\cdot\|$ to denote the $\ell_2$ vector norm and use $[N]$ to denote $\{1,\dots,N\}$ for $N\in\mathbb{N}_+$. We use  the notation $\lesssim$ to denote inequalities that hold up to numeric numbers; notations $\gtrsim $ and $\asymp$ are utilized similarly. 
Before proceeding further, we first state the assumptions required for our convergence analysis throughout this paper.

\begin{assumption}[\sc Smoothness] \label{asp:smooth}
	Each local objective $f_{i}$ has $L$-Lipschitz gradient, \ie, for any $x,y\in\RR^d$ and $1\leq i\leq N$, it holds that
	$$
	\left\|\nabla f_{i}(x)-\nabla f_{i}(y)\right\| \leq L\|x-y\|.
	$$
\end{assumption}

\begin{assumption}[\sc Gradient stochasticity]\label{asp:gd-noise}
	There exists $\sigma\geq 0$ such that for any $x\in\RR^d$ and $1\leq i\leq N$,
	\begin{align*}
	\EE_{\xi_i}[\nabla F(x;\xi_i)]=\nabla f_i(x)\quad \text{ and }\quad \EE_{}[\|\nabla F(x;\xi_i)-\nabla f_i(x)\|^2]\leq \sigma^2,
	\end{align*}
 where $\xi_i\sim \cD_i$ are iid random samples for each client $i$.
\end{assumption}
Assumptions~\ref{asp:smooth}-\ref{asp:gd-noise} are standard in the analysis of FL algorithms. It is worth highlighting that, these are the {\em only} two assumptions required for all the theoretical analysis in this paper.

\section{SCALLION: Single-round Compressed Communication}

In this section, we first revisit the seminal \scaffold algorithm~\citep{karimireddy2020scaffold}, which requires communicating two variables (of the same size as the model) from client to server per communication round. We present a new formulation with only a {\em single variable for uplink communication} for each  client participating in a training round. We then propose \scallion, which employs  unbiased compressors to further reduce the communication workload of \scaffold. \scallion  embraces arbitrary data heterogeneity, local updates, and partial client participation. Theoretical analysis is provided, showing that \scallion converges at a state-of-the-art rate under standard unbiased compressibility.

 % In the non-compression regime, various methods have been proposed to enhance the convergence of \fedavg~\citep{li2020federated,zhang2021fedpd,durmus2021federated,alghunaim2023local,cheng2024momentum}. In particular, \scaffold~\citep{karimireddy2020scaffold} maintains control variables on clients to compensate for ``client drift'' in its local updates, yielding convergence robust to data heterogeneity and client sampling. However,  \scaffold incurs  a doubling of the communication cost compared to vanilla \fedavg method~\citep{mcmahan2017communication} as each client needs to additional transmit the update of the control variable when participating in training. This heavily compromises the practical utility of \scaffold. Moreover, when communication compression is to be employed, the two-round uplink communication in  \scaffold poses obstacles to achieving stable performance and challenges in analyzing convergence, particularly for biased compressors. Other efforts 

\subsection{Background of SCAFFOLD}
The \scaffold  approach~\citep{karimireddy2020scaffold} maintains local control variables $\{c_i^t\}_{i=1}^N$ on clients and a global control variable $c^t$ on the server. Let $\cS^t\subseteq [N]$ (with $|\cS^t|=S$) be the set of accessible (active) clients to interact with the server in the $t$-th round. In each training round, \scaffold conducts $K$ local updates within each accessible client $i\in\cS^t$ by
\begin{equation}\label{eqn:vncxivnxc}
    y_i^{t,k+1}:=y_i^{t,k}-\eta_l(\nabla F(y_i^{t,k};\xi_i^{t,k})-c_i^t+c^t),\quad \text{for }k=0,\dots,K-1,
\end{equation}
where $y_i^{t,k}$ is the local model in client $i$ initialized with the server model  $y_i^{t,0}:=x^t$ and $\eta_l$ is the local learning rate. Here,
the subscript $i$
represents the client index, while the superscripts $t$ and $k$ denote the outer and inner loop indexes corresponding to communication rounds and  local-update steps, respectively.
Upon the end of local training steps, clients update local control variables as\footnote{The original \scaffold~\citep{karimireddy2020scaffold} paper has a second option in updating local control variables $c_i^{t+1}=\frac{1}{K}\sum_{k=0}^{K-1}\nabla F(x^t;\tilde \xi_i^{t,k})$ by sweeping additional pass over the local
data. Since this option requires extra computation and is not widely employed in literature,
we focus on  \eqref{eqn:bvibsdvxc} throughout the paper.}:
\begin{align}\label{eqn:bvibsdvxc}
    c_{i}^{t+1}:=\begin{cases}
        c_i^t-c^t+\frac{x^t-y_i^{t,K}}{\eta_l K},&\text{if }i\in\cS^t,\\
        c_i^t,&\text{otherwise.}
    \end{cases}
\end{align}
The increments of local model $y_i^{t,K}-x^t$ and control variable $c_i^{t+1}-c_i^t$, of each participating client $i\in\cS^t$, are then sent to the central server and aggregated to
update the global model parameters:
\begin{align}\label{eqn:nvidsvzx}
    x^{t+1}:=x^t+\frac{\eta_g}{S} \sum_{i\in\cS^t}(y_i^{t,K}-x^t),\quad c^{t+1}:=c^t+\frac{1}{N}\sum_{i\in\cS^t}(c_i^{t+1}-c_i^t),
\end{align}
where $\eta_g$ is the global learning rate. The detailed description of \scaffold can be found in Appendix~\ref{app:scaffold}. Notably, the control variables of \scaffold track local gradients such that $c_i^t\approx \nabla f_i(x^t) $ and $c^t\approx \nabla f(x^t) $, thereby  mimicking the ideal  update through $\nabla F(y_i^{t,k};\xi_i^{t,k})-c_i^t+c^t\approx \nabla f(x^t)$ given $\nabla F(y_i^{t,k};\xi_i^{t,k})\approx \nabla f_i(y_i^{t,k})$  and $y_i^{t,k}\approx x^t$. Consequently, the local updates are nearly synchronized in the presence of data heterogeneity without suffering from client drift.

While the introduction of control variables enables \scaffold to converge robustly with arbitrarily heterogeneous clients and partial client participation, the original implementation of \scaffold described above requires clients  to communicate  both  updates of local models $y_i^{t,K}-x^t$ and control variables $c_i^{t+1}-c_i^t$ (also see ~\cite[Alg. 1, line 13]{karimireddy2020scaffold}). This results in a doubled client-to-server communication cost and more obstacles to employing communication compression, compared to its counterparts without control variables such as {\sc FedAvg}. 

\subsection{Development of SCALLION} Now we present an equivalent implementation of \scaffold which only requires a single variable for uplink communication and is readily employable for communication compression.
Expanding the updates of local models $y_i^{t,K}-x^t$ and control variables $c_i^{t+1}-c_i^t$ used by exploiting \eqref{eqn:vncxivnxc} and \eqref{eqn:bvibsdvxc}, we have
\begin{equation}\label{eqn:nvcix1}
    c_i^{t+1}-c_i^t=\frac{x^t-y_i^{t,K}}{\eta_l K}-c^t=\frac{1}{K}\sum_{k=0}^{K-1}\nabla F(y_i^{t,k};\xi_i^{t,k})-c_i^t\triangleq \Delta_i^t,
\end{equation}
and
\begin{align}\label{eqn:nvcix2}
    y_{i}^{t,K}-x^t=-\eta_l\sum_{k=0}^{K-1}\left(\nabla F(y_i^{t,k};\xi_i^{t,k})-c_i^t+c^t\right)=-\eta_l K (\Delta_i^t+c^t).
\end{align}
In \eqref{eqn:nvcix1} and \eqref{eqn:nvcix2}, we see that  the updates of local models and control variables share a common component, the {\em increment} variables $\Delta_i^t$. Since the global control variable $c^t$ is inherently maintained by the server,  updates $c_i^{t+1}-c_i^t$ and thus $y_{i}^{t,K}-x^t$ can be  recovered by the server upon receiving  the increment variables $\Delta_i^t$. Therefore,  the server model and control variable can be equivalently updated as 
\begin{align}\label{eqn:server-update}
    x^{t+1}=x^t-\frac{\eta_g\eta_l K }{S}\sum_{i\in\cS^t}\left(\Delta_i^t+c^t\right)\quad \text{and}\quad c^{t+1}=c^t+\frac{1 }{N}\sum_{i\in\cS^t}\Delta_i^t.
\end{align}

\paragraph{We only need to communicate $\Delta_i^t$.} 
Based on our above formulation, by communicating the increment variables $\Delta_i^t$ and applying the server-side updates \eqref{eqn:server-update} accordingly, \scaffold can be implemented equivalently with a halved uplink communication cost, compared to the original one~\citep{karimireddy2020scaffold}. A detailed description of the new implementation can be found in Algorithm~\ref{alg:scaffold-new} in Appendix~\ref{app:scaffold}.
It is also worth noting that the new implementation only modifies the communication procedure, and 
the same local updates as in~\cite{karimireddy2020scaffold} remain in our implementation.

\paragraph{Benefits of compressing $\Delta_i^t$.} 
Importantly, the new implementation of \scaffold provides a  simpler and more natural backbone for  communication compression as only the transmission of $\Delta_i^t$ is to be compressed. Moreover, unlike compressing local gradients as adopted in~\citet{reisizadeh2020fedpaq,haddadpour2021federated,basu2019qsparse,gao2021convergence,li2023analysis}, compressing $\Delta_i^t$ asymptotically eliminates compression errors even in the presence of client heterogeneity.
Consider the case of deterministic gradients  for simplicity.
Based on the update rules of \scaffold (Algorithm \ref{alg:scaffold-new}), if hypothetically the training approached a steady stage where $x^t$ is close to a  stationary point $x^\star$, we expect to have $c_i^t \approx \nabla f_i(x^\star)$ and $c^t=\frac{1}{N}\sum_{i=1}^N c_i^t\approx \nabla f(x^\star) =0$. Consequently, the directions for local updates satisfy $\nabla f_i(y^{t,k})-c_i^t+c^t\approx 0$  so that 
$ x^\star\approx x^t\approx y^{t,1}\approx\cdots \approx y^{t,K}$. Therefore, the definition of $\Delta_i^t$ in \eqref{eqn:nvcix1} implies
\[
\Delta_i^t=\frac{1}{K}\sum_{k=0}^{K-1}\nabla f_i(y_i^{t,k})-c_i^t\approx  0.
\]
% \[
% \Delta_i^t=\frac{1}{K}\sum_{k=0}^{K-1}\nabla f_i(y_i^{t,k})-c_i^t\approx \frac{1}{K}\sum_{k=0}^{K-1}\nabla f_i(x^{t})-c_i^t\,\to\, 0.
% \]
Namely, the increment variable $\Delta_i^t$ gradually vanishes as the algorithm iterates. 
Therefore, taking an $\omega$-unbiased compressor as an example (see Definition~\ref{def:unbiased}), compressing $\Delta_i^t$ results in a vanishing compression error 
\[
\EE[\|\cC_i(\Delta_i^t)-\Delta_i^t\|^2]\leq \omega \|\Delta_i^t\|^2\,\to\, 0,
\]
regardless of data heterogeneity. In contrast, if one considers compressing  local gradients directly, a constantly large  compression error is introduced in each communication round
\[
\EE[\|\cC_i(\nabla f_i(x^t))-\nabla f_i(x^t)\|^2]\leq \omega \|\nabla f_i(x^t)\|^2\,\to \,\omega \|\nabla f_i(x^\star)\|^2\neq 0.
\]
The constant $\|\nabla f_i(x^\star)\|^2$  can be extremely large when the data heterogeneity is severe, resulting in the susceptibility of algorithms to data heterogeneity. 
% See more related discussion in~\citet{richtarik2021ef21}.

% \subsubsection{}
% Since the new implementation for \scaffold enables a  more natural and superior framework  for communication compression,
Following the above argument regarding compression,
we now propose to transmit the compressed proxy $\tilde \delta_i^t\triangleq \cC_i(\alpha \Delta_i^t)$ of the increment variable $\Delta_i^t$, leading to the \scallion method as presented in Algorithm~\ref{alg:scallion}. Here $\cC_i$ is the compressor utilized by client $i$ while the scaling factor $\alpha \in[0,1]$  is introduced to stabilize the updates of control variables $\{c_i^t\}_{i=1}^N$ and can be viewed as the learning rate of  control variables. When $\alpha =1$ and $\{\cC_i\}_{i=1}^N$ are the identity mappings (\ie, no compression), \scallion will reduce to  \scaffold with our new implementation.

\begin{algorithm}[h]
	\caption{\scallion: \scaffold with single compressed uplink communication}
	\label{alg:scallion}
	\begin{algorithmic}[1]
		\STATE \noindent {\bfseries Input:} initial model $x^0$ and control variables $\{c_i^0\}_{i=1}^N$, $c^0$; local learning rate $\eta_l$; global learning rate $\eta_g$; local steps $K$; number of sampled clients $S$; scaling factor $\alpha \in[0,1]$
		\FOR{$t=0,\cdots,T-1$}
		\STATE Uniformly sample clients $\cS^t\subseteq[N]$ with $|\cS^t|=S$
            \FOR{client $i\in\cS^t$ in parallel}
            \STATE Receive $x^t$ and $c^t$; initialize $y_i^{t,0}=x^t$
                \FOR{$k=0,\dots, K-1$}
                \STATE Compute a mini-batch gradient $g_i^{t,k}=\nabla F(y_i^{t,k};\xi_i^{t,k})$
                \STATE Locally update $y_i^{t,k+1}=y_i^{t,k}-\eta_l(g_i^{t,k}-c_i^t+c^t)$
                \ENDFOR
                % \STATE{
                % {Update $v_i^{t+1}=(1-\alpha)v_i^t+\alpha \left(\frac{x^t-y_i^{t,K}}{\eta_l K}+c_i^t-c^t\right)$ (for $i\notin \cS^t$, $v_i^{t+1}=v_i^t$)}\\
                % {Compute $
                %     \delta_i^t=v_i^{t+1}-c_i^t$ and send $\tilde{\delta}_i^t=\mathcal C_i(\delta_i^t)$ to the server}
                % }
                
                \STATE \colorbox{Ocean}{\hspace{-1mm}Compute $
                    \delta_i^t=\alpha\left(\frac{x^t-y_i^{t,K}}{\eta_l K}-c^t\right)$\hspace{-1mm}}
                \STATE \colorbox{Ocean}{\hspace{-1mm}Compress and send $\tilde{\delta}_i^t=\mathcal C_i(\delta_i^t)$ to the server\hspace{-1mm}} \hfill $\triangleright\mbox{\footnotesize{ $\alpha=1$ and $\cC_i=I$ recovers \scaffold}}$  
                \STATE Update $c_i^{t+1}=
                    c_i^t+\tilde{\delta}_i^t$ (for $i\notin \cS^t$, $c_i^{t+1}=c_i^t$)
            \ENDFOR
  		\STATE Update $x^{t+1}=x^t-\frac{\eta_g \eta_l K}{S}\sum_{i\in\cS^t}(\tilde{\delta}_i^t+c^t)$
                        \STATE Update $c^{t+1}=c^t+\frac{1}{N}\sum_{i\in\cS^t}\tilde{\delta}_i^t$
		\ENDFOR 
		% \RETURN
	\end{algorithmic}
\end{algorithm}

% \paragraph{Comparison with \scaffold~\citep{karimireddy2020scaffold}.} While the \scallion method can be traced back to \scaffold, particularly in the utilization of control variables, \scallion conducts only a single round of compressed communication, as opposed to two rounds in \scaffold; see line 13 in~\cite[Alg. 1]{karimireddy2020scaffold}. Consequently, \scallion is much more communication-efficient than \scaffold.  

\paragraph{Comparison with \fedpaq~\citep{reisizadeh2020fedpaq}, {\sc FedCOM}~\citep{haddadpour2021federated},  \fedef~\citep{li2023analysis}.}  All of them boil down to the \fedavg algorithm~\citep{mcmahan2017communication} when no compression is conducted. As such, their convergence is significantly hampered  by data heterogeneity across clients due to client drift. The former two works do not consider partial participation, and \fedef suffers from an extra slow-down factor in the convergence rate under partial participation. In opposition, \scallion roots from \scaffold, and is robust to arbitrary data heterogeneity and partial participation.

% \paragraph{Comparison with \fedcom~\citep{haddadpour2021federated} and \vrlsgd~\citep{liang2019variance}.}
\paragraph{Comparison with \fedcom~\citep{haddadpour2021federated}.} 
\fedcom applies compression over the uplink communication of the \vrlsgd algorithm~\citep{liang2019variance}, a gradient-tracking-based FL method which is different from \scaffold. 
% While having the same communication overhead as  \scallion per iteration, \scaffold  is different from \fedcom in terms of the algorithmic structure. 
\fedcom
suggests conducting $K=O(1/(N\epsilon))$ local steps, demanding solving local problems to an extremely accurate resolution. Moreover, it additionally requires  uniformly bounded compression errors $\EE[\|\frac{1}{N}\sum_{i\in[N]}\cC_i(x_i)\|^2-\|\frac{1}{N}\sum_{i\in[N]}x_i\|^2]\leq G_A^2$, which are invalid for practical compressors such as random sparsification~\citep{wangni2018gradient} and random dithering~\citep{alistarh2017qsgd}. In addition, both the convergence for \fedcom and \vrlsgd is only established when all clients participate in training. It is unclear if their convergence can be adapted to partial client participation.
In contrast,  \scallion converges at a state-of-the-art rate that admits a flexible number of local steps and client sampling and employs standard  unbiased compressors (\ie, Definition~\ref{def:unbiased}).

\subsection{Convergence of SCALLION}\label{sec:scallion-conv}

To study the convergence of \scallion under communication compression, we consider  compressors satisfying the following standard {\em unbiased} compressibility.

\begin{definition}[\sc \text{$\omega$}-unbiased compressor]\label{def:unbiased}
	There exists $\omega\geq 0$ such that for any input $x\in\RR^d$ and each client-associated compressor $\mathcal C_i: \RR^d\rightarrow \RR^d$, 
	\begin{equation*}
	\EE[\mathcal C_i(x)]=x\quad\text{and}\quad \EE[\|\mathcal C_i(x)-x\|^2]\leq \omega\|x\|^2,
	\end{equation*}
	where the expectation is taken over the randomness of the compressor $\mathcal C_i$. 
 % Additionally, we assume $\{\cC_i(x_i)\}_{i=1}^N$ are mutually independent for any input $\{x_i\}_{i=1}^N$.    
\end{definition}
Examples that satisfy Definition~\ref{def:unbiased} include random sparsification and  dithering as stated below.
\begin{example}[\sc Random sparsification~\citep{wangni2018gradient}]\label{eg:rands}
For any  $s\in[d]$, the random-$s$ sparsification is defined as $\cC:x\mapsto\frac{d}{s}(\xi \odot x)$ where $\odot$ denotes the entry-wise product and $\xi\in\{0,1\}^d$ is a uniformly random binary vector with $s$ non-zero entries. This random-$s$ sparsification  is an $\omega$-unbiased compressor with $\omega = d/s-1$.
\end{example}

\begin{example}[\sc Random dithering~\citep{alistarh2017qsgd}]\label{eg:sto-quant}
    For any $b\in\mathbb{N}_{+}$, the random dithering with $b$-bits per entry is defined as $\cC:x\mapsto\|x\| \times \sgn(x)\odot \zeta(x)$ where $\{\zeta_k\}_{k=1}^d$ are independent random variables such that 
    \[
    \zeta_k(x):=\begin{cases}
        \left\lfloor {2^b |x_k|}/{\|x\|}\right\rfloor /2^b,&\text{with probability }\lceil 2^b|x_k|/\|x\|\rceil -2^b|x_k|/\|x\|,\\
        \left\lceil {2^b |x_k|}/{\|x\|}\right\rceil /2^b,&\text{otherwise,}
    \end{cases}
    \]
    where $\lfloor \cdot \rfloor$ and $\lceil \cdot \rceil$ are the floor and ceiling functions, respectively.
    This random dithering with $b$-bits per entry  is an $\omega$-unbiased compressor with $\omega = \min\{d/4^b,\sqrt{d}/2^b\}$.
\end{example}

When communication compression with  $\omega$-unbiased compressors is employed, the convergence of the proposed \scallion (Algorithm~\ref{alg:scallion}) is justified as follows.
\begin{theorem}[\sc \scallion with unbiased compression]\label{thm:scallion}
    Under Assumptions~\ref{asp:smooth} and~\ref{asp:gd-noise}, supposing clients apply mutually independent $\omega$-unbiased compressors, if we initialize $c_i^0=\nabla f_i(x^0)$ and $c^0=\nabla f(x^0)$, and 
    set learning rates $\eta_l$, $\eta_g$ and the scaling factor $\alpha$ as in \eqref{eqn:scallion-para}, 
    then \scallion converges as
    \begin{equation}\label{eqn:scallion-rate}
    \frac{1}{T}\sum_{t=0}^{T-1}\EE[\|\nabla f(x^t)\|^2]
    \lesssim \sqrt{\frac{(1+\omega)L\Delta \sigma^2}{SKT}} +\left(\frac{(1+\omega)N^2L^2\Delta^2\sigma^2}{S^3KT^2}\right)^{1/3}+\frac{(1+\omega)NL\Delta}{ST},
    \end{equation}
    where $\Delta \triangleq f(x^0)-\min f(x)$. A detailed version and the proof are in Appendix~\ref{app:scallion}.
\end{theorem}
% \begin{theorem}[\sc \scallion with unbiased compression]\label{thm:scallion}
%     Under Assumptions~\ref{asp:smooth} and~\ref{asp:gd-noise}, supposing clients apply mutually independent $\omega$-unbiased compressors, if we initialize $c_i^0=\frac{1}{B}\sum_{b=1}^B\nabla F(x^0;\xi_i^b)$ and $c^0=\frac{1}{N}\sum_{i=1}^N c_i^0$ with $\{\xi_i^b\}_{b=1}^B\overset{iid}{\sim}\cD_i$ and $B=\lceil \frac{\sigma^2}{NL\Delta}\rceil$, and 
%     % set $\eta_g \eta_lKL =\frac{27\alpha S}{N}$, $\eta_l KL\leq \sqrt{\frac{\alpha (1+\omega)}{36e^2 N(24+{131\alpha}/{S})}} $, $B=\lceil \frac{\sigma^2}{NL\Delta}\rceil$, and 
%     % \begin{equation}
%     %     \alpha =\left(4(1+\omega)+\left(\frac{(1+\omega)TS\sigma^2}{N^2KL\Delta}\right)^{1/3}+\left(\frac{(1+\omega)T\sigma^2}{NKL\Delta}\right)^{1/3}\right)^{-1},
%     % \end{equation}
%     set learning rates $\eta_l$, $\eta_g$ as well as scaling factor $\alpha$ properly, 
%     then \scallion converges as
%     \begin{equation}\label{eqn:scallion-rate}
%     \frac{1}{T}\sum_{t=0}^{T-1}\EE[\|\nabla f(x^t)\|^2]
%     \lesssim \sqrt{\frac{(1+\omega)L\Delta \sigma^2}{SKT}} +\left(\frac{(1+\omega)N^2L^2\Delta^2\sigma^2}{S^3KT^2}\right)^{1/3}+\frac{(1+\omega)NL\Delta}{ST},
%     \end{equation}
%     % where $\Delta \triangleq f(x^0)-\min f(x)$ and notation $\lesssim$ omits numeric numbers in inequalities. A detailed version and proof are in Appendix~\ref{app:scallion}.
%     where $\Delta \triangleq f(x^0)-\min f(x)$. A detailed version and the proof are in Appendix~\ref{app:scallion}.
% \end{theorem}
\begin{remark}
    The initialization of $\{c_i^0\}_{i=1}^N$ and $c^0$ 
    does not affect the convergence rate  and the asymptotic complexities. The one
    in Theorem~\ref{thm:scallion} is conducted for neatness. In practice, we can simply set $c_i^0=0$.
\end{remark}
\begin{remark}[\sc Weak assumptions of this work]\label{rmk:cond}
% \label{rmk:add-asp}
    Due to comprehensive challenges in compressed FL, to facilitate convergence analysis, most existing approaches require additional stringent conditions or assumptions that are not necessarily valid in practice, including but not restricted to:
    \begin{align}
    &\max_{1\leq i\leq N}\|\nabla f_i(x)\|\leq G, \tag*{(Bounded Gradient Norm~\citep{basu2019qsparse})}\\
    &\frac{1}{N}\sum_{i=1}^N\|\nabla f_i(x)-\nabla f(x)\|^2\leq \zeta^2,\tag*{(Bounded Gradient Dissimilarity~\citep{Jiang2018linear,li2023analysis,gao2021convergence})}\\
    % &\EE\left[\left\|\frac{1}{N}\sum_{i=1}^N\cC_i(x_i)-\frac{1}{N}\sum_{i=1}^Nx_i\right\|^2\right]\leq  q_A^2\left\|\frac{1}{N}\sum_{i=1}^Nx_i\right\|^2\text{ with }q_A\in[0,1),\tag*{(Averaged Contraction~)}\\
    &\EE\left[\left\|\frac{1}{N}\sum_{i=1}^N\cC_i(x_i)\right\|^2-\left\|\frac{1}{N}\sum_{i=1}^Nx_i\right\|^2\right]\leq G_A^2.\tag*{(Bounded  Compression Error~\citep{haddadpour2021federated})}\\
    &\EE\left[\left\|\frac{1}{N}\sum_{i=1}^N\cC_i(x_i)-\frac{1}{N}\sum_{i=1}^Nx_i\right\|^2\right]\leq q_A^2\left\|\frac{1}{N}\sum_{i=1}^Nx_i\right\|^2.\tag*{(Averaged Contraction~\citep{alistarh2018convergence,li2023analysis})}
    % &\EE[\|\cC_i(z)-z\|^2]\leq \epsilon^2.\tag*{(Bounded Compression Error~{doublesqueeze;ecd-psgd})} 
    \end{align}
As a result, their convergence rates inevitably depend on  the large constants $G$, $\zeta$, $G_A$, $(1-q_A)^{-1}$.
In contrast, the results presented in our work  do  {\bf not} rely on any such condition. 
\end{remark}

\paragraph{Asymptotic complexities of SCALLION.}
When using full-batch gradients (\ie, $\sigma\to 0$), all terms involving $\sigma$ in \eqref{eqn:scallion-rate} vanish. Consequently, the bottleneck of FL algorithms boils down to the rounds of client-to-server communication. On the other hand, when gradients are noisy (\ie, $\sigma $ is extremely large), the $\sigma/\sqrt{T}$-dependent term dominates others in which $\sigma$ are with lower orders. In this case, the performance is mainly hampered by the number of gradient evaluations.
    Following~\citep{fatkhullin2023momentum}, we refer to the total number of communication rounds in the regime $\sigma\to 0$ and gradient evaluations  in the regime $\epsilon\to 0$ required by per client to attain $\EE[\|\nabla f(\hat x)\|^2]\leq \epsilon$ as the asymptotic communication complexity and computation complexity\footnote{Essentially, the two complexities justify the convergence rates of non-convex algorithms in terms of the $\sigma/\sqrt{T}$-dependent and $1/T$-dependent terms. For example, given the rate listed in \eqref{eqn:scallion-rate}, the asymptotic communication complexity $T\asymp\frac{N(1+\omega)}{S\epsilon}$ is derived from $\frac{(1+\omega)NL\Delta}{ST} \asymp \epsilon$  while the asymptotic computation complexity $\frac{SKT}{N}\asymp \frac{1+\omega}{N\epsilon^2}$ follows from $\sqrt{\frac{(1+\omega)L\Delta \sigma^2}{SKT}} \asymp \epsilon$. All the complexities listed in Table \ref{tab:full_results} are calculated in the same manner.}. Theorem~\ref{thm:scallion} shows  $\frac{N(1+\omega)}{S\epsilon}$ asymptotic communication complexity  and $\frac{1+\omega}{S\epsilon^2}$  asymptotic computation complexity of \scallion. Here, we focus on presenting the impact of stationarity $\epsilon$, compression $\omega$, client participation $S$ and $N$, and local steps $K$ in asymptotic complexities.
    % When using full-batch gradients (\ie, $\sigma\to 0$), all the $\sigma^2$-dependent terms in \eqref{eqn:scallion-rate} vanish. Consequently, the bottleneck of FL algorithms boils down to the rounds of client-to-server communication. On the other hand, to attain a highly stationary point $\EE[\|\nabla f(\hat x)\|^2]\leq \epsilon$ with $\epsilon\to 0$, the $\sigma/\sqrt{T}$-dependent term dominates others and thus  the performance is mainly hampered by the number of gradient evaluations.
    % Following~\citep{fatkhullin2023momentum}, we refer to the total number of communication rounds in the regime $\sigma\to 0$ and gradient evaluations  in the regime $\epsilon\to 0$ required by per client to attain $\EE[\|\nabla f(\hat x)\|^2]\leq \epsilon$ as the asymptotic communication complexity and computation complexity\footnote{When dasdasdsa}. In consequence, Theorem~\ref{thm:scallion} shows   $\mathcal{O}(\frac{N(1+\omega)}{S\epsilon})$ asymptotic communication complexity  and $\mathcal{O}(\frac{1+\omega}{S\epsilon^2})$  asymptotic computation complexity of \scallion. Here, we focus on presenting the impact of precision $\epsilon$, compression $\omega$, client participation $S$ and $N$, and local steps $K$ in asymptotic complexities with notation $\cO(\cdot)$.

\paragraph{Comparison with prior compressed FL methods.}
Table~\ref{tab:full_results}
provides a summary of 
non-convex FL methods employing unbiased compressors under full client participation.  
We observe that \scallion matches the state-of-the-art asymptotic communication and computation complexities under non-iid clients. 
In particular, while having the same asymptotic complexities as \fedcom~\citep{haddadpour2021federated}, \scallion does not incur the dependence on a large uniform bound of compression errors (see Remark \ref{rmk:cond}) in convergence and thus has  a superior  convergence rate.
% In particular, \scallion outperforms  \fedef~\citep{li2023analysis} by a factor of $1+\omega $ in both  complexities and does not suffer from multiple local updates, as opposed to \fedpaq~\citep{reisizadeh2020fedpaq}. In comparison to \fedcom~\citep{haddadpour2021federated}, while having the same asymptotic complexities, \scallion does not incur the dependence on a large uniform bound of compression errors (see the discussion below) in convergence, reaching a superior  convergence rate.

To sum up, based on the above discussion,  we demonstrate that \scallion theoretically improves existing FL methods with unbiased compression. In particular, 
\scallion is the {\em first} stochastic FL method, to the best of our knowledge, that accommodates arbitrary data heterogeneity, partial client participation, and local
updates, without any additional assumptions on compression errors.

\section{SCAFCOM: Biased Compression with Momentum}\label{sec:scafcom}

While \scallion achieves superior convergence speed under unbiased compression, its analysis cannot be adapted to {\em biased} compressors (also known as {\em contractive} compressors) to attain fast convergence rates. In this section, we propose an algorithm called \scafcom as a complement of \scallion to accommodate biased communication compression in FL.

\subsection{Development of SCAFCOM}

In the literature, biased compressors are commonly modeled by the following contractive compressibility.

\begin{definition}[\sc \text{$q^2$}-contractive compressor]\label{def:contract}
	There exists $q\in[0,1)$ such that for any input $x\in\RR^d$ and each client-associated compressor $\mathcal C_i: \RR^d\rightarrow \RR^d$, 
	\begin{equation*}
	\EE[\|\mathcal C_i(x)-x\|^2]\leq q^2\|x\|^2,
	\end{equation*}
	where the expectation is taken over the randomness of the compressor $\mathcal C_i$. 
 % Additionally, we assume $\{\cC_i(x_i)\}_{i=1}^N$ are mutually independent for any input $\{x_i\}_{i=1}^N$.    
\end{definition}
Notably, compared to unbiased compressors satisfying Definition~\ref{def:unbiased}, contractive compressors, though potentially having smaller squared compression errors,  {\em no longer enjoy the unbiasedness}. Common examples of contractive compressors include~\citep{li2023analysis}:
\begin{example}[\sc Top-$r$ operator]\label{eg:topr}
For any  $r\in[0,1]$, 
the Top-$r$ operator is defined as $\cC:x\mapsto (\one\{k\in\cS_r(x)\} x_k)_{k=1}^d$ where $\cS_r(x)$ is the set of the largest $ r\times d$ entries of $x$ in absolute values. Top-$r$ operator  is a $q^2$-contractive compressor with $q^2 = 1-r$.
\end{example}

\begin{example}[\sc Grouped sign]\label{eg:group-sign}
Given a partition of $[d]$ with $M$ groups (\eg, layers of neural networks) $\{\cI_m\}_{m=1}^M$, the grouped sign with partition $\{\cI_m\}_{m=1}^M$ is defined as $\cC:x\mapsto\sum_{m=1}^M\|x_{\cI_m}\|_1\odot \sgn(x)_{\cI_m}/|\cI_m|$. This grouped sign operator is a $q^2$-contractive compressor with $q^2 = 1-1/\max_{1\leq m\leq M}|\cI_m| $.
\end{example}

Due to the lack of unbiasedness, compared to their counterparts with unbiased compressors, approaches employing with biased compressors in the literature typically (i) require stringent assumptions, \eg, bounded gradients~\citep{seide20141,koloskova2019decentralized,basu2019qsparse,li2022distributed} or bounded gradient dissimilarity~\citep{huang2022lower,li2023analysis}, (ii) rely on impractical algorithmic structure, \eg, large data batches in gradient computation~\citep{huang2022lower}, (iii)  have weak convergence guarantees, \eg, no improvement in the scaling of the number of clients (\ie, linear speedup)~\citep{fatkhullin2021ef21} or worse dependence on compression parameter $q$~\citep{zhao2022beer}. 

Recently, \citet{fatkhullin2023momentum} shows that tactfully incorporating momentum into communication compression can effectively mitigate the influence of biased compression. Inspired by their findings, we introduce an extra momentum variable $v_i^t$ on each client $i$ to overcome the adverse effect of biased compression. This leads to the \scafcom method, as presented in Algorithm~\ref{alg:scafcom}. When client $i$ participates in the $t$-th round, an additional  momentum variable $v_i^t$ is updated as
\begin{align}
    \colorbox{BPink}{\hspace{-1mm}$v_i^{t+1}$\hspace{-1mm}}:=(1-\beta)\colorbox{BPink}{\hspace{-1mm}$v_i^t$\hspace{-1mm}}+\beta \left(\frac{x^t-y_i^{t,K}}{\eta_l K}+c_i^t-c^t\right)=(1-\beta)\colorbox{BPink}{\hspace{-1mm}$v_i^t$\hspace{-1mm}}+ \frac{\beta}{K}\sum_{k=0}^{K-1}\nabla F(y_i^{t,k};\xi_i^{t,k}),
\end{align}
where $\{y_i^{t,k}\}$ are the intermediate local models and $\beta$ is the momentum factor. We then set $v_i^{t+1}-c_i^t=(1-\beta)v_i^t+ \beta K^{-1}\sum_{k=0}^{K-1}\nabla F(y_i^{t,k};\xi_i^{t,k}) -c_i^t$ as the  message to be communicated, as opposed to $K^{-1}\sum_{k=0}^{K-1}\nabla F(y_i^{t,k};\xi_i^{t,k}) -c_i^t$ in \scaffold and $\alpha (K^{-1}\sum_{k=0}^{K-1}\nabla F(y_i^{t,k};\xi_i^{t,k}) -c_i^t)$ in \scallion.
Compared to  the gradient $K^{-1}\sum_{k=0}^{K-1}\nabla F(y_i^{t,k};\xi_i^{t,k}) $ yielded by a single local loop, the momentum variable $v_i^{t+1}$ has  a smaller variance due to its accumulation nature,
thereby refining the convergence behavior under biased compression.
Finally, note that similar to \scallion, \scafcom only transmits one compressed variable in the uplink communication, and recovers \scaffold when $\beta =1$ and $\{\cC_i\}_{i=1}^N$ are the identity mapping (\ie, no compression).

\begin{algorithm}[t]
	\caption{\scafcom: \scaffold with momentum-enhanced compression}
	\label{alg:scafcom}
	\begin{algorithmic}[1]
		\STATE \noindent {\bfseries Input:} initial model $x^0$ and control variables $\{c_i^0\}_{i=1}^N$, $c^0$; local learning rate $\eta_l$; global learning rate $\eta_g$; local steps $K$; number of sampled clients $S$; momentum $\beta\in[0,1]$
		\FOR{$t=0,\cdots,T-1$}
		\STATE Uniformly sample clients $\cS^t\subseteq[N]$ with $|\cS^t|=S$
            \FOR{client $i\in\cS^t$ in parallel}
            \STATE Receive $x^t$ and $c^t$; initialize $y_i^{t,0}=x^t$
                \FOR{$k=0,\dots, K-1$}
                \STATE Compute a mini-batch gradient $g_i^{t,k}=\nabla F(y_i^{t,k};\xi_i^{t,k})$
                \STATE Locally update $y_i^{t,k+1}=y_i^{t,k}-\eta_l(g_i^{t,k}-c_i^t+c^t)$
                \ENDFOR
                \STATE{
                \colorbox{BPink}{\hspace{-1mm}Update $v_i^{t+1}=(1-\beta)v_i^t+\beta \left(\frac{x^t-y_i^{t,K}}{\eta_l K}+c_i^t-c^t\right)$ (for $i\notin \cS^t$, $v_i^{t+1}=v_i^t$)\hspace{-1mm}}
                
                \STATE \colorbox{BPink}{\hspace{-1mm}Compute $
                    \delta_i^t=v_i^{t+1}-c_i^t$\hspace{-1mm}}
                }
                
                \STATE \colorbox{BPink}{\hspace{-1mm}Compress and send $\tilde{\delta}_i^t=\mathcal C_i(\delta_i^t)$ to the server\hspace{-1mm}}\hfill $\triangleright\mbox{\footnotesize{ $\beta=1$ and $\cC_i=I$ recovers \scaffold}}$  
                % \STATE{{
                %  Compute $
                %     \delta_i^t=\frac{x^t-y_i^{t,K}}{\eta_l K}-c^t$ and send $\tilde{\delta}_i^t=\mathcal C_i(\delta_i^t)$ to the server}}
                \STATE Update $c_i^{t+1}=
                    c_i^t+\tilde{\delta}_i^t$ (for $i\notin \cS^t$, $c_i^{t+1}=c_i^t$)
            \ENDFOR
  		\STATE Update $x^{t+1}=x^t-\frac{\eta_g \eta_l K}{S}\sum_{i\in\cS^t}(\tilde{\delta}_i^t+c^t)$
                        \STATE Update $c^{t+1}=c^t+\frac{1}{N}\sum_{i\in\cS^t}\tilde{\delta}_i^t$
		\ENDFOR 
		% \RETURN 
	\end{algorithmic}
\end{algorithm}

\paragraph{Connection with SCALLION. }
Notably, the  difference between \scafcom and \scallion lies in the utilization of momentum; see the colored highlights in Algorithm~\ref{alg:scallion} and ~\ref{alg:scafcom}. Specifically, if we replace line 10 of \scafcom with the following formula:
\begin{align}
     \colorbox{BPink}{\hspace{-1mm}$v_i^{t+1}$\hspace{-1mm}}:=(1-\alpha)\colorbox{Ocean}{\hspace{-1mm}$c_i^t$\hspace{-1mm}}+\alpha \left(\frac{x^t-y_i^{t,K}}{\eta_l K}+c_i^t-c^t\right)=(1-\alpha)\colorbox{Ocean}{\hspace{-1mm}$c_i^t$\hspace{-1mm}}+ \frac{\alpha}{K}\sum_{k=0}^{K-1}\nabla F(y_i^{t,k};\xi_i^{t,k}),
\end{align}
then \scafcom recovers \scallion (Algorithm~\ref{alg:scallion}) with $\beta=\alpha$. Note that in this case, the memorization of $v_i^t$ is {\em no longer needed to be retained}, which is consistent with the design of \scallion. We also remark that the roles of the scaling factor $\alpha$  and momentum $\beta$ vary in \scallion and \scafcom. In \scallion, $\alpha$ stabilizes the updates of control variables $\{c_i^t\}_{i=1}^N$ while \scafcom sets $\beta$ to  mainly address the biasedness issue of contractive compressors.

\paragraph{Connection with error feedback.}
While \scafcom does not directly pertain to the vanilla  error feedback~\citep{seide20141,stich2019local}, a technique widely used to tackle biased compression, \scafcom relates to the newly proposed EF21 mechanism~\citep{richtarik2021ef21}. If one sets $\beta=1$ in \scafcom, then the message $K^{-1}\sum_{k=0}^{K-1}\nabla F(y_i^{t,k};\xi_i^{t,k}) -c_i^t$ would be compressed and the control variable would be updated as $c_i^{t+1}=c_i^t+\cC_i(K^{-1}\sum_{k=0}^{K-1}\nabla F(y_i^{t,k};\xi_i^{t,k}) -c_i^t)$. Under the simplification  where $\sigma=0$ (\ie, full-batch gradients), $K=1$ (\ie, no local updates),  $S=N$ (\ie, full client participation), it becomes $c_i^{t+1}=c_i^t+\cC_i(\nabla f_i(x^t) -c_i^t)$ and the global model is updated through $x^{t+1}=x^t-{\eta_g \eta_l }c^{t+1}$ with $c^{t+1}=\frac{1}{N}\sum_{i=1}^N c_i^{t+1}$, recovering the recursion of EF21.

\subsection{Convergence of SCAFCOM}

With the help of local momentum, the convergence of
\scafcom under $q^2$-contractive compression can be established as follows.

\begin{theorem}[\sc \scafcom with biased compression]\label{thm:scafcom}
    Under Assumption~\ref{asp:smooth} and~\ref{asp:gd-noise}, supposing clients apply $q^2$-contractive compressors $\{\cC_i\}_{i=1}^N$, if we initialize $c_i^0=v_i^0=\nabla f_i(x^0)$ and $c^0=\nabla f(x^0)$, and set learning rates $\eta_l$, $\eta_g$, and momentum $\beta$ as in \eqref{eqn:scafcom-para}, then \scafcom converges as
    \begin{equation}
    \frac{1}{T}\sum_{t=0}^{T-1}\EE[\|\nabla f(x^t)\|^2]
    \lesssim \sqrt{\frac{L\Delta \sigma^2}{SKT}} +\left(\frac{N^2L^2\Delta^2\sigma^2}{(1-q)S^2KT^2}\right)^{1/3}+\left(\frac{N^3L^3\Delta^3\sigma^2}{(1-q)^2S^3KT^3}\right)^{1/4}+\frac{NL\Delta}{(1-q)ST},
    \end{equation}
    where $\Delta \triangleq f(x^0)-\min f(x)$. A detailed version and the
proof are in Appendix~\ref{app:scafcom}.
\end{theorem}

\begin{remark}
    The initialization of $\{c_i^0\}_{i=1}^N$, $c^0$, $\{v_i^0\}_{i=1}^N$ 
     does not affect the convergence rate  and the asymptotic complexities. The one
    in Theorem~\ref{thm:scafcom} is conducted for neatness. In practice, we can simply set $c_i^0=v_i^0=0$.
\end{remark}

Furthermore, it is known that one can convert any $\omega$-unbiased compressor $\cC_i$ into a $q^2$-contractive compressors with $q^2=\frac{\omega}{1+\omega}$ through scaling $\frac{1}{1+\omega}\cC_i: x\mapsto \frac{1}{1+\omega}\cC_i(x)$ (see, \eg, \citep[Lemma 1]{safaryan2022uncertainty} and~\citep[Lemma 1]{huang2022lower}). Consequently, \scafcom can also employ unbiased compressors after the scaling with convergence guaranteed as:
\begin{corollary}[\sc \scafcom with unbiased compression]\label{cor:scafcom}
    When employing unbiased compressors (after scaling) in communication compression, then \scafcom converges as
        \begin{equation}\label{eqn:scafcom-u}
    \frac{1}{T}\sum_{t=0}^{T-1}\EE[\|\nabla f(x^t)\|^2]
    \lesssim \sqrt{\frac{L\Delta \sigma^2}{SKT}} +\left(\frac{(1+\omega)N^2L^2\Delta^2\sigma^2}{S^2KT^2}\right)^{1/3}+\left(\frac{(1+\omega)^2N^3L^3\Delta^3\sigma^2}{S^3KT^3}\right)^{1/4}+\frac{(1+\omega)NL\Delta}{ST}.
    \end{equation}
\end{corollary}
\begin{remark}
    Corollary~\ref{cor:scafcom} is obtained by directly plugging in the relation $q^2=\frac{\omega}{1+\omega}$ into Theorem~\ref{thm:scafcom} without exploiting the unbiasedness property of the compressors. However,  it is feasible to  refine the $1/T^{2/3}$ and $1/T^{3/4}$ terms in \eqref{eqn:scafcom-u} by taking advantage of unbiasedness. We omit the proof here for conciseness.
\end{remark}

\paragraph{Asymptotic complexities of SCAFCOM.}
Following the result of Theorem~\ref{thm:scafcom},
\scafcom with biased compression has an asymptotic communication complexity of  $\frac{N}{S(1-q)\epsilon}$ and an asymptotic computation complexity of $\frac{1}{S\epsilon^2}$ to attain $\EE[\|\nabla f(\hat x)\|^2]\leq \epsilon$. On the other hand, when adopting unbiased compressors with scaling, Corollary~\ref{cor:scafcom} reveals that \scafcom has an asymptotic communication complexity of  $\frac{N(1+\omega)}{S\epsilon}$ and an asymptotic computation complexity of $\frac{1}{S\epsilon^2}$.
% which enhances the result $\cO(\frac{1+\omega}{S\epsilon^2})$ of \scallion.

\paragraph{Comparison with prior compressed FL methods.}
In Table~\ref{tab:full_results}, we compare \scafcom with existing FL algorithms with biased compression under full client participation. 
We observe that \scafcom outperforms prior results with biased compression ({\sc QSPARSE-SGD}~\citep{basu2019qsparse}, {\sc Local-SGD-C}~\citep{gao2021convergence}, and {\sc Fed-EF}~\citep{li2023analysis}, etc) in the asymptotic communication  complexity by at least a factor $1/(1-q)$.
Moreover, inferior to \scafcom, the existing FL methods with biased compression  cannot tolerate unbounded data heterogeneity or even require homogeneous data. In addition, {\sc QSPARSE-SGD} and  {\sc Local-SGD-C} only converge under full client participation. Notably, when employing unbiased compression,   \scafcom  enhances the asymptotic computation complexity by a factor of $1+\omega$ compared to \scallion, surpassing all prior FL methods with unbiased compression. Furthermore, under partial client participation, our rate is better than that of \fedef~\citep{li2023analysis} by a factor of $\sqrt{\frac{N}{S}}$ thanks to control variables, overcoming the drawback of the standard error feedback under partial participation in distributed/federated learning.

Based on discussions in Section~\ref{sec:scafcom}, we demonstrate that 
\scafcom, as a unified approach, outperforms existing compressed FL methods under both unbiased and biased compression. In particular, \scafcom
is the {\em first} stochastic FL method, to the best of our knowledge, that accommodates arbitrary client heterogeneity, partial client participation, and local
updates, as well as support communication compression relying only on 
standard contractive compressibility.

\section{Experiments}

We present a set of experiments on FL benchmark datasets to demonstrate the efficacy of our proposed algorithms. Since the (substantial) saving in communication overhead of various compressors is straightforward and has been well demonstrated in prior compressed FL works (\eg, via communication vs. test accuracy plots in~\citet{haddadpour2021federated,li2023analysis}), in this section, our empirical results mainly focus on:
\begin{enumerate}
    \item Validating that \scallion and \scafcom can empirically match the full-precision \scaffold with considerably reduced communication costs.

    \item Showing the advantages of \scallion and \scafcom over prior methods with the same communication budget and training rounds.
\end{enumerate}

\subsection{Datasets, Algorithms and Training Setup}

\paragraph{Datasets and model.} We test our algorithms on two standard FL datasets: MNIST dataset~\citep{lecun1998mnist} and Fashion MNIST dataset~\citep{xiao2017fashion}. The MNIST dataset contains 60,000 training images and 10,000 test images. Each image is a gray-scale handwritten digit from 0 to 9 (10 classes in total) with 784 pixels. The FMNIST dataset has the same training and test dataset sizes and the number of pixels per image whereas each image falls into 10  categories of fashion products (\eg, bag, dress), making the learning task more challenging. Following~\citep{karimireddy2020scaffold}, we train a (non-convex) fully-connected neural network with 2 hidden layers with 256 and 128 neurons, respectively. We use ReLU as the activation function and the cross-entropy loss as the training objective.

\paragraph{Algorithms.} We implement our two proposed methods and two recent compressed FL algorithms, with biased and unbiased compression, respectively:

\begin{itemize}
    \item (Biased) {\sc Fed-EF}~\citep{li2023analysis}: Federated learning with biased compression and standard error feedback. Since our proposed algorithms conduct SGD-type updates in the server, we compare them with its {\sc Fed-EF-SGD} variant.

    \item (Biased) \scafcom (our Algorithm~\ref{alg:scafcom}): Biased compression for FL with stochastic controlled averaging and local momentum. The momentum $\beta$ in Algorithm~\ref{alg:scafcom} is tuned over a fine grid  on $[0.05, 1]$. 

    \item (Unbiased) \fedcom\citep{haddadpour2021federated}: Federated learning with unbiased compression. This algorithm uses the gradient-tracking technique to alleviate data heterogeneity.

    \item (Unbiased) \scallion (our Algorithm~\ref{alg:scallion}): Unbiased compression for FL with stochastic controlled averaging. The local scaling factor $\alpha$ in Algorithm~\ref{alg:scallion} is tuned over a fine grid on $[0.05, 1]$.
\end{itemize}
Besides the compressed FL algorithms, we also test the corresponding full-precision baselines: {\sc Fed-SGD} 
(also known as \fedavg~\citep{yang2021achieving}) and \scaffold~\citep{karimireddy2020scaffold}, both with two-sided (global and local) learning rates.
For a fair comparison, we execute \scaffold with our new implementation in experiments, corresponding to the special cases of  \scafcom  ($\cC_i=I$, $\beta= 1$) and of  \scallion ($\cC_i=I$, $\alpha= 1$). Notably, under a fixed random seed, our implementation yields the same training trajectory as~\citet{karimireddy2020scaffold}  at a halved uplink communication cost (by only sending one variable per participating client).
% Specifically, for \scaffold, beyond the original implementation in~\citet{karimireddy2020scaffold}, we additionally introduce a local momentum term similar to $v_i^t$ in Algorithm~\ref{alg:scafcom}, and run the same set of $\beta$ values as for \scafcom. In other words, the SCAFFOLD baseline in our experiments is \scafcom (Algorithm~\ref{alg:scafcom}) without compression, \ie, $\mathcal C(x)\equiv x$. This is mainly for the consistency and fairness of comparison between SCAFFOLD and \scafcom. Note that $\beta=1$ yields the exactly same SCAFFOLD algorithm in~\citet{karimireddy2020scaffold}.

In the experiments, biased compression is simulated with {\sc Top}-$r$ operators (our Example \ref{eg:topr}). Specifically, we experiment with {\sc Top}-0.01 and {\sc Top}-0.05, where only the largest $1\%$ and $5\%$ entries in absolute values are transmitted in communication. For unbiased compression, we  utilize random dithering (our Example \ref{eg:sto-quant}), with 2 bits and 4 bits per entry, respectively. We tune the combination of the global learning rate $\eta_g$ and the local learning rate $\eta_l$ over the 2D grid $\{0.001,0.003,0.01,0.03,0.1,0.3,1,3,10\}^2$. The  combination of learning rates with the highest test accuracy is reported for each algorithm and hyper-parameter choice (\eg, $\beta$, $\alpha$, and degree of compression).

\paragraph{Federated learning setting.} In our experiments, the training data are distributed across $N=200$ clients, in a highly heterogeneous setting following~\citep{li2023analysis}. The training data samples are split into 400 shards each containing samples from only one class. Then, each client is randomly assigned two shards of data. Therefore, every client only possesses training samples from at most two classes. All the clients share the same initial model at $T=0$.  In each round of client-server interaction, we uniformly randomly pick $S=20$ clients to participate in FL training, \ie, the partial participation rate is $10\%$. Each participating client performs $K=10$ local training steps using the local data, with a mini-batch size 32. All the presented results are averaged over 5 independent runs with  the same model initialization for all the algorithms. 

\begin{figure}[t]
    \begin{center}

        \mbox 
        {\hspace{-0.1in}
        \includegraphics[width=2.7in]{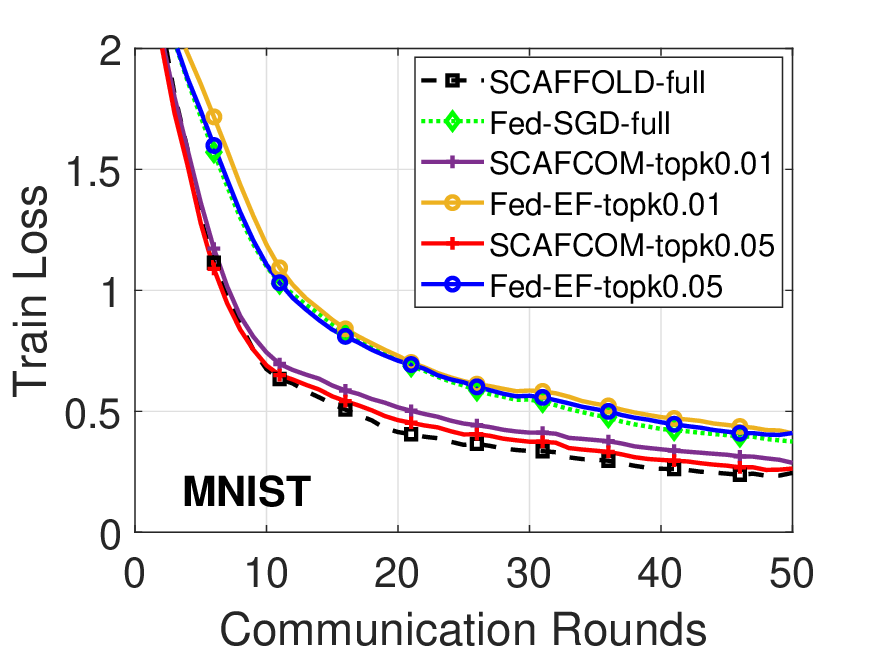}
        \includegraphics[width=2.7in]{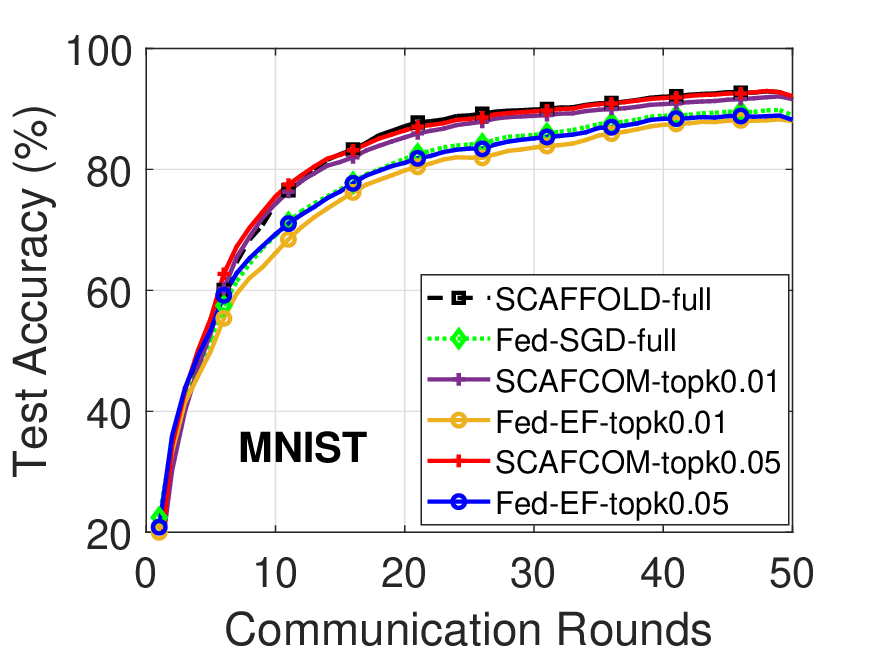}
        }

        \mbox 
        {\hspace{-0.1in}
        \includegraphics[width=2.7in]{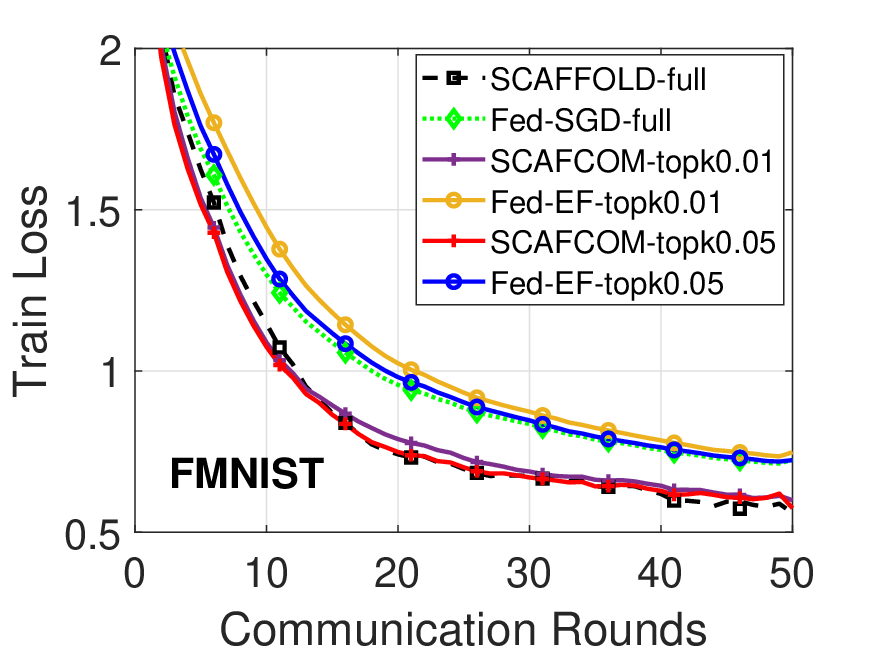}
        \includegraphics[width=2.7in]{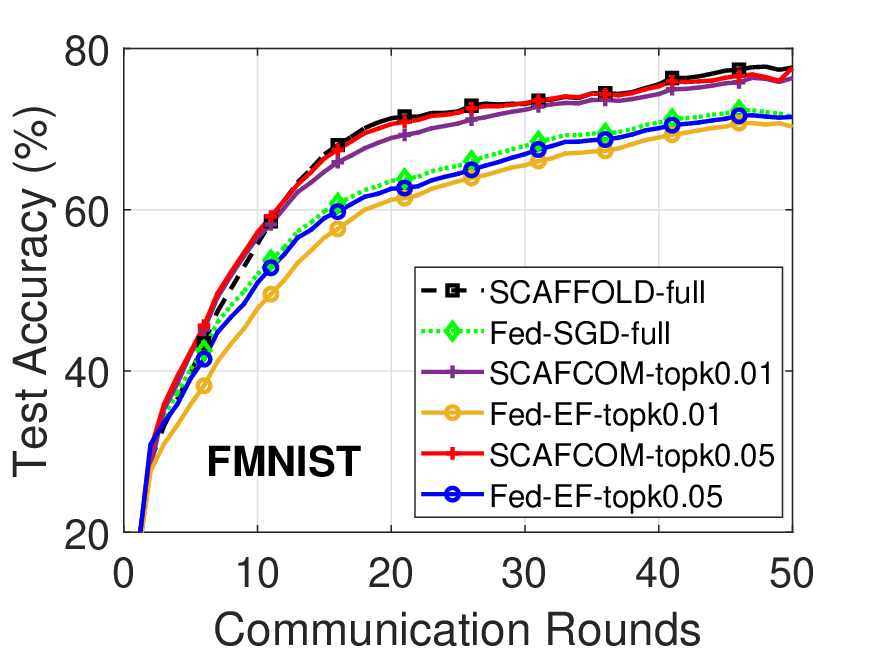}
        }
        
    \end{center}
    
    \vspace{-0.2in}
    
    \caption{Train loss and test accuracy of \scafcom (Algorithm~\ref{alg:scafcom}) and \fedef~\citep{li2023analysis} with biased {\sc Top}-$r$ compressors on MNIST (top row) and FMNIST (bottom row). 
    % For \scafcom, the curve is for $\beta=0.2$.
    }	\label{fig:biased}
\end{figure}

\subsection{Results}

Since all the compressed FL methods in our experiments require transmitting one variable in the uplink communication, their communication costs are essentially the same when the same compressor is applied. Therefore, for clarity of comparisons, we will plot the metrics versus the number of training rounds.

\paragraph{SCAFCOM with biased compression.} In Figure~\ref{fig:biased}, we first present the train loss and test accuracy of our proposed \scafcom (Algorithm~\ref{alg:scafcom}) with momentum $\beta=0.2$ and {\sc Fed-EF}~\citep{li2023analysis}, both using biased {\sc Top}-$r$ compressors. We observe:
\begin{itemize}
    \item In general, under the same degree of compression (\ie, the value of $r$ in the case), \scafcom outperforms \fedef in terms of both training loss and test accuracy, thanks to controlled variables and the local momentum  in \scafcom.
    
    \item On both datasets, \scafcom with {\sc Top}-0.01 can achieve very close test accuracy  as the full-precision \scaffold, and \scafcom with {\sc Top}-0.05 essentially match those of full-precision \scaffold. Hence, we can reach the same performance while saving 20 - 100x uplink communication costs.
    
    \item For both \scafcom and \fedef, as the degree of compression decreases (\ie, $r$ increases), their performance approaches that of the corresponding FL methods under full-precision communication (\ie, \scaffold and {\sc Fed-SGD}). 

\end{itemize}

\paragraph{SCALLION with unbiased compression.} In Figure~\ref{fig:unbiased}, we plot the same set of experimental results and compare \scallion ($\alpha=0.1$) with FedCOMGATE~\citep{haddadpour2021federated}, both applying unbiased random dithering~\citep{alistarh2017qsgd} with $2$ and $4$ bits per entry. Similarly, we see that \scallion outperforms FedCOMGATE under the same degree of compression (number of bits per entry). The \scallion curves of both 2-bit and 4-bit compression basically overlap that of SCAFFOLD, and 4-bit compression slightly performs better than 2-bit compression in later training rounds. Since random dithering also introduces sparsity in compressed variables, the 4-bit compressor already provides around 100x communication compression, and the 2-bit compressor saves more communication costs.

\begin{figure}[t]
    \begin{center}

        \mbox 
        {\hspace{-0.1in}
        \includegraphics[width=2.7in]{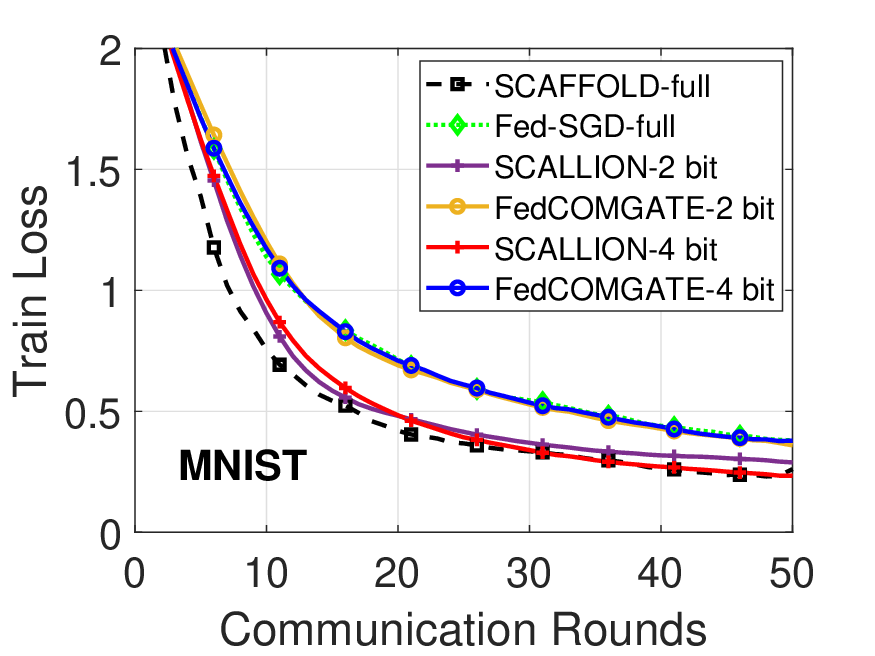}
        \includegraphics[width=2.7in]{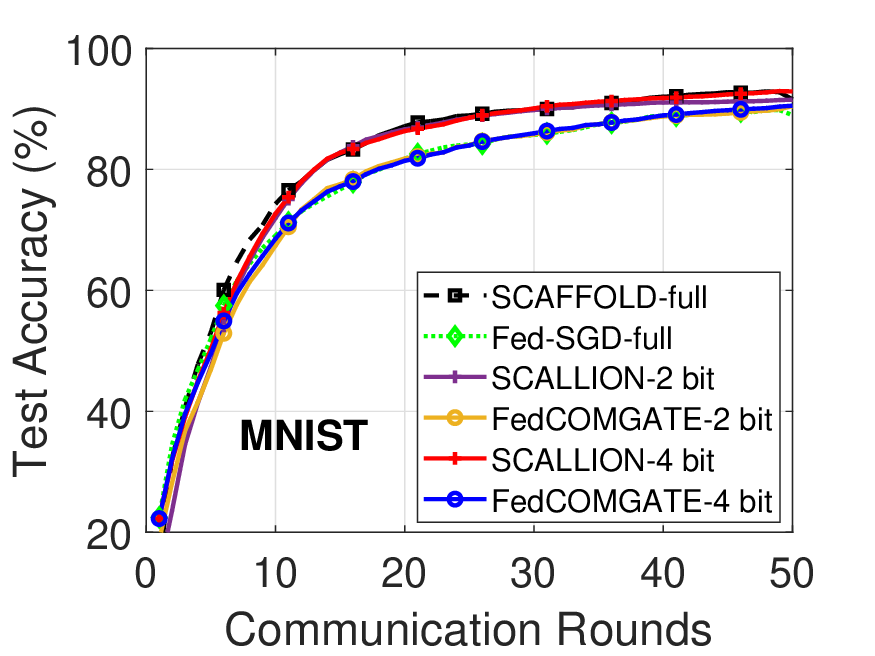}
        }

        \mbox 
        {\hspace{-0.1in}
        \includegraphics[width=2.7in]{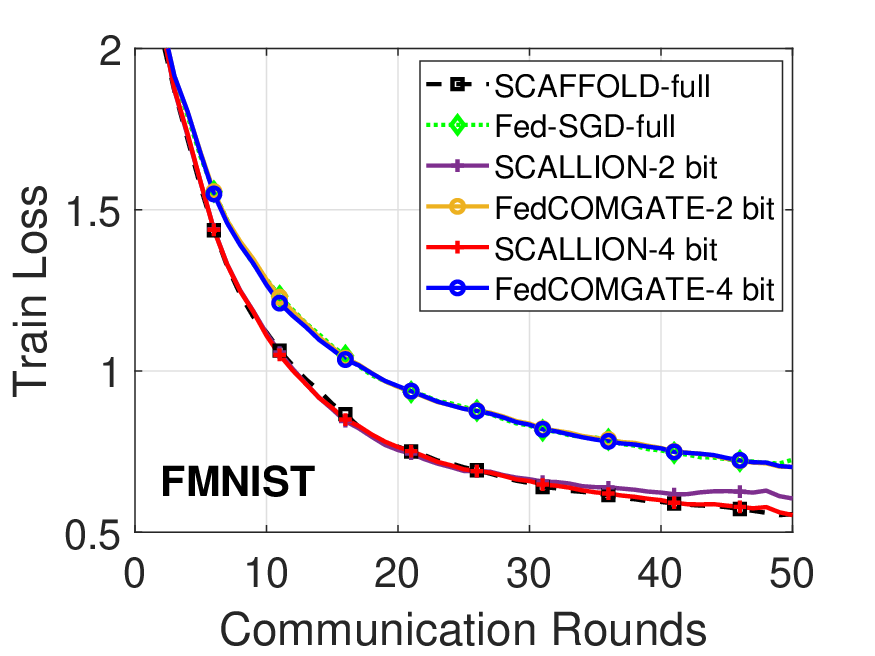}
        \includegraphics[width=2.7in]{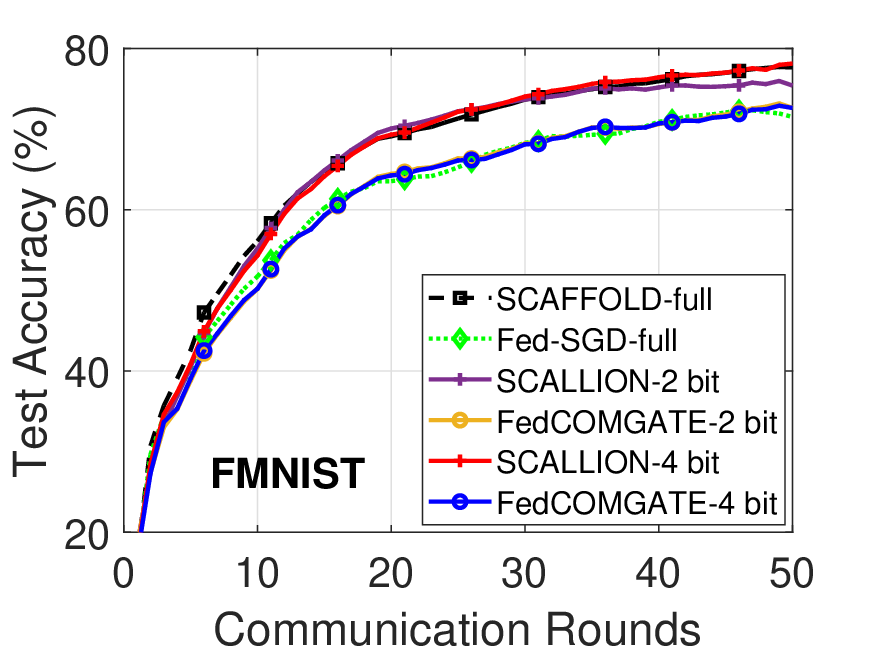}
        }
        
    \end{center}
    
    \vspace{-0.15in}
    
    \caption{Train loss and test accuracy of \scallion (Algorithm~\ref{alg:scallion}) and {\sc FedCOMGATE}~\citep{haddadpour2021federated} with unbiased random dithering on MNIST (top row) and FMNIST (bottom row).
    % For \scallion, the curve is for $\alpha=0.1$. Unbiased stochastic quantization.
    }	\label{fig:unbiased}
\end{figure}

\paragraph{Impact of $\beta$ and $\alpha$.} The momentum factor $\beta$ in \scafcom and the scaling factor $\alpha$ in \scallion are two important tuning parameters of our proposed methods. As an example, in Figure~\ref{fig:ablation}, we report the test accuracy of \scafcom with {\sc Top}-0.01 (left column) and \scallion (right column) 2-bit random dithering, for various $\beta$ and $\alpha$ values, respectively. From the results, we see that \scafcom can converge with a wide range of $\beta\in [0.05, 1]$, and $\beta=0.2$ performs the best on both datasets (so we presented the results with $\beta=0.2$ in Figure~\ref{fig:biased}). For \scallion, we report three $\alpha$-values, $\alpha=0.05, 0.1, 0.2$. When $\alpha>0.5$, the training of \scallion becomes unstable for 2-bit quantization. As we use more bits, larger $\alpha$ could be allowed. This is because, random dithering may hugely scale up the transmitted (compressed) entries, especially for low-bit quantization. When the scaling factor $\alpha$ is too large in this case, the updates of local control variables become unstable, which further incapacitates the proper dynamic the local/global training. Thus, for \scallion with low-bit random dithering, we typically need a relatively small $\alpha$. As presented in Figure~\ref{fig:unbiased}, $\alpha=0.1$ yields the best overall performance. In general, we should tune parameter $\beta$ and $\alpha$ in \scafcom and \scallion practically to reach the best performance.

\begin{figure}[h]
    \begin{center}

        \mbox 
        {\hspace{-0.1in}
        \includegraphics[width=2.7in]{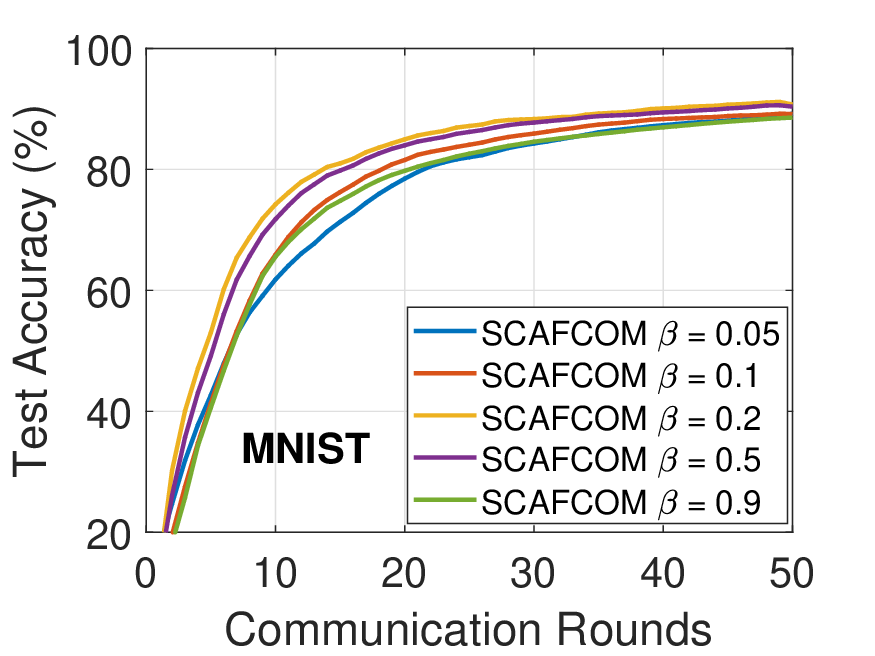}
        \includegraphics[width=2.7in]{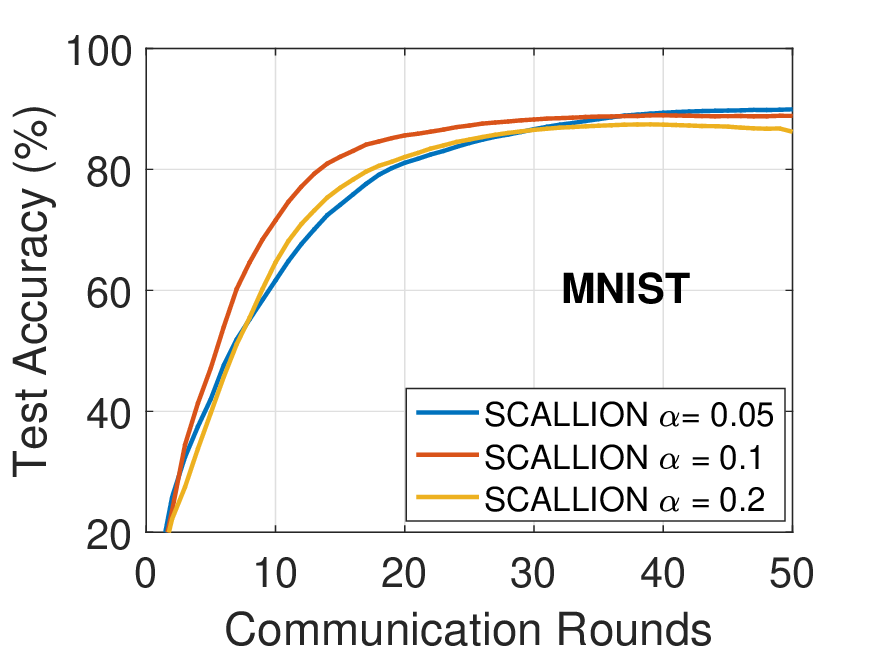}
        }

        \mbox 
        {\hspace{-0.1in}
        \includegraphics[width=2.7in]{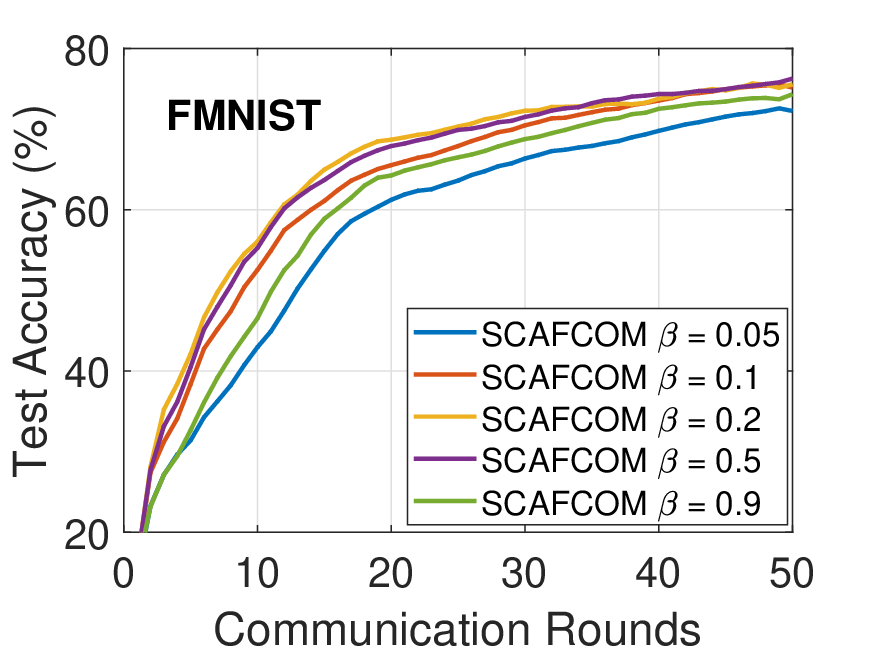}
        \includegraphics[width=2.7in]{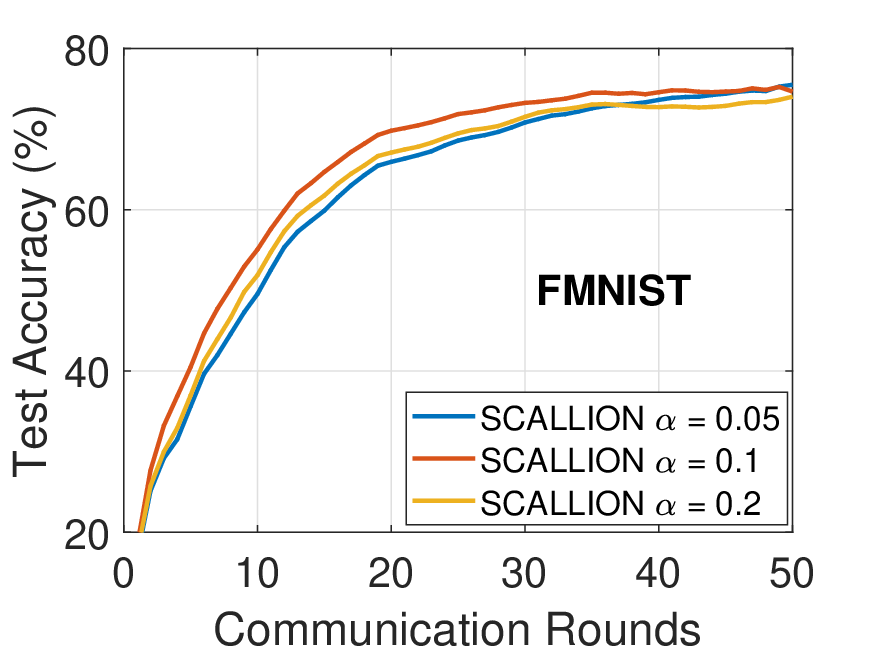}
        }
        
    \end{center}
    
    \vspace{-0.15in}
    
    \caption{Test accuracy on MNIST and FMNIST of \scafcom (biased {\sc Top}-0.01) and \scallion (unbiased 2-bit random dithering), with different $\beta$ and $\alpha$ values.}	\label{fig:ablation}
\end{figure}

\section{Conclusion}

This paper proposes two compressed federated learning (FL) algorithms, \scallion and \scafcom, to support unbiased and biased compression in FL. The proposed methods are built upon our new implementation of the stochastic controlled averaging approach (\scaffold), along with local momentum, and communication compression. Theoretically, under minimal assumptions, \scallion and \scafcom match or improve the state-of-the-art convergence rates and complexities of compressed FL algorithms. Specifically, \scallion and \scafcom are the {\em first} stochastic FL methods, to the best of our knowledge, that exhibit robustness
to arbitrary data heterogeneity, partial participation, local updates, and also accommodate communication
compression relying solely on standard compressibilities. Empirically, experiments show that \scallion and  \scafcom  outperform prior compressed FL methods and perform comparably to full-precision FL approaches at a  substantially reduced communication cost. In the future, our algorithms and techniques might be integrated with or extended to, for example, adaptive optimization, privacy, and fairness in federated learning.

%\clearpage\newpage

\bibliography{refs_scholar}
\bibliographystyle{plainnat}

\newpage
\appendix

% \begin{algorithm}[h]
% 	\caption{One main round of \scaffold~\citep{karimireddy2020scaffold}}
% 	\label{alg:scaffold}
% 	\begin{algorithmic}
% 		\STATE Sample clients $\cS^t\subseteq[N]$
%             \FOR{client $i\in\cS^t$ in parallel}
%             \STATE Receive $x^t$ and $c^t$; initialize $y_i^{t,0}=x^t$
%                 \FOR{$k=0,\dots, K-1$}
%                 \STATE Compute a mini-batch grad. $g_i^{t,k}=\nabla F(y_i^{t,k};\xi_i^{t,k})$
%                 \STATE Locally update $y_i^{t,k+1}=y_i^{t,k}-\eta_l(g_i^{t,k}-c_i^t+c^t)$
%                 \ENDFOR
%                 \STATE Update $c_i^{t+1}=
%                     c_i^t-c^t+\frac{x^t-y_i^{t,K}}{\eta_l K}$ (for $i\notin \cS^t$, $c_i^{t+1}=c_i^t$)\\
%                 \STATE Send $y_i^{t,K}-x^t$ and $c_i^{t+1}-c_i^t$ to the server
%             \ENDFOR
%   		\STATE Update $x^{t+1}=x^t+\frac{\eta_g}{S} \sum_{i\in\cS^t}(y_i^{t,K}-x^t)$
%                         \STATE Update $c^{t+1}=c^t+\frac{1}{N}\sum_{i\in\cS^t}(c_i^{t+1}-c_i^t)$
% 	\end{algorithmic}
% \end{algorithm}

\section{Detailed Implementations of \scaffold}\label{app:scaffold}
\begin{algorithm}[h]
	\caption{\scaffold: Stochastic controlled averaging for FL~\citep{karimireddy2020scaffold}}
	\label{alg:scaffold}
	\begin{algorithmic}[1]
		\STATE \noindent {\bfseries Input:} initial model $x^0$ and control variables $\{c_i^0\}_{i=1}^N$, $c^0$; local learning rate $\eta_l$; global learning rate $\eta_g$; local steps $K$; number of sampled clients $S$
		\FOR{$t=0,\cdots,T-1$}
		\STATE Uniformly sample clients $\cS^t\subseteq[N]$ with $|\cS^t|=S$
            \FOR{client $i\in\cS^t$ in parallel}
            \STATE Receive $x^t$ and $c^t$; initialize $y_i^{t,0}=x^t$
                \FOR{$k=0,\dots, K-1$}
                \STATE Compute a mini-batch gradient $g_i^{t,k}=\nabla F(y_i^{t,k};\xi_i^{t,k})$
                \STATE Locally update $y_i^{t,k+1}=y_i^{t,k}-\eta_l(g_i^{t,k}-c_i^t+c^t)$
                \ENDFOR

                \STATE Update $c_i^{t+1}=
                    c_i^t-c^t+\frac{x^t-y_i^{t,K}}{\eta_l K}$ (for $i\notin \cS^t$, $c_i^{t+1}=c_i^t$)\\
                \STATE Send $y_i^{t,K}-x^t$ and $c_i^{t+1}-c_i^t$ to the server
            \ENDFOR
  		\STATE Update $x^{t+1}=x^t+\frac{\eta_g}{S} \sum_{i\in\cS^t}(y_i^{t,K}-x^t)$
                         \STATE Update $c^{t+1}=c^t+\frac{1}{N}\sum_{i\in\cS^t}(c_i^{t+1}-c_i^t)$
		\ENDFOR 
		% \RETURN 
	\end{algorithmic}
\end{algorithm}
The original implementation of \scaffold~\citep{karimireddy2020scaffold} is stated in Algorithm~\ref{alg:scaffold} where no compression is employed in communication. In this implementation, each participating client needs to transmit the increments of both local model $y^{t,K}-y^{t,0}$ and control variable $c_i^{t+1}-c_i^t$ to the server at the end of local updates, resulting to two rounds of uplink communication for per training iteration. 

By communicating the increment variable $\Delta_i^t$, we can implement \scaffold equivalently with only a single round of uplink communication for each participating client, as described in Algorithm~\ref{alg:scaffold-new}.
\begin{algorithm}[h]
	\caption{\scaffold: An equivalent implementation with single-variable uplink communication}
	\label{alg:scaffold-new}
	\begin{algorithmic}[1]
		\STATE \noindent {\bfseries Input:} initial model $x^0$ and control variables $\{c_i^0\}_{i=1}^N$, $c^0$; local learning rate $\eta_l$; global learning rate $\eta_g$; local steps $K$; number of sampled clients $S$
		\FOR{$t=0,\cdots,T-1$}
		\STATE Uniformly sample clients $\cS^t\subseteq[N]$ with $|\cS^t|=S$
            \FOR{client $i\in\cS^t$ in parallel}
            \STATE Receive $x^t$ and $c^t$; initialize $y_i^{t,0}=x^t$
                \FOR{$k=0,\dots, K-1$}
                \STATE Compute a mini-batch gradient $g_i^{t,k}=\nabla F(y_i^{t,k};\xi_i^{t,k})$
                \STATE Locally update $y_i^{t,k+1}=y_i^{t,k}-\eta_l(g_i^{t,k}-c_i^t+c^t)$
                \ENDFOR

                \STATE Compute $\Delta_i^t = \frac{x^t-y_i^{t,K}}{\eta_l K} - c^t$ and send $\Delta_i^t$ to the server
                \STATE Update $c_i^{t+1}=
                    c_i^t+\Delta_i^t$ (for $i\notin \cS^t$, $c_i^{t+1}=c_i^t$)\\
            \ENDFOR
  		\STATE Update $x^{t+1}=x^t-\frac{\eta_g\eta_l K}{S} \sum_{i\in\cS^t}(\Delta_i^t+c^t)$
                         \STATE Update $c^{t+1}=c^t+\frac{1}{N}\sum_{i\in\cS^t}\Delta_i^t$
		\ENDFOR 
		% \RETURN 
	\end{algorithmic}
\end{algorithm}

\section{Preliminaries of Proofs}
% We first consider contractive compression, which is presumably more challenging than unbiased compression. 

% Let $\EE_r [\cdot]:=\EE [\cdot|\mathcal{F}^{r-1}]$ be the expectation with respect to the random variables $\{\cS^t,\{\xi_i^{r,k}\}_{i\in[N],0\leq k<K}\}$ in the $r$-th iteration. 
% We  use $\EE[\cdot]$ to denote the global expectation over all randomness in algorithms.
% We define the following auxiliary variables to facilitate proofs:

Letting $\gamma\triangleq \eta_g\eta_l K$, the recursion of \scallion and \scafcom can be formulated as $x^{t+1}=x^t-\gamma \tilde d^{t+1}$ where $\tilde d^{t+1}=
    \frac{1}{S}\sum_{i\in\cS^t}\tilde \delta_i^t+c^t$ and $\tilde \delta_i^t=\cC_i(\delta_i^t)$.
For clarity, 
% we define some notational variables as follows. W
we let $g_i^{t}\triangleq \frac{1}{K}\sum_{k=0}^{K-1}g_i^{t,k}$, $g^t\triangleq \frac{1}{N}\sum_{i=1}^Ng_i^t$, $d^{t+1}\triangleq\frac{1}{S}\sum_{i\in\cS^t} \delta_i^t+c^t$.
We will abbreviate $\sum_{i=1}^N$, $\sum_{k=0}^{K-1}$, $\sum_{i=1}^N\sum_{k=0}^{K-1}$ as $\sum_{i}$, $\sum_{k}$, $\sum_{i,k}$, respectively, when there is no confusion.
We also define the  auxiliary variable $U^t := \frac{1}{NK}\sum_{i=1}^N\sum_{k=1}^{K-1}\EE [\|y_i^{t,k}-x^t\|]^2$ to facilitate the analyses. It is  worth noting that $c^t\equiv \frac{1}{N}\sum_{i=1}^Nc_i^t$ for all $t\geq 0$ in both \scallion and \scafcom. Besides, the exclusive recursions are as follows.

\paragraph{For \scallion.}
Due to client sampling, it holds that
\begin{align}\label{eqn:key-update-scallion}
    c_i^{t+1}=\begin{cases}
        c_i^t+\cC_i(\alpha(g_i^t-c_i^t))&\text{if }i\in\cS^t\\
        c_i^t&\text{otherwise}
    \end{cases}
\end{align}
and $d^{t+1}=\frac{1}{S}\sum_{i\in\cS^t}\alpha(g_i^t-c_i^t)+c^t$

\paragraph{For \scafcom.}
We additionally let $u_i^{t+1}\triangleq v_i^t+\beta(g_i^t-v_i^t)$. Then, due to client sampling, it holds that
\begin{align}\label{eqn:key-update-scafcom}
    (v_i^{t+1},c_i^{t+1})=\begin{cases}
        (u_i^{t+1},c_i^t+\cC_i(u_i^{t+1}-c_i^t))&\text{if }i\in\cS^t\\
        (v_i^t,c_i^t)&\text{otherwise}
    \end{cases}
\end{align}
and $d^{t+1}=\frac{1}{S}\sum_{i\in\cS^t} (u_i^{t+1}-c_i^t)+c^t$. Similarly, we let $v^t\triangleq \frac{1}{N}\sum_{i=1}^N v_i^t$ and  $u^{t+1}\triangleq \frac{1}{N}\sum_{i=1}^Nu_i^{t+1}=(1-\beta)v^t+\beta g^t$.

Let $\mathcal{F}^{-1}=\emptyset$ and 
$\mathcal{F}^{t,k}_i:= \sigma(\cup\{\xi^{t,j}_i\}_{0\leq j< k}\cup \mathcal{F}^{t-1})$ 
and  $\mathcal{F}^{t}:=\sigma((\cup_{i\in[N]} \mathcal{F}^{t,K}_i)\cup\{\cS^t\} )$ for all $t\geq 0$ where $\sigma(\cdot)$ means the $\sigma$-algebra. We use $\EE[\cdot]$ to indicate the expectation taking all sources of randomness into account.
We will  frequently use the following fundamental lemmas.
\begin{lemma}[\citep{cheng2024momentum}]\label{lem:par_sample}
    Given any $a_1, \cdots, a_N, b\in \RR^d$ and $  a=\frac{1}{N}\sum_{i\in[N]}a_i$ and uniform sampling $\mathcal{S}\subset [N]$ without replacement such that $|\mathcal{S}|=S$, it holds that
    \begin{align*}
            \EE_{\cS} \left[\left\|\frac{1}{S}\sum_{i\in \mathcal{S}}a_i\right\|^2 \right]
            % =&\left(1-\frac{N-S}{S(N-1)}\right)\|  a\|^2+\frac{N-S}{S(N-1)}\frac{1}{N}\sum_{i\in[N]}\|a_i\|^2\\
            \leq  &\|  a\|^2 + \frac{1}{SN}\sum_{i}\|a_i-{a}\|^2\leq \|  a\|^2 + \frac{1}{SN}\sum_{i}\|a_i-b\|^2.
    \end{align*}
\end{lemma}
% In the below analysis, we denote factor $\q$ as $\theta$. Note that $0\leq \theta\leq 1$, $\frac{1-\theta}{N}+\theta=\frac{1}{S}$, and $\theta=0$ so long as $S=N$.

\begin{lemma}[\citep{karimireddy2020scaffold}]\label{lem:bias-var}
Suppose $\{X_1,\cdots,X_{\tau}\}\subseteq \RR^d$ be  random variables that are potentially dependent. If their marginal means and variances satisfy $\EE[X_i] = \mu_i$ and  $\EE[\|X_i - \mu_i\|^2]\leq \sigma^2$, then it holds that
\[
    \EE\left[\left\|\sum_{i=1}^\tau X_i\right\|^2\right] \leq \left\|\sum_{i=1}^\tau \mu_i\right\|^2+ \tau^2 \sigma^2.
\]
If they are correlated in the Markov sense with conditional means and variances satisfying  $\EE [X_i | X_{i-1}, \cdots X_{1}] = \mu_i$ and  $\EE[\|X_i - \mu_i\|^2 \mid \mu_i]\leq \sigma^2$. Then a tighter bound holds 
\[
    \EE\left[\left\|\sum_{i=1}^\tau X_i\right\|^2\right] \leq 2\EE\left[\left\|\sum_{i=1}^\tau \mu_i\right\|^2\right]+ 2\tau\sigma^2.
\]
\end{lemma}

\begin{lemma}\label{lem:split-noise}Under Assumption \ref{asp:smooth}, for any $\theta\in[0,1]$ and $v,v_1,\dots,v_N\in \cF^{t-1}$, it holds that
    \begin{align}\label{eqn:key-bound1}
    \EE[\|(1-\theta)v+\theta (g^t-\nabla f(x^t))\|^2]\leq\min\left\{ 2\EE[\|(1-\theta)v\|^2]+3\theta^2L^2 U^t,(1-\theta)\EE[\|v\|^2]+2\theta L^2U^t\right\}+\frac{2\theta^2\sigma^2}{NK}
\end{align}
and
\begin{align}
    &\frac{1}{N}\sum_i\EE[\|(1-\theta)v_i+\theta (g_i^t-\nabla f_i(x^t))\|^2]\\
        \leq &\min\left\{ \frac{2}{N}\sum_i\EE[\|(1-\theta)v_i\|^2]+3\theta^2L^2U^t, \frac{1-\theta}{N}\sum_i\EE[\|v_i\|^2]+2\theta L^2 U^t\right\}+\frac{2\theta^2\sigma^2}{K}.\label{eqn:key-bound2}
\end{align}
% and
% \begin{align}\label{eqn:key-bound3}
%     &\EE\left[\left\|\frac{1}{S}\sum_{i\in\cS^t}\left((1-\theta)v_i+\theta (g_i^t-\nabla f_i(x^t))\right)\right\|^2\right]\\
%         \leq &\min\Bigg\{ 2\left((1-\theta)\EE[\|v\|^2]+\frac{\theta}{SN}\sum_i\EE[\|v_i\|^2]\right)+3\theta^2U^t,\\
%         & \qquad \qquad (1-\theta)\left((1-\theta) \EE[\|v\|^2]+\frac{\theta}{N}\sum_i\EE[\|v_i\|^2]\right)+2\theta U^t\Bigg\}+\frac{2\theta^2\sigma^2}{SK}.
% \end{align}
\end{lemma}
\begin{proof}
We mainly focus on proving \eqref{eqn:key-bound1} as \eqref{eqn:key-bound2} can be established similarly. Using Lemma~\ref{lem:bias-var}, we have
    \begin{align}
        &\EE[\|(1-\theta)v+\theta (g^t-\nabla f(x^t))\|^2]\\
=&\EE[\|(1-\theta)v\|^2]+\EE\left[\left\langle (1-\theta)v, \frac{\theta }{NK}\sum_{i,k}\nabla F(y_i^{t,k};\xi_i^{t,k})-\nabla f_i(x^t)\right\rangle\right]+\theta^2\EE\left[\left\|\frac{1}{NK}\sum_{i,k}\nabla F(y_i^{t,k};\xi_i^{t,k})-\nabla f_i(x^t)\right\|^2\right]\\
        \leq &\EE[\|(1-\theta)v\|^2]+\EE\left[\left\langle (1-\theta)v, \frac{\theta }{NK}\sum_{i,k}\nabla f_i(y_i^{t,k})-\nabla f_i(x^t)\right\rangle\right]\\
        &\quad +2\theta^2\EE\left[\left\|\frac{1}{NK}\sum_{i,k}\nabla f_i(y_i^{t,k})-\nabla f_i(x^t)\right\|^2\right]+\frac{2\theta^2\sigma^2}{NK}\\
        \leq &\EE\left[\left\|(1-\theta)v+ \frac{\theta }{NK}\sum_{i,k}\nabla f_i(y_i^{t,k})-\nabla f_i(x^t)\right\|^2\right]+\theta^2\EE\left[\left\|\frac{1}{NK}\sum_{i,k}\nabla f_i(y_i^{t,k})-\nabla f_i(x^t)\right\|^2\right]+\frac{2\theta^2\sigma^2}{NK}.
    \end{align}
    By further applying Sedrakyan's inequality $\|(1-\theta)v+\theta v^\prime\|^2\leq (1-\theta)\|v\|^2+\theta \|v^\prime\|^2$ and Assumption~\ref{asp:smooth}, we have
    \begin{align}
        &\EE[\|(1-\theta)v+\theta (g^t-\nabla f(x^t))\|^2]\\
        \leq &(1-\theta)\EE[\|v\|^2]+ \theta(1+\theta)\EE\left[\left\|\frac{1}{NK}\sum_{i,k}\nabla f_i(y_i^{t,k})-\nabla f_i(x^t)\right\|^2\right]+\frac{2\theta^2\sigma^2}{NK}\\
        \leq &(1-\theta)\EE[\|v\|^2]+2\theta L^2 U^t+\frac{2\theta^2\sigma^2}{NK}.
    \end{align}
    The other upper bound of \eqref{eqn:key-bound1} follows from $\|(1-\theta)v+\theta v^\prime\|^2\leq 2\|(1-\theta)v\|^2+2\theta^2 \|v^\prime\|^2$.
    % We can similarly establish \eqref{eqn:key-bound2}. 
    % For \eqref{eqn:key-bound3}, we shall use Lemma~\ref{lem:par_sample} first to obtain
    % \begin{align}
    %     &\EE\left[\left\|\frac{1}{S}\sum_{i\in\cS^t}\left((1-\theta)v_i+\theta (g_i^t-\nabla f_i(x^t))\right)\right\|^2\right]\\
    %     =&\left(1-\theta\right)\EE[\|(1-\theta)v+\theta (g^t-\nabla f(x^t))\|^2]+\frac{\theta}{N}\sum_i\EE[\|(1-\theta)v_i+\theta (g_i^t-\nabla f_i(x^t))\|^2]
    % \end{align}
    % and the rest follows \eqref{eqn:key-bound1} and \eqref{eqn:key-bound2}.
\end{proof}

% Throughout this section, we let $\Delta^0 :=f(x^0)-f^*, G^0 :=\frac{1}{N}\sum_{i=1}^N\|\nabla f_i(x^0)\|^2$ and $x^{-1}:=x^0$. We will use the following lemma for all our algorithms.

\newpage
\section{Proof of \scallion}\label{app:scallion}
In this section, we prove the convergence result of \scallion with unbiased compression, where we  additionally define $x^{-1}:=x^0$. {We thus have $\EE[\|x^t-x^{t-1}\|^2]=0$ for $t=0$. Note that $x^{-1}$ is defined for the purpose of notation and is not utilized in our algorithms.}

\begin{lemma}[\sc Descent lemma]\label{lem:descent-u}
  Under Assumptions~\ref{asp:smooth} and~\ref{asp:gd-noise}, it holds for all $t\geq 0$ and $\gamma>0$ that
  \begin{align}
      &\EE[f(x^{t+1})]\\
      \leq &\EE[f(x^t)]-\frac{\gamma}{2}\EE[\|\nabla f(x^t)\|^2]-\left(\frac{1}{2\gamma}-\frac{L}{2}\right)\EE[\|x^{t+1}-x^t\|^2]+4\gamma \left(1+\frac{(1+\omega)\alpha^2}{S}\right)L^2\EE[\|x^t-x^{t-1}\|^2]\\
      &+\frac{8\gamma(1+\omega)\alpha^2\sigma^2}{SK}+4\gamma \left(1+\frac{(1+\omega)\alpha^2}{S}\right)L^2U^t\\
      &+\gamma \left(4\EE[\|c^t-\nabla f(x^{t-1})\|^2]+\frac{4(1+\omega)\alpha^2}{S}\frac{1}{N}\sum_{i}\EE[\|c_i^{t}-\nabla f_i(x^{t-1})\|^2]\right).\label{eqn:descentu}
  \end{align}
  % Additionally for $t=0$, we have 
  % \begin{align}
  %     \EE[f(x^{1})]\leq &\EE[f(x^0)]-\frac{\gamma}{2}\EE[\|\nabla f(x^0)\|^2]-\left(\frac{1}{2\gamma}-\frac{L}{2}\right)\EE[\|x^{1}-x^0\|^2]+\left(S^{-1}+q^2\right)\frac{4\gamma\beta^2\sigma^2}{K}+9\gamma\beta^2L^2U^0\\
  %     &+\gamma \left(4\EE[\|v^0-\nabla f(x^{0})\|^2]+\beta^2\left(S^{-1}+q^2\right)\frac{6}{N}\sum_{i}\EE[\|v_i^{0}-\nabla f_i(x^{0})\|^2]+\left(S^{-1}+q^2\right)\frac{6}{N}\sum_{i}\EE[\|v_i^{0}-c_i^0\|^2]\right).
  % \end{align}
\end{lemma}
\begin{proof}
    By Lemma 2 of~\citet{li2021page}, we have 
    \begin{equation}
      f(x^{t+1})\leq f(x^t)-\frac{\gamma}{2}\|\nabla f(x^t)\|^2-\left(\frac{1}{2\gamma}-\frac{L}{2}\right)\|x^{t+1}-x^t\|^2+\frac{\gamma}{2}\|\tilde d^{t+1}-\nabla f(x^t)\|^2
  \end{equation}
  where $\tilde d^{t+1}=\frac{1}{S}\sum_{i\in\cS^t}\tilde \delta_i^t+c^t$. Letting $ d^{t+1}\triangleq \frac{1}{S}\sum_{i\in\cS^t} \delta_i^t+c^t$,  we further have
  \begin{equation}
      f(x^{t+1})\leq f(x^t)-\frac{\gamma}{2}\|\nabla f(x^t)\|^2-\left(\frac{1}{2\gamma}-\frac{L}{2}\right)\|x^{t+1}-x^t\|^2+\gamma\|\tilde d^{t+1}-d^{t+1}\|^2+\gamma\|d^{t+1}-\nabla f(x^{t})\|^2.\label{eqn:bnifnbsdgsef0u}
  \end{equation}
    For $\|d^{t+1}-\nabla f(x^{t})\|^2$, using Lemma~\ref{lem:bias-var} and the fact that $c^t\equiv \frac{1}{N}\sum_{i}c_i^t$ and $d^{t+1}=\frac{1}{S}\sum_{i\in\cS^t} \alpha(g_i^t-c_i^t)+c^t$, we have 
  \begin{align}
      \EE[\|d^{t+1}-\nabla f(x^{t})\|^2]=&\EE\left[\left\|c^t+\frac{1}{S}\sum_{i\in\cS^t}\alpha(g_i^t-c_i^t)-\nabla f(x^{t})\right\|^2\right]\\
      \leq &\EE[\|(1-\alpha)(c^t-\nabla f(x^{t}))+\alpha (g^t-\nabla f(x^t))\|^2]+ \frac{\alpha^2}{SN}\sum_{i}\EE[\|g_i^t-c_i^t\|^2].\label{eqn:vninvsd0u}
  \end{align}
    Using \eqref{eqn:key-bound1} and Assumption~\ref{asp:smooth}, we have 
  \begin{align}
      &\EE[\|(1-\alpha)(c^t-\nabla f(x^t))+\alpha (g^t-\nabla f(x^{t})\|^2]\\
      \leq & 2\EE[\|c^t-\nabla f(x^t)\|^2]+3\alpha^2U^t+\frac{2\alpha^2\sigma^2}{NK}\\
      \leq &4\EE[\|c^t-\nabla f(x^{t-1})\|^2]+4L^2\EE[\|x^t-x^{t-1}\|^2]+3\alpha^2L^2U^t+\frac{2\alpha^2\sigma^2}{NK}.\label{eqn:vninvsd1u}
  \end{align}
    Similarly,  using \eqref{eqn:key-bound2} and Assumption~\ref{asp:smooth}, we have 
    \begin{align}
        \frac{\alpha^2}{SN}\sum_{i}\EE[\|g_i^t-c_i^t\|^2]\leq &\frac{2\alpha^2}{SN}\sum_{i}\EE[\|c_i^{t}-\nabla f_i(x^t)\|^2]+\frac{2\alpha^2}{SN}\sum_{i}\EE[\|g_i^t-\nabla f_i(x^t)\|^2]\\
        \leq  &\frac{2\alpha^2}{SN}\sum_{i}\EE[\|c_i^{t}-\nabla f_i(x^t)\|^2]+\frac{4\alpha^2}{S}L^2U^t+\frac{4\alpha^2\sigma^2}{SK}\\
        \leq  &\frac{4\alpha^2}{SN}\sum_{i}\EE[\|c_i^{t}-\nabla f_i(x^{t-1})\|^2]+\frac{4\alpha^2L^2}{S}\EE[\|x^{t}-x^{t-1}\|^2]+\frac{4\alpha^2}{S}L^2U^t+\frac{4\alpha^2\sigma^2}{SK}.\label{eqn:vninvsd2u}
    \end{align}
    Plugging \eqref{eqn:vninvsd1u} and \eqref{eqn:vninvsd2u} into \eqref{eqn:vninvsd0u}, we obtain
    \begin{align}
        &\EE[\|d^{t+1}-\nabla f(x^{t})\|^2]\\
        \leq &4\EE[\|c^t-\nabla f(x^{t-1})\|^2]+\frac{4\alpha^2}{S}\frac{1}{N}\sum_{i}\EE[\|c_i^{t}-\nabla f_i(x^{t-1})\|^2]\\
        &\quad  + 4(1+S^{-1})\alpha^2 L^2U^t+4\left(1+\frac{\alpha^2}{S}\right)L^2\EE[\|x^t-x^{t-1}\|^2]+(N^{-1}+S^{-1})\frac{4\alpha^2\sigma^2}{K}.\label{eqn:bnifnbsdgsef1u}
    \end{align}
    For $\|\tilde d^{t+1}-d^{t+1}\|^2$, using mutual independence and Definition~\ref{def:contract}, we have 
    \begin{align}
        \EE[\|\tilde d^{t+1}-d^{t+1}\|^2]=&\EE\left[\left\|\frac{1}{S}\sum_{i\in\cS^t}\cC_i(\alpha(g_i^{t}-c_i^t))-\alpha(g_i^{t}-c_i^t)\right\|^2\right]\\
        \leq &\frac{\omega \alpha^2}{S^2}\EE\left[\sum_{i\in\cS^t}\|g_i^{t}-c_i^t\|^2\right]=\frac{\omega\alpha^2}{S}\frac{1}{N}\sum_{i}\EE[\|g_i^{t}-c_i^t\|^2].
    \end{align}
    Then applying the same relaxation in \eqref{eqn:vninvsd2u}, we obtain
    \begin{align}
        \EE[\|\tilde d^{t+1}-d^{t+1}\|^2]
        \leq &\frac{4\omega\alpha^2}{SN}\sum_{i}\EE[\|c_i^{t}-\nabla f_i(x^{t-1})\|^2]+\frac{4\omega\alpha^2L^2}{S}\EE[\|x^{t}-x^{t-1}\|^2]+\frac{4\omega \alpha^2}{S}U^t+\frac{4\omega\alpha^2\sigma^2}{SK}.\label{eqn:bnifnbsdgsef2u}
    \end{align}
    Plugging \eqref{eqn:bnifnbsdgsef1u} and \eqref{eqn:bnifnbsdgsef2u} into \eqref{eqn:bnifnbsdgsef0u} and noting $N^{-1}\leq S^{-1}$, we complete the proof.
\end{proof}

Given Lemma~\ref{lem:descent-u}, the rest is to bound $\|c^t-\nabla f(x^{t-1})\|^2$, $\|c_i^t-\nabla f_i(x^{t-1})\|^2$.
\begin{lemma}\label{lem:c-f-bound-u}
    Under Assumptions~\ref{asp:smooth} and~\ref{asp:gd-noise}, it holds for all $t\geq 0$ that 
    \begin{align}
        &\EE[\|c^{t+1}-\nabla f(x^{t})\|^2]\\
        \leq& \left(1-\frac{S\alpha}{2N}\right)\EE[\|c^{t}-\nabla f(x^{t})\|^2]+\frac{4(1+\omega)\alpha^2S}{N^2}\frac{1}{N}\sum_i\EE[\|c_i^t-\nabla f_i(x^{t-1})\|^2]\\
        &\quad +\left(\frac{2N}{S\alpha}+\frac{4(1+\omega)\alpha^2S}{N^2}\right)L^2\EE[\|x^{t}-x^{t-1}\|^2]+\left(\frac{2S}{N}+\frac{4(1+\omega)\alpha S}{N^2}\right)\alpha L^2U^t+\frac{6(1+\omega)\alpha^2S\sigma^2}{N^2K}.\label{eqn:vhidhfsv cxu}
    \end{align}
\end{lemma}
\begin{proof}
Using \eqref{eqn:key-update-scallion} and Lemma~\ref{lem:par_sample}, we have
    \begin{align}
        \EE[\|c^{t+1}-\nabla f(x^{t})\|^2]=& \EE\left[\left\|\frac{1}{S}\sum_{i\in\cS^{t}}\frac{S}{N}\cC_i(\alpha( {g}_i^{t}-c_i^{t}))+c^{t}-\nabla f(x^{t})\right\|^2\right]\\
        \leq & \EE\left[\left\|\frac{1}{S}\sum_{i\in\cS^{t}}\frac{S\alpha}{N}( {g}_i^{t}-c_i^{t})+c^{t}-\nabla f(x^{t})\right\|^2\right]+\frac{\omega\alpha^2}{N^2}\sum_{i\in\cS^t}\EE[\|g_i^t-c_i^t\|^2]\\
        \leq & \EE\left[\left\|\frac{S\alpha}{N}( {g}^{t}-c^{t})+c^{t}-\nabla f(x^{t})\right\|^2\right]+\frac{(1+\omega)\alpha^2S}{N^2}\frac{1}{N}\sum_{i}\EE[\|g_i^t-c_i^t\|^2]\label{eqn:nvosdnvsd1-scal}
    \end{align}
    Using \eqref{eqn:key-bound1}, Young's inequality, and Assumption~\ref{asp:smooth}, we further have
    \begin{align}
        &\EE\left[\left\|\frac{S\alpha}{N}( {g}^{t}-c^{t})+c^{t}-\nabla f(x^{t})\right\|^2\right]\\
        \leq &\left(1-\frac{S\alpha}{N}\right)\EE[\|c^{t}-\nabla f(x^{t})\|^2]+\frac{2\alpha SL^2}{N}U^t+\frac{2\alpha^2S^2\sigma^2}{N^3K}\\
        \leq &\left(1-\frac{S\alpha}{2N}\right)\EE[\|c^{t}-\nabla f(x^{t-1})\|^2]+\frac{2NL^2}{S\alpha}\EE[\|x^t-x^{t-1}\|^2]+\frac{2\alpha SL^2}{N}U^t+\frac{2\alpha^2S^2\sigma^2}{N^3K}.\label{eqn:nvosdnvsd2-scal}
    \end{align}
    Using Young's inequality and Assumption~\ref{asp:smooth}, we can obtain 
    \begin{align}
        &\frac{(1+\omega)\alpha^2 S }{N^2}\frac{1}{N}\sum_{i}\EE[\left\|g_i^{t}-c_i^t\right\|^2]\\
        \leq & \frac{(1+\omega)\alpha^2 S }{N^2}\frac{1}{N}\sum_{i}\left(2\EE[\|c_i^t-\nabla f_i(x^t)\|^2]+2\EE[\|g_i^{t}-\nabla f_i(x^t)\|^2]\right)\\
        \leq &\frac{(1+\omega)\alpha^2 S }{N^2}\frac{1}{N}\sum_{i}\left(4\EE[\|c_i^t-\nabla f_i(x^{t-1})\|^2]+4L^2\EE[\|x^t-x^{t-1}\|^2]+2\EE[\|g_i^{t}-\nabla f_i(x^t)\|^2]\right).\label{eqn:nvosdnvsd3-scal}
    \end{align}
    Using \eqref{eqn:key-bound2}, we have 
    \begin{align}
        \frac{1}{N}\sum_{i}\EE[\|g_i^{t}-\nabla f_i(x^t)\|^2]\leq 2L^2U^t +\frac{2\sigma^2}{K}.\label{eqn:vidngsdgs-scal}
    \end{align}
    Plugging \eqref{eqn:vidngsdgs-scal} into \eqref{eqn:nvosdnvsd3-scal}, we reach
    \begin{align}
        &\frac{(1+\omega)\alpha^2 S }{N^2}\frac{1}{N}\sum_{i}\EE[\left\|g_i^{t}-c_i^t\right\|^2]\\
        \leq &\frac{(1+\omega)\alpha^2 S }{N^2}\frac{1}{N}\sum_{i}\left(4\EE[\|c_i^t-\nabla f_i(x^{t-1})\|^2]+4L^2\EE[\|x^t-x^{t-1}\|^2]+4L^2U^t+\frac{4\sigma^2}{K}\right).\label{eqn:nvosdnvsd4-scal}
    \end{align}
    Combining \eqref{eqn:nvosdnvsd1-scal}, \eqref{eqn:nvosdnvsd2-scal}, \eqref{eqn:nvosdnvsd4-scal} together and using $\frac{\alpha^2S^2\sigma^2}{N^3K}\leq \frac{(1+\omega)\alpha^2S\sigma^2}{N^2K}$ completes the proof.
\end{proof}

\begin{lemma}\label{lem:vi-fi-bound-scal}
    Under Assumptions~\ref{asp:smooth} and~\ref{asp:gd-noise}, suppose $0\leq \alpha \leq \frac{1}{4(\omega+1)}$, then it holds for all $t\geq 0$ that
    \begin{align}
      \frac{1}{N}\sum_{i}\EE[\|c_i^{t+1}-\nabla f_i(x^{t})\|^2]
        \leq &\left(1-\frac{S\alpha}{4N}\right) \frac{1}{N}\sum_{i}\EE[\|c_i^{t}-\nabla f_i(x^{t-1})\|^2]+\frac{4NL^2}{S\alpha}\EE[\|x^t-x^{t-1}\|^2]\\
        &\quad +\frac{3\alpha L^2S}{N} U^t+\frac{2(1+\omega)\alpha^2S\sigma^2}{NK}.\label{eqn:vidnvdcxvcu}
   \end{align}
\end{lemma}
\begin{proof}
Using \eqref{eqn:key-update-scallion}, we have
    \begin{align}
        &\frac{1}{N}\sum_{i}\EE[\|c_i^{t+1}-\nabla f_i(x^{t})\|^2]\\
        =&\frac{1}{N}\sum_{i}\left(\left(1-\frac{S}{N}\right) \EE[\|c_i^{t}-\nabla f_i(x^{t})\|^2]+\frac{S}{N}\EE[\|\cC_i(\alpha( {g}_i^{t}-c_i^{t}))+c_i^{t}-\nabla f_i(x^{t})\|^2]\right)\\
        \leq &\frac{1}{N}\sum_{i}\left(\left(1-\frac{S}{N}\right) \EE[\|c_i^{t}-\nabla f_i(x^{t})\|^2]+\frac{S}{N}\EE[\|\alpha( {g}_i^{t}-c_i^{t})+c_i^{t}-\nabla f_i(x^{t})\|^2]+\frac{S\omega\alpha^2 }{N}\EE[\|g_i^t-c_i^t\|^2]\right)\label{eqn:nvosdnvsd11u}
    \end{align}
    Note that 
    \begin{align}
        \frac{S}{N}\frac{1}{N}\sum_{i}\EE[\|\alpha( {g}_i^{t}-c_i^{t})+c_i^{t}-\nabla f_i(x^{t})\|^2]=&\frac{S}{N}\frac{1}{N}\sum_{i}\EE[\|\alpha( {g}_i^{t}-\nabla f_i(x^t))+(1-\alpha)(c_i^{t}-\nabla f_i(x^{t}))\|^2]\\
        \leq &\frac{S}{N}\left(\frac{1-\alpha}{N}\sum_{i}\EE[\|c_i^{t}-\nabla f_i(x^{t})\|^2]+2\alpha L^2 U^t+\frac{2\alpha^2\sigma^2}{K}\right)\label{eqn:vnixnvixcvx1u}
    \end{align}
    and by applying  \eqref{eqn:key-bound2},
    \begin{align}
        \frac{\omega\alpha^2S}{N}\frac{1}{N}\sum_i\EE[\|g_i^t-c_i^t\|^2]=&\frac{4\omega\alpha^2S}{N}\frac{1}{N}\sum_i\EE\left[\left\|\frac{1}{2}(g_i^t-\nabla f_i(x^t))-\frac{1}{2}(\nabla f_i(x^t)-c_i^t)\right\|^2\right]\\
        \leq &\frac{4\omega\alpha^2S}{N}\left(\frac{1}{2N}\sum_i\EE[\|\nabla f_i(x^t)-c_i^t\|^2] +L^2U^t+\frac{\sigma^2}{2K}\right).\label{eqn:vnixnvixcvx2u}
    \end{align}
    Plugging \eqref{eqn:vnixnvixcvx1u} and \eqref{eqn:vnixnvixcvx2u} into \eqref{eqn:nvosdnvsd11u}, we obtain
    \begin{align}
        &\frac{1}{N}\sum_{i}\EE[\|c_i^{t+1}-\nabla f_i(x^{t})\|^2]\\
        \leq &\left(1-\frac{S\alpha(1-2\omega \alpha)}{N}\right)\frac{1}{N}\sum_{i} \EE[\|c_i^{t}-\nabla f_i(x^{t})\|^2]+\frac{(2+4\omega \alpha)\alpha SL^2}{N}U^t+\frac{2(1+\omega)\alpha^2S \sigma^2}{NK}\\
        \leq &\left(1-\frac{S\alpha}{2N}\right)\frac{1}{N}\sum_{i} \EE[\|c_i^{t}-\nabla f_i(x^{t})\|^2]+\frac{3\alpha SL^2}{N}U^t+\frac{2(1+\omega)\alpha^2S \sigma^2}{NK}\label{eqn:nvosdnvsd1dsa1u}
    \end{align}
    where we use $\alpha \leq \frac{1}{4(\omega+1)}$ in the last inequality. By further using Young's inequality and Assumption~\ref{asp:smooth}, we obtain
    \begin{align}
        \frac{1}{N}\sum_{i}\EE[\|c_i^{t+1}-\nabla f_i(x^{t})\|^2]
        \leq &\left(1-\frac{S\alpha}{4N}\right) \frac{1}{N}\sum_{i}\EE[\|c_i^{t}-\nabla f_i(x^{t-1})\|^2]+\frac{4NL^2}{S\alpha}\EE[\|x^t-x^{t-1}\|^2]\\
        &\quad +\frac{3\alpha L^2S}{N} U^t+\frac{2(1+\omega)\alpha^2S\sigma^2}{NK}.
    \end{align}
\end{proof}

\begin{lemma}\label{lem:u-bound-u}
    Under Assumptions~\ref{asp:smooth} and~\ref{asp:gd-noise},  it holds for any $t\geq 0$ and $\eta_l KL\leq \frac{1}{2}$ that 
    \begin{align}\label{eqn:Lvnisbvvcb-scal}
        U^t\leq \frac{9e^2K^2\eta_l^2}{N}\sum_i \left(\EE[\|c_i^t-\nabla f_i(x^{t-1})\|^2]+L^2\EE[\|x^t-x^{t-1}\|^2]+\EE[\|\nabla f(x^t)\|^2]\right)+e^2K\eta_l^2\sigma^2.
    \end{align}
    % Additionally for $t=0$, we have 
    % \begin{align}
    %     U^0\leq \frac{9e^2K^2\eta_l^2}{N}\sum_i \left(\EE[\|c_i^0-v_i^0\|^2]+\EE[\|v_i^t-\nabla f_i(x^{0})\|^2]+\EE[\|\nabla f(x^0)\|^2]\right).
    % \end{align}
\end{lemma}
\begin{proof}
    The proof is similar to that of \ref{lem:u-bound}.
\end{proof}

\begin{theorem}
    Under Assumptions~\ref{asp:smooth} and~\ref{asp:gd-noise}, suppose clients are associated with mutually independent $\omega$-unbiased compressors, if we initialize $c_i^0=\frac{1}{B}\sum_{b=1}^B\nabla F(x^0;\xi_i^b)$,  $c^0=\frac{1}{N}\sum_{i=1}^N c_i^0$ with $\{\xi_i^b\}_{b=1}^B\overset{iid}{\sim}\cD_i$ and $B\gtrsim \frac{\sigma^2}{NL\Delta}$ ($c_i^0\to \nabla f_i(x^0)$ as $B\to \infty$), set
    \begin{equation}\label{eqn:scallion-para}
        \begin{aligned}
            &\eta_l KL\leq \sqrt{\frac{\alpha (1+\omega)}{1400e^2 N}} , \quad \eta_g \eta_lKL =\frac{27\alpha S}{N},\\ 
             &\alpha =\left(4(1+\omega)+\left(\frac{(1+\omega)TS\sigma^2}{N^2KL\Delta}\right)^{1/2}+\left(\frac{(1+\omega)T\sigma^2}{NKL\Delta}\right)^{1/3}\right)^{-1},
        \end{aligned}
    \end{equation}
    % $\eta_g \eta_lKL =\frac{27\alpha S}{N}$, $\eta_l KL\leq \sqrt{\frac{\alpha (1+\omega)}{1400e^2 N}} $, and 
    % \begin{equation}
    %     \alpha =\left(4(1+\omega)+\left(\frac{(1+\omega)TS\sigma^2}{N^2KL\Delta}\right)^{1/2}+\left(\frac{(1+\omega)T\sigma^2}{NKL\Delta}\right)^{1/3}\right)^{-1},
    % \end{equation}
    then \scallion converges as 
    \begin{equation}
    \frac{1}{T}\sum_{t=0}^{T-1}\EE[\|\nabla f(x^t)\|^2]
    \lesssim \sqrt{\frac{(1+\omega)L\Delta \sigma^2}{SKT}} +\left(\frac{(1+\omega)N^2L^2\Delta^2\sigma^2}{S^3KT^2}\right)^{1/3}+\frac{(1+\omega)NL\Delta}{ST}
    \end{equation}
    where $\Delta \triangleq f(x^0)-\min f(x)$.
\end{theorem}
\begin{proof}

  % \begin{align}
  %     &\EE[f(x^{t+1})]\\
  %     \leq &\EE[f(x^t)]-\frac{\gamma}{2}\EE[\|\nabla f(x^t)\|^2]-\left(\frac{1}{2\gamma}-\frac{L}{2}\right)\EE[\|x^{t+1}-x^t\|^2]+4\gamma \left(1+\frac{(1+\omega)\alpha^2}{S}\right)L^2\EE[\|x^t-x^{t-1}\|^2]+\frac{8\gamma(1+\omega)\alpha^2\sigma^2}{SK}\\
  %     &+4\gamma \left(1+\frac{(1+\omega)\alpha^2}{S}\right)L^2U^t+\gamma \left(4\EE[\|c^t-\nabla f(x^{t-1})\|^2]+\frac{4(1+\omega)\alpha^2}{S}\frac{1}{N}\sum_{i}\EE[\|c_i^{t}-\nabla f_i(x^{t-1})\|^2]\right).\label{eqn:descent}
  % \end{align}

  %   \begin{align}
  %       &\EE[\|c^{t+1}-\nabla f(x^{t})\|^2]\\
  %       \leq& \left(1-\frac{S\alpha}{2N}\right)\EE[\|c^{t}-\nabla f(x^{t})\|^2]+\frac{4(1+\omega)\alpha^2S}{N^2}\frac{1}{N}\sum_i\EE[\|c_i^t-\nabla f_i(x^{t-1})\|^2]\\
  %       &\quad +\left(\frac{2N}{S\alpha}+\frac{4(1+\omega)\alpha^2S}{N^2}\right)L^2\EE[\|x^{t}-x^{t-1}\|^2]+\left(\frac{2S}{N}+\frac{4(1+\omega)\alpha S}{N^2}\right)\alpha L^2U^t+\frac{6(1+\omega)\alpha^2S\sigma^2}{N^2K}.\label{eqn:vhidhfsv cx}
  %   \end{align}

   %  \begin{align}
   %    \frac{1}{N}\sum_{i}\EE[\|c_i^{t+1}-\nabla f_i(x^{t})\|^2]
   %      \leq &\left(1-\frac{S\alpha}{4N}\right) \frac{1}{N}\sum_{i}\EE[\|c_i^{t}-\nabla f_i(x^{t-1})\|^2]+\frac{4NL^2}{S\alpha}\EE[\|x^t-x^{t-1}\|^2]\\
   %      &\quad +\frac{3\alpha L^2S}{N} U^t+\frac{2(1+\omega)\alpha^2S\sigma^2}{NK}.\label{eqn:vidnvdcxvc}
   % \end{align}
Adding $\eqref{eqn:vhidhfsv cxu}\times \frac{10\gamma N}{\alpha S}$ to \eqref{eqn:descentu}, we have
\begin{align}
    &\EE[f(x^{t+1})]+\frac{10\gamma N}{\alpha S}\EE[\|c^{t+1}-\nabla f(x^t)\|^2]\\
      \leq &\EE[f(x^t)]+\gamma\left(\frac{10 N}{\alpha S}-1\right)\EE[\|c^{t}-\nabla f(x^{t-1})\|^2]\\
      &\quad -\frac{\gamma}{2}\EE[\|\nabla f(x^t)\|^2]-\left(\frac{1}{2\gamma}-\frac{L}{2}\right)\EE[\|x^{t+1}-x^t\|^2]+\gamma L^2\left(\frac{44(1+\omega)\alpha}{S}+\frac{24N^2}{\alpha^2S^2}\right)\EE[\|x^t-x^{t-1}\|^2]\\
      &\quad +\gamma(1+\omega)\left(\frac{8\alpha^2}{S}+\frac{60\alpha}{N}\right)\frac{\sigma^2}{K}+\gamma L^2\left(24+\frac{4(1+\omega)\alpha^2}{S}+\frac{40(1+\omega)\alpha}{N}\right)U^t\\
      &\quad +\frac{40\gamma (1+\omega)\alpha}{N}\frac{1}{N}\sum_{i}\EE[\|c_i^{t}-\nabla f_i(x^{t-1})\|^2].\label{eqn:nvicxnvbcx}
\end{align}
Adding $\eqref{eqn:vidnvdcxvcu}\times \frac{164\gamma (1+\omega)}{S}$ to \eqref{eqn:nvicxnvbcx}, we have
\begin{align}
    &\EE[f(x^{t+1})]+\frac{10\gamma N}{\alpha S}\EE[\|c^{t+1}-\nabla f(x^t)\|^2]+\frac{164\gamma (1+\omega)}{S}\frac{1}{N}\sum_{i}\EE[\|c_i^{t+1}-\nabla f_i(x^{t})\|^2]\\
      \leq &\EE[f(x^t)]+\gamma\left(\frac{10 N}{\alpha S}-1\right)\EE[\|c^{t}-\nabla f(x^{t-1})\|^2]\\
      &\quad +\left(\frac{164\gamma (1+\omega)}{S}-\frac{\gamma(1+\omega)}{N}\right)\frac{1}{N}\sum_{i}\EE[\|c_i^{t}-\nabla f_i(x^{t-1})\|^2]\\
      &\quad -\frac{\gamma}{2}\EE[\|\nabla f(x^t)\|^2]-\left(\frac{1}{2\gamma}-\frac{L}{2}\right)\EE[\|x^{t+1}-x^t\|^2]\\
      &\quad +\gamma L^2\left(\frac{44(1+\omega)\alpha}{S}+\frac{24N^2}{\alpha^2S^2}+\frac{656(1+\omega)N}{\alpha S^2}\right)\EE[\|x^t-x^{t-1}\|^2]\\
      &\quad +\gamma(1+\omega)\left(\frac{8\alpha^2}{S}+\frac{60\alpha}{N}+\frac{328(1+\omega)\alpha^2}{N}\right)\frac{\sigma^2}{K}+\gamma L^2\left(24+\frac{4(1+\omega)\alpha^2}{S}+\frac{522(1+\omega)\alpha}{N}\right)U^t.\label{eqn:nvicxnvbcxdsd}
\end{align}
Defining the Lyapunov function ($x^{-1}:=x^0$)
\begin{align}
\Phi^t=\EE[f(x^{t})]+\frac{10\gamma N}{\alpha S}\EE[\|c^{t}-\nabla f(x^{t-1})\|^2]+\frac{164\gamma (1+\omega)}{S}\frac{1}{N}\sum_{i}\EE[\|c_i^{t}-\nabla f_i(x^{t-1})\|^2],
\end{align}
following \eqref{eqn:nvicxnvbcxdsd}, we obtain
\begin{align}
    &\Phi^{t+1}-\Phi^t\\
      \leq &-\frac{\gamma}{2}\EE[\|\nabla f(x^t)\|^2]-\gamma\EE[\|c^{t}-\nabla f(x^{t-1})\|^2]-\frac{\gamma(1+\omega)}{N}\frac{1}{N}\sum_{i}\EE[\|c_i^{t}-\nabla f_i(x^{t-1})\|^2]\\
      &\quad -\left(\frac{1}{2\gamma}-\frac{L}{2}\right)\EE[\|x^{t+1}-x^t\|^2]+\gamma L^2\left(\frac{44(1+\omega)\alpha}{S}+\frac{24N^2}{\alpha^2S^2}+\frac{656(1+\omega)N}{\alpha S^2}\right)\EE[\|x^t-x^{t-1}\|^2]\\
      &\quad +\gamma(1+\omega)\left(\frac{8\alpha^2}{S}+\frac{60\alpha}{N}+\frac{328(1+\omega)\alpha^2}{N}\right)\frac{\sigma^2}{K}+\gamma L^2\left(24+\frac{4(1+\omega)\alpha^2}{S}+\frac{522(1+\omega)\alpha}{N}\right)U^t.\label{eqn:bvifdbvcxcu}
\end{align}
{Since $N\geq S\geq1$ and $\alpha \leq 1/(4(1+\omega))$, we have
\begin{align}
   & 9e^2K^2\eta_l^2L^2\left(24+\frac{4(1+\omega)\alpha^2}{S}+\frac{522(1+\omega)\alpha}{N}\right)\\
    \leq & 9e^2K^2\eta_l^2L^2\left(24+\frac{1}{4}+\frac{261}{2}\right)\leq 1400 e^2K^2\eta_l^2L^2\leq \frac{\alpha(1+\omega)}{N}\leq \frac{1}{4}.\label{eqn:vnidsnvz}\\
\end{align}}
Combining \eqref{eqn:vnidsnvz} with  Lemma~\ref{lem:u-bound-u}, 
% Using $9e^2K^2\eta_l^2L^2\left(24+\frac{4(1+\omega)\alpha^2}{S}+\frac{522(1+\omega)\alpha}{N}\right)\leq \frac{\alpha (1+\omega)}{N}\leq \frac{1}{4}$ and Lemma~\ref{lem:u-bound-u},
we have
\begin{align}
        &\gamma L^2\left(24+\frac{4(1+\omega)\alpha^2}{S}+\frac{522(1+\omega)\alpha}{N}\right)U^t\\
        \leq& \gamma\EE[\|c^{t}-\nabla f(x^{t-1})\|^2]+\frac{\gamma(1+\omega)}{N}\frac{1}{N}\sum_{i}\EE[\|c_i^{t}-\nabla f_i(x^{t-1})\|^2]\\
        &\quad +\frac{\gamma}{4}\EE[\|\nabla f(x^t)\|^2]+\frac{\gamma L^2}{4}\EE[\|x^t-x^{t-1}\|^2]+\frac{\gamma \alpha(1+\omega)\sigma^2}{NK}.\label{eqn:bingdcvcxu}
    \end{align}
Plugging \eqref{eqn:bingdcvcxu} into \eqref{eqn:bvifdbvcxcu}, we reach 
\begin{align}
    &\Phi^{t+1}-\Phi^t\\
      \leq & -\frac{\gamma}{4}\EE[\|\nabla f(x^t)\|^2] +\gamma(1+\omega)\left(\frac{8\alpha^2}{S}+\frac{60\alpha}{N}+\frac{328(1+\omega)\alpha^2}{N}\right)\frac{\sigma^2}{K}\\
      &-\left(\frac{1}{2\gamma}-\frac{L}{2}\right)\EE[\|x^{t+1}-x^t\|^2]+\gamma L^2\left(\frac{44(1+\omega)\alpha}{S}+\frac{25N^2}{\alpha^2S^2}+\frac{656(1+\omega)N}{\alpha S^2}\right)\EE[\|x^t-x^{t-1}\|^2].\label{eqn:bvifdbvdsdscxcu}
\end{align}
Due to the choice of $\gamma=\eta_g \eta_l K$ and $\alpha \leq \frac{1}{4(\omega+1)}$, it holds that
\begin{align}
    \frac{1}{2\gamma}\geq \frac{L}{2}+\gamma L^2\left(\frac{44(1+\omega)\alpha}{S}+\frac{25N^2}{\alpha^2S^2}+\frac{656(1+\omega)N}{\alpha S^2}\right).
\end{align}
Averaging \eqref{eqn:bvifdbvdsdscxcu} over $k$ and noting $\|x^0-x^{-1}\|^2=0$, $\alpha =O((1+\omega)^{-1})$,
we obtain 
\begin{align}
    \frac{1}{T}\sum_{t=0}^{T-1}\EE[\|\nabla f(x^t)\|^2]
    \lesssim &\frac{\Phi^0-\Phi^{T}}{\gamma T}+(1+\omega)\left(\frac{\alpha^2}{S}+\frac{\alpha}{N}+\frac{(1+\omega)\alpha^2}{N}\right)\frac{\sigma^2}{K}\\
    \lesssim &\frac{\Phi^0-\Phi^{T}}{\gamma T}+(1+\omega)\left(\frac{\alpha^2}{S}+\frac{\alpha}{N}\right)\frac{\sigma^2}{K}.\label{eqn:ngidnfgsdcx} 
\end{align}
By the definition of $\Phi^t$, it holds that
\begin{align}
    \frac{\Phi^0-\Phi^{T}}{\gamma T}\lesssim &\frac{L\Delta}{\gamma T}+\frac{ N}{\alpha S}\frac{\EE[\|c^{0}-\nabla f(x^{0})\|^2]}{T}+\frac{(1+\omega)}{S}\frac{\frac{1}{N}\sum_i\EE[\|c_i^{0}-\nabla f_i(x^{0})\|^2]}{T}\\
    \lesssim&\frac{L\Delta}{T}\frac{N}{\alpha S}+\frac{\sigma^2}{\alpha SBT}
\end{align}
where we use the choice of $\gamma$, $\alpha$, and the initialization of $\{c_i^0\}_{i\in[N]}$ and $c^0$ in the second inequality.
Due to the choice of $B$, we have  $\frac{\sigma^2}{\alpha SBT}\lesssim \frac{L\Delta}{T}\frac{N}{\alpha S}$ and thus
\begin{align}
    \frac{1}{T}\sum_{t=0}^{T-1}\EE[\|\nabla f(x^t)\|^2]
    \lesssim &\frac{L\Delta}{T}\frac{N}{\alpha S}+(1+\omega)\left(\frac{\alpha^2}{S}+\frac{\alpha}{N}\right)\frac{\sigma^2}{K}.
\end{align}
Plugging the choice of $\alpha$ into \eqref{eqn:ngidnfgsdcx} completes the proof.
\end{proof}

\newpage
\section{Proof of \scafcom}\label{app:scafcom}
In this section, we prove the convergence result of \scafcom with biased compression, where we  additionally define $x^{-1}:=x^0$. {We thus have $\EE[\|x^t-x^{t-1}\|^2]=0$ for $t=0$. Note that $x^{-1}$ is defined for the purpose of notation and is not utilized in our algorithms.}

\begin{lemma}[\sc Descent lemma]\label{lem:descent}
  Under Assumptions~\ref{asp:smooth} and~\ref{asp:gd-noise}, it holds for all $t\geq 0$ and $\gamma>0$ that
  \begin{align}
      &\EE[f(x^{t+1})]\\
      \leq &\EE[f(x^t)]-\frac{\gamma}{2}\EE[\|\nabla f(x^t)\|^2]-\left(\frac{1}{2\gamma}-\frac{L}{2}\right)\EE[\|x^{t+1}-x^t\|^2]+16\gamma L^2\EE[\|x^t-x^{t-1}\|^2]+\frac{8\gamma\beta^2\sigma^2}{K}+9\gamma\beta^2L^2U^t\\
      &+\gamma \left(4\EE[\|v^t-\nabla f(x^{t-1})\|^2]+\frac{12\beta^2}{N}\sum_{i}\EE[\|v_i^{t}-\nabla f_i(x^{t-1})\|^2]+\frac{12}{N}\sum_{i}\EE[\|v_i^{t}-c_i^t\|^2]\right).\label{eqn:descent}
  \end{align}
  % Additionally for $t=0$, we have 
  % \begin{align}
  %     \EE[f(x^{1})]\leq &\EE[f(x^0)]-\frac{\gamma}{2}\EE[\|\nabla f(x^0)\|^2]-\left(\frac{1}{2\gamma}-\frac{L}{2}\right)\EE[\|x^{1}-x^0\|^2]+\left(S^{-1}+q^2\right)\frac{4\gamma\beta^2\sigma^2}{K}+9\gamma\beta^2L^2U^0\\
  %     &+\gamma \left(4\EE[\|v^0-\nabla f(x^{0})\|^2]+\beta^2\left(S^{-1}+q^2\right)\frac{6}{N}\sum_{i}\EE[\|v_i^{0}-\nabla f_i(x^{0})\|^2]+\left(S^{-1}+q^2\right)\frac{6}{N}\sum_{i}\EE[\|v_i^{0}-c_i^0\|^2]\right).
  % \end{align}
\end{lemma}
\begin{proof}
    By Lemma 2 of~\citet{li2021page}, we have 
    \begin{equation}
      f(x^{t+1})\leq f(x^t)-\frac{\gamma}{2}\|\nabla f(x^t)\|^2-\left(\frac{1}{2\gamma}-\frac{L}{2}\right)\|x^{t+1}-x^t\|^2+\frac{\gamma}{2}\|\tilde d^{t+1}-\nabla f(x^t)\|^2
  \end{equation}
  where $\tilde d^{t+1}=\frac{1}{S}\sum_{i\in\cS^t}\tilde \delta_i^t+c^t$. Letting $d^{t+1}\triangleq \frac{1}{S}\sum_{i\in\cS^t} \delta_i^t+c^t$, we further have
  \begin{equation}
      f(x^{t+1})\leq f(x^t)-\frac{\gamma}{2}\|\nabla f(x^t)\|^2-\left(\frac{1}{2\gamma}-\frac{L}{2}\right)\|x^{t+1}-x^t\|^2+\gamma\|\tilde d^{t+1}-d^{t+1}\|^2+\gamma\|d^{t+1}-\nabla f(x^{t})\|^2.\label{eqn:bnifnbsdgsef0}
  \end{equation}
    For $\|d^{t+1}-\nabla f(x^{t})\|^2$, using Lemma~\ref{lem:bias-var} and the fact that $c^t\equiv \frac{1}{N}\sum_{i}c_i^t$ and $d^{t+1}=\frac{1}{S}\sum_{i\in\cS^t} (u_i^{t+1}-c_i^t)+c^t$, we have 
  \begin{align}
      &\EE[\|d^{t+1}-\nabla f(x^{t})\|^2]\\
      =&\EE\left[\left\|c^t+\frac{1}{S}\sum_{i\in\cS^t}(u_i^{t+1}-c_i^t)-\nabla f(x^{t})\right\|^2\right]\\
      \leq &\EE[\|u^{t+1}-\nabla f(x^{t})\|^2]+ \frac{1}{SN}\sum_{i}\EE[\|u_i^{t+1}-c_i^t\|^2]\\
      = &\EE[\|(1-\beta)(v^t-\nabla f(x^t))+\beta (g^t-\nabla f(x^{t}))\|^2]+ \frac{1}{SN}\sum_{i}\EE[\|v_i^{t}+\beta(g_i^t-v_i^t)-c_i^t\|^2].\label{eqn:vninvsd0}
  \end{align}
    Using \eqref{eqn:key-bound1} and Assumption~\ref{asp:smooth}, we have 
  \begin{align}
      &\EE[\|(1-\beta)(v^t-\nabla f(x^t))+\beta (g^t-\nabla f(x^{t})\|^2]\\
      \leq & 2\EE[\|v^t-\nabla f(x^t)\|^2]+3\beta^2L^2U^t+\frac{2\beta^2\sigma^2}{NK}\\
      \leq &4\EE[\|v^t-\nabla f(x^{t-1})\|^2]+4L^2\EE[\|x^t-x^{t-1}\|^2]+3\beta^2L^2U^t+\frac{2\beta^2\sigma^2}{NK}.\label{eqn:vninvsd1}
  \end{align}
    Similarly,  using \eqref{eqn:key-bound2} and Assumption~\ref{asp:smooth}, we have 
    \begin{align}
        &\frac{1}{SN}\sum_{i}\EE[\|v_i^{t}+\beta(g_i^t-v_i^t)-c_i^t\|^2]\\
        = &\frac{1}{SN}\sum_{i}\EE[\|v_i^{t}+\beta(\nabla f_i(x^t)-\nabla f_i(x^{t-1}))+\beta(\nabla f_i(x^{t-1})-v_i^t)+\beta(g_i^t-\nabla f_i(x^t))-c_i^t\|^2]\\
        \leq &\frac{2}{SN}\sum_{i}\EE[\|v_i^{t}+\beta(\nabla f_i(x^t)-\nabla f_i(x^{t-1}))+\beta(\nabla f_i(x^{t-1})-v_i^t)-c_i^t\|^2]+\frac{3\beta^2L^2U^t}{S}+\frac{2\beta^2\sigma^2}{SK}\\
        \leq &\frac{6}{SN}\sum_{i}\left(\EE[\|v_i^{t}-c_i^t\|^2]+\beta^2\EE[\|v_i^t-\nabla f_i(x^{t-1})\|^2]+\beta^2L^2\EE[\|x^t-x^{t-1}\|^2]\right)+\frac{3\beta^2L^2U^t}{S}+\frac{2\beta^2\sigma^2}{SK}.\label{eqn:vninvsd2}
    \end{align}
    Plugging \eqref{eqn:vninvsd1} and \eqref{eqn:vninvsd2} into \eqref{eqn:vninvsd0}, we obtain
    \begin{align}
        \EE[\|d^{t+1}-\nabla f(x^{t})\|^2]\leq &4\EE[\|v^t-\nabla f(x^{t-1})\|^2]+\frac{6\beta^2}{S}\frac{1}{N}\sum_{i}\EE[\|v_i^{t}-\nabla f_i(x^{t-1})\|^2]+\frac{6}{SN}\sum_i\EE[\|v_i^t-c_i^t\|^2]\\
        &\quad + 3(1+S^{-1})\beta^2 L^2U^t+\left(4+\frac{6\beta^2}{S}\right)L^2\EE[\|x^t-x^{t-1}\|^2]+(N^{-1}+S^{-1})\frac{2\beta^2\sigma^2}{K}.\label{eqn:bnifnbsdgsef1}
    \end{align}
    For $\|\tilde d^{t+1}-d^{t+1}\|^2$, using Young's inequality and Definition~\ref{def:contract}, we have 
    \begin{align}
        \EE[\|\tilde d^{t+1}-d^{t+1}\|^2]=&\EE\left[\left\|\frac{1}{S}\sum_{i\in\cS^t}\cC_i(u_i^{t+1}-c_i^t)-(u_i^{t+1}-c_i^t)\right\|^2\right]\\
        \leq &\frac{q^2}{S}\EE\left[\sum_{i\in\cS^t}\|u_i^{t+1}-c_i^t\|^2\right]=\frac{q^2}{N}\sum_{i}\EE[\|u_i^{t+1}-c_i^t\|^2]\\
        =&\frac{q^2}{N}\sum_{i}\EE[\|v_i^{t}+\beta(g_i^t-v_i^t)-c_i^t\|^2].
    \end{align}
    Then applying the same relaxation in \eqref{eqn:vninvsd2}, we obtain
    \begin{align}
        \EE[\|\tilde d^{t+1}-d^{t+1}\|^2]
        \leq &\frac{6q^2}{N}\sum_{i}\left(\EE[\|v_i^{t}-c_i^t\|^2]+\beta^2\EE[\|v_i^t-\nabla f_i(x^{t-1})\|^2]+\beta^2L^2\EE[\|x^t-x^{t-1}\|^2]\right)\\
        &\quad +{3\beta^2q^2L^2U^t}+\frac{2\beta^2q^2\sigma^2}{K}.\label{eqn:bnifnbsdgsef2}
    \end{align}
    Plugging \eqref{eqn:bnifnbsdgsef1} and \eqref{eqn:bnifnbsdgsef2} into \eqref{eqn:bnifnbsdgsef0} and noting $N^{-1}\leq S^{-1}\leq 1$, $q^2\leq 1$, we complete the proof.
\end{proof}

Given Lemma~\ref{lem:descent}, the rest is to bound $\|v^t-\nabla f(x^{t-1})\|^2$, $\|v_i^t-\nabla f_i(x^{t-1})\|^2$, and $\|v_i^t-c_i^t\|^2$.
\begin{lemma}\label{lem:v-f-bound}
    Under Assumptions~\ref{asp:smooth} and~\ref{asp:gd-noise}, it holds for all $t\geq 0$ that 
    \begin{align}
        \EE[\|v^{t+1}-\nabla f(x^{t})\|^2]\leq& \left(1-\frac{S\beta}{2N}\right)\EE[\|v^{t}-\nabla f(x^{t-1})\|^2]+\frac{4\beta^2S}{N^2}\frac{1}{N}\sum_i\EE[\|v_i^t-\nabla f_i(x^{t-1})\|^2]\\
        &\quad +\frac{6NL^2}{S\beta}\EE[\|x^{t}-x^{t-1}\|^2]+\frac{6\beta SL^2}{N}U^t+\frac{6\beta^2S\sigma^2}{N^2K}.\label{eqn:vhidhfsv cx}
    \end{align}
\end{lemma}
\begin{proof}
Using \eqref{eqn:key-update-scafcom} and Lemma~\ref{lem:par_sample}, we have
    \begin{align}
        \EE[\|v^{t+1}-\nabla f(x^{t})\|^2]=& \EE\left[\left\|\frac{1}{S}\sum_{i\in\cS^{t}}\frac{S\beta}{N}( {g}_i^{t}-v_i^{t})+v^{t}-\nabla f(x^{t})\right\|^2\right]\\
        \leq &\EE\left[\left\|\frac{S\beta}{N}({g}^{t}-v^{t})+v^{t}-\nabla f(x^{t})\right\|^2\right]+\frac{\beta^2 S }{N^2}\frac{1}{N}\sum_{i}\EE[\left\|g_i^{t}-v_i^t\right\|^2]\\
        = &\EE\left[\left\|\left(1-\frac{S\beta}{N}\right)(v^{t}-\nabla f(x^{t}))+\frac{S\beta}{N}({g}^{t}-\nabla f(x^t))\right\|^2\right]+\frac{\beta^2 S }{N^2}\frac{1}{N}\sum_{i}\EE[\left\|g_i^{t}-v_i^t\right\|^2].\label{eqn:nvosdnvsd1}
    \end{align}
    Using \eqref{eqn:key-bound1}, Young's inequality, and Assumption~\ref{asp:smooth}, we further have
    \begin{align}
        &\EE\left[\left\|\left(1-\frac{S\beta}{N}\right)(v^{t}-\nabla f(x^{t}))+\frac{S\beta}{N}({g}^{t}-\nabla f(x^t))\right\|^2\right]\\
        \leq &\left(1-\frac{S\beta}{N}\right)\EE[\|v^{t}-\nabla f(x^{t})\|^2]+\frac{2\beta SL^2}{N}U^t+\frac{2\beta^2S^2\sigma^2}{N^3K}\\
        \leq &\left(1-\frac{S\beta}{2N}\right)\EE[\|v^{t}-\nabla f(x^{t-1})\|^2]+\frac{2NL^2}{S\beta}\EE[\|x^t-x^{t-1}\|^2]+\frac{2\beta SL^2}{N}U^t+\frac{2\beta^2S^2\sigma^2}{N^3K}.\label{eqn:nvosdnvsd2}
    \end{align}
    Using Young's inequality and Assumption~\ref{asp:smooth}, we can obtain 
    \begin{align}
        \frac{\beta^2 S }{N^2}\frac{1}{N}\sum_{i}\EE[\left\|g_i^{t}-v_i^t\right\|^2]\leq & \frac{\beta^2 S }{N^2}\frac{1}{N}\sum_{i}\left(2\EE[\|\nabla f_i(x^t)-v_i^t\|^2]+2\EE[\|g_i^{t}-\nabla f_i(x^t)\|^2]\right)\\
        \leq &\frac{\beta^2 S }{N^2}\frac{1}{N}\sum_{i}\left(4\EE[\|v_i^t-\nabla f_i(x^{t-1})\|^2]+4L^2\EE[\|x^t-x^{t-1}\|^2]+2\EE[\|g_i^{t}-\nabla f_i(x^t)\|^2]\right).\label{eqn:nvosdnvsd3-scaf}
    \end{align}
    Using Lemma~\ref{lem:bias-var} and Assumption~\ref{asp:smooth}, we have 
    \begin{align}
        \frac{1}{N}\sum_{i}\EE[\|g_i^{t}-\nabla f_i(x^t)\|^2]\leq& \frac{2}{N}\sum_{i}\EE\left[\left\|\frac{1}{K}\sum_{k}\nabla f_i(y_i^{t,k})-\nabla f_i(x^t)\right\|^2\right] +\frac{2\sigma^2}{K}\leq 2L^2U^t +\frac{2\sigma^2}{K}.\label{eqn:vidngsdgs}
    \end{align}
    Plugging \eqref{eqn:vidngsdgs} into \eqref{eqn:nvosdnvsd3-scaf}, we reach
    \begin{align}
        \frac{\beta^2 S }{N^2}\frac{1}{N}\sum_{i}\EE[\left\|g_i^{t}-v_i^t\right\|^2]
        \leq &\frac{\beta^2 S }{N^2}\frac{1}{N}\sum_{i}\left(4\EE[\|v_i^t-\nabla f_i(x^{t-1})\|^2]+4L^2\EE[\|x^t-x^{t-1}\|^2]+4L^2U^t+\frac{4\sigma^2}{K}\right).\label{eqn:nvosdnvsd4}
    \end{align}
    Combining \eqref{eqn:nvosdnvsd1}, \eqref{eqn:nvosdnvsd2}, \eqref{eqn:nvosdnvsd4} together and using $\frac{\beta^2S^2\sigma^2}{N^3K}\leq \frac{\beta^2S\sigma^2}{N^2K}$, $\frac{\beta^2SL^2}{N^2}\leq \frac{NL^2}{S\beta}$, $\frac{\beta^2SL^2}{N^2}\leq \frac{\beta SL^2}{N}$ completes the proof.
\end{proof}

\begin{lemma}\label{lem:vi-fi-bound-scaf}
    Under Assumptions~\ref{asp:smooth} and~\ref{asp:gd-noise}, it holds for all $t\geq 0$ that
    \begin{align}
       \frac{1}{N}\sum_{i}\EE[\|v_i^{t+1}-\nabla f_i(x^{t})\|^2]
        \leq &\left(1-\frac{S\beta}{2N}\right) \frac{1}{N}\sum_{i}\EE[\|v_i^{t}-\nabla f_i(x^{t-1})\|^2]+\frac{2NL^2}{S\beta}\EE[\|x^t-x^{t-1}\|^2]\\
        &\quad +\frac{2\beta L^2S}{N} U^t+\frac{2\beta^2S\sigma^2}{NK}.\label{eqn:vidnvdcxvc}
   \end{align}
\end{lemma}
\begin{proof}
Using \eqref{eqn:key-update-scafcom}, we have
    \begin{align}
        \frac{1}{N}\sum_{i}\EE[\|v_i^{t+1}-\nabla f_i(x^{t})\|^2]=&\frac{1}{N}\sum_{i}\left(\left(1-\frac{S}{N}\right) \EE[\|v_i^{t}-\nabla f_i(x^{t})\|^2]+\frac{S}{N}\EE\left[\left\|\beta( {g}_i^{t}-v_i^{t})+v_i^{t}-\nabla f_i(x^{t})\right\|^2\right]\right)\\
        \leq &\left(1-\frac{S\beta}{N}\right) \frac{1}{N}\sum_{i}\EE[\|v_i^{t}-\nabla f_i(x^{t})\|^2]+\frac{2\beta L^2S}{N} U^t+\frac{2\beta^2S\sigma^2}{NK}\label{eqn:nvosdnvsd11}
    \end{align}
    where the inequality due to
    \begin{align}
        \frac{1}{N}\sum_{i}\EE\left[\left\|\beta( {g}_i^{t}-v_i^{t})+v_i^{t}-\nabla f_i(x^{t})\right\|^2\right]=&\frac{1}{N}\sum_{i}\EE\left[\left\|\beta( {g}_i^{t}-\nabla f_i(x^t))+(1-\beta)(v_i^{t}-\nabla f_i(x^{t}))\right\|^2\right]\\
        \leq &\frac{1-\beta}{N}\sum_{i}\EE[\|v_i^{t}-\nabla f_i(x^{t})\|^2]+2\beta L^2 U^t+\frac{2\beta^2\sigma^2}{K}
    \end{align}
    by applying  \eqref{eqn:key-bound2}. By further using Young's inequality and Assumption~\ref{asp:smooth}, we obtain
    \begin{align}
        \frac{1}{N}\sum_{i}\EE[\|v_i^{t+1}-\nabla f_i(x^{t})\|^2]
        \leq &\left(1-\frac{S\beta}{2N}\right) \frac{1}{N}\sum_{i}\EE[\|v_i^{t}-\nabla f_i(x^{t-1})\|^2]+\frac{2NL^2}{S\beta}\EE[\|x^t-x^{t-1}\|^2]\\
        &\quad +\frac{2\beta L^2S}{N} U^t+\frac{2\beta^2S\sigma^2}{NK}.
    \end{align}
\end{proof}

\begin{lemma}\label{lem:v-c-bound}
    Under Assumptions~\ref{asp:smooth} and~\ref{asp:gd-noise}, it holds for all $t\geq 0$ that 
    \begin{align}
        \frac{1}{N}\sum_i\EE[\|v_i^{t+1}-c_i^{t+1}\|^2]\leq& \left(1-\frac{S(1-q)}{N}\right)\frac{1}{N}\sum_i\EE[\|v_i^{t}-c_i^{t}\|^2]+\frac{4\beta^2q^2S}{(1-q)N}\frac{1}{N}\sum_i\EE[\|v_i^t-\nabla f_i(x^{t-1})\|^2]\\
        &\quad +\frac{4\beta^2L^2q^2S}{(1-q)N} \EE[\|x^t-x^{t-1}\|^2]  +\frac{3\beta^2q^2SL^2}{(1-q)N}U^t+\frac{2\beta^2q^2 S\sigma^2}{NK}.\label{eqn:vidfnvsdd}
    \end{align}
    % Additionally for $t=0$, we have 
    % \begin{align}
    %     \frac{1}{N}\sum_i\EE[\|v_i^{1}-c_i^{1}\|^2]\leq& \left(1-\frac{S(1-q)}{N}\right)\frac{1}{N}\sum_i\EE[\|v_i^{0}-c_i^{0}\|^2]+\frac{4\beta^2q^2S}{(1-q)N}\frac{1}{N}\sum_i\EE[\|v_i^0-\nabla f_i(x^{0})\|^2]\\
    %     &\quad+\frac{3\beta^2q^2SL^2}{(1-q)N}U^0+\frac{2\beta^2q^2 S\sigma^2}{NK}.
    % \end{align}
\end{lemma}
\begin{proof}
   Using \eqref{eqn:key-update-scafcom} and Definition~\ref{def:contract}, we have
    \begin{align}
        \EE[\|v_i^{t+1}-c_i^{t+1}\|^2]=&\left(1-\frac{S}{N}\right)\EE[\|v_i^t-c_i^t\|^2]+\frac{S}{N}\EE[\|u_i^{t+1}-\cC_i(u_i^{t+1}-c_i^t)-c_i^t)\|^2]\\
        \leq &\left(1-\frac{S}{N}\right)\EE[\|v_i^t-c_i^t\|^2]+\frac{S}{N}\frac{q^2}{N}\sum_i \EE[\|u_i^{t+1}-c_i^t\|^2]\\
        =&\left(1-\frac{S}{N}\right)\EE[\|v_i^t-c_i^t\|^2]+\frac{S}{N}\frac{q^2}{N}\sum_i \EE[\|v_i^{t}+\beta(g_i^t-v_i^t)-c_i^t\|^2]\label{eqn:nvosdnvsdvcc}
    \end{align}
    where $u_i^{t+1}\triangleq v_i^t+\beta(g_i^t-v_i^t)$.
    Using Lemma~\ref{lem:bias-var} and Assumption~\ref{asp:smooth}, we have 
    \begin{align}
        &\frac{q^2}{N}\sum_i \EE[\|v_i^{t}+\beta(g_i^t-v_i^t)-c_i^t\|^2]=\frac{q^2}{N}\sum_i \EE[\|v_i^{t}+\beta(\nabla f_i(x^t)-v_i^t)+\beta(g_i^t-\nabla f_i(x^t))-c_i^t\|^2]\\
        =&\frac{q^2}{N}\sum_i\Bigg(\EE[\|v_i^{t}-c_i^t+\beta(\nabla f_i(x^t)-v_i^t)\|^2]+2\beta \EE\left[\left\langle v_i^t-c_i^t+\beta(\nabla f_i(x^t)-v_i^t),\frac{1}{K}\sum_{k}\nabla f_i(y_i^{t,k})-\nabla f_i(x^t)\right\rangle\right]\\
        &\qquad+\beta^2\EE\left[\left\|\frac{1}{K}\sum_{k}\nabla F(y_i^{t,k};\xi_i^{t,k})-\nabla f_i(x^t)\right\|^2\right]\Bigg)\\
        \leq &\frac{q^2}{N}\sum_i\Bigg(\EE[\|v_i^{t}-c_i^t+\beta(\nabla f_i(x^t)-v_i^t)\|^2]+2\beta \EE\left[\left\langle v_i^t-c_i^t+\beta(\nabla f_i(x^t)-v_i^t),\frac{1}{K}\sum_{k}\nabla f_i(y_i^{t,k})-\nabla f_i(x^t)\right\rangle\right]\\
        &\qquad+2\beta^2\EE\left[\left\|\frac{1}{K}\sum_{k}\nabla f_i(y_i^{t,k})-\nabla f_i(x^t)\right\|^2\right]+\frac{2\beta^2 \sigma^2}{K}\Bigg)\\
        \leq &\frac{q^2}{N}\sum_i\EE\left[\left\|v_i^{t}-c_i^t+\beta(\nabla f_i(x^t)-v_i^t)+ \frac{\beta}{K}\sum_{k}(\nabla f_i(y_i^{t,k})-\nabla f_i(x^t))\right\|^2\right]+\beta^2q^2L^2U^t+\frac{2\beta^2q^2 \sigma^2}{K}.
    \end{align}
    By further using Sedrakyan's inequality and Assumption~\ref{asp:smooth}, we obtain
    \begin{align}
        &\frac{q^2}{N}\sum_i \EE[\|v_i^{t}+\beta(g_i^t-v_i^t)-c_i^t\|^2]\\
        % \leq &\frac{q^2}{N}\sum_i\EE\left[\left\|v_i^{t}-c_i^t+\beta(\nabla f_i(x^t)-v_i^t)+ \frac{\beta}{K}\sum_{k}(\nabla f_i(y_i^{t,k})-\nabla f_i(x^t))\right\|^2\right]+\beta^2q^2L^2U^t+\frac{2\beta^2q^2 \sigma^2}{K}\\
        \leq &\frac{q^2}{N}\sum_i\left(\frac{1}{q}\EE[\|v_i^{t}-c_i^t\|^2]+\frac{2\beta^2}{1-q}\EE[\|\nabla f_i(x^t)-v_i^t\|^2]+\frac{2\beta^2}{1-q} \EE\left[\left\|\frac{1}{K}\sum_{k}\nabla f_i(y_i^{t,k})-\nabla f_i(x^t)\right\|^2\right]\right)\\
        &\quad +\beta^2q^2L^2U^t+\frac{2\beta^2q^2 \sigma^2}{K}\\
        \leq &\frac{q^2}{N}\sum_i\Bigg(\frac{1}{q}\EE[\|v_i^{t}-c_i^t\|^2]+\frac{4\beta^2}{1-q}\EE[\|v_i^t-\nabla f_i(x^{t-1})\|^2]+\frac{4\beta^2L^2}{1-q}\EE[\|x^t-x^{t-1}\|^2]\\
        &\qquad\qquad  +\frac{2\beta^2L^2}{1-q} \frac{1}{K}\sum_i\EE[\|y_i^{t,k}-x^t\|^2]\Bigg)+\beta^2q^2L^2U^t+\frac{2\beta^2q^2 \sigma^2}{K}\\
        \leq &\frac{q}{N}\sum_i\EE[\|v_i^{t}-c_i^t\|^2]+\frac{4\beta^2q^2}{1-q}\frac{1}{N}\sum_i\EE[\|\nabla f_i(x^{t-1})-v_i^t\|^2]+\frac{4\beta^2L^2q^2}{1-q} \EE[\|x^t-x^{t-1}\|^2]    \\
        &\quad +\left(1+\frac{2}{1-q}\right)\beta^2q^2L^2U^t+\frac{2\beta^2q^2 \sigma^2}{K}.\label{eqn:vndisnvsd}
    \end{align}
    % Consequently, noting $1\leq 1/(1-q)$, it holds that 
    % \begin{align}
    %     \frac{1}{N}\sum_i\EE[\|v_i^{t+1}-c_i^{t+1}\|^2]\leq& \left(1-\frac{S(1-q)}{N}\right)\frac{1}{N}\sum_i\EE[\|v_i^{t}-c_i^{t}\|^2]+\frac{4\beta^2q^2S}{(1-q)N}\frac{1}{N}\sum_i\EE[\|\nabla f_i(x^{t-1})-v_i^t\|^2]\\
    %     &\quad +\frac{4\beta^2L^2q^2S}{(1-q)N} \EE[\|x^t-x^{t-1}\|^2]  +\frac{3\beta^2q^2SL^2}{(1-q)N}U^t+\frac{2\beta^2q^2 S\sigma^2}{NK}
    % \end{align}
    By combinining \eqref{eqn:vndisnvsd} with \eqref{eqn:nvosdnvsdvcc} and using $1\leq 1/(1-q)$, we finish the proof.
\end{proof}

\begin{lemma}\label{lem:u-bound}
    Under Assumptions~\ref{asp:smooth} and~\ref{asp:gd-noise},  it holds for any $t\geq 0$ and $\eta_l KL\leq \frac{1}{2}$ that 
    \begin{align}\label{eqn:Lvnisbvvcb-scaf}
        U^t\leq \frac{9e^2K^2\eta_l^2}{N}\sum_i \left(\EE[\|c_i^t-v_i^t\|^2]+\EE[\|v_i^t-\nabla f_i(x^{t-1})\|^2]+L^2\EE[\|x^t-x^{t-1}\|^2]+\EE[\|\nabla f(x^t)\|^2]\right)+e^2K\eta_l^2\sigma^2.
    \end{align}
    % Additionally for $t=0$, we have 
    % \begin{align}
    %     U^0\leq \frac{9e^2K^2\eta_l^2}{N}\sum_i \left(\EE[\|c_i^0-v_i^0\|^2]+\EE[\|v_i^t-\nabla f_i(x^{0})\|^2]+\EE[\|\nabla f(x^0)\|^2]\right).
    % \end{align}
\end{lemma}
\begin{proof}
    When $K=1$, $U^t=0$ trivially for all $t\geq 0$ so we consider $K\geq 2$ below.
    Using Young's inequality, we have 
    \begin{align}
        \EE[\|y_i^{t,k+1}-x^t\|^2]=&\EE[\|y_i^{t,k}-\eta_l(g_i^{t,k}-c_i^t+c^t)-x^t\|^2]\\
        \leq &\EE[\|y_i^{t,k}-\eta_l(\nabla f(y_i^{t,k})-c_i^t+c^t)-x^t\|^2]+\eta_l^2\sigma^2\\
        \leq &\left(1+\frac{1}{K-1}\right)\EE[\|y_i^{t,k}-x^t\|^2]+K\eta_l^2\EE[\|\nabla f(y_i^{t,k})-c_i^t+c^t\|^2]+\eta_l^2\sigma^2.
    \end{align}
    By further using Young's inequality and Assumption~\ref{asp:smooth}, we obtain
    \begin{align}
       &K\eta_l^2\EE[\|\nabla f(y_i^{t,k})-c_i^t+c^t\|^2]\\
       =&K\eta_l^2\EE[\|\nabla f(y_i^{t,k})-\nabla f_i(x^t)-(c_i^t-\nabla f_i(x^t))+c^t-\nabla f(x^t)+\nabla f(x^t)\|^2]\\
        \leq & 3K\eta_l^2L^2\EE[\|y_i^{t,k}-x^t\|^2]+3K\eta_l^2\EE[\|c_i^t-\nabla f_i(x^t)-c^t+\nabla f(x^t)\|^2]+3K\eta_l^2\EE[\|\nabla f(x^t)\|^2].
    \end{align}
    Using Young's inequality, we have
    \begin{align}
        &\frac{3K\eta_l^2}{N}\sum_i \EE[\|c_i^t-\nabla f_i(x^t)-c^t+\nabla f(x^t)\|^2]\\
        \leq &\frac{3K\eta_l^2}{N}\sum_i \EE[\|c_i^t-\nabla f_i(x^t)\|^2]\\
        \leq &\frac{9K\eta_l^2}{N}\sum_i \left(\EE[\|c_i^t-v_i^t\|^2]+\EE[\|v_i^t-\nabla f_i(x^{t-1})\|^2]+L^2\EE[\|x^t-x^{t-1}\|^2]\right)
    \end{align}
    By combinining the above inequalities together, we have 
    \begin{align}
        &\frac{1}{N}\sum_i\EE[\|y_i^{t,k+1}-x^t\|^2]\\
        \leq &\left(1+\frac{1}{K-1}+3K\eta_l^2L^2\right)\frac{1}{N}\sum_i\EE[\|y_i^{t,k}-x^t\|^2]+\eta_l^2\sigma^2\\
        &\quad +\frac{9K\eta_l^2}{N}\sum_i \left(\EE[\|c_i^t-v_i^t\|^2]+\EE[\|v_i^t-\nabla f_i(x^{t-1})\|^2]+L^2\EE[\|x^t-x^{t-1}\|^2]+\EE[\|\nabla f(x^t)\|^2]\right)\\
        \leq & \cdots\leq \sum_{\ell=0}^k\left(1+\frac{1}{K-1}+3K\eta_l^2L^2\right)^{\ell}\\\
        &\qquad \times \left(\frac{9K\eta_l^2}{N}\sum_i \left(\EE[\|c_i^t-v_i^t\|^2]+\EE[\|v_i^t-\nabla f_i(x^{t-1})\|^2]+L^2\EE[\|x^t-x^{t-1}\|^2]+\EE[\|\nabla f(x^t)\|^2]\right)+\eta_l^2\sigma^2\right)\\
        \leq &  \sum_{\ell=0}^k\left(1+\frac{2}{K-1}\right)^{\ell}\\\
        &\qquad \times \left(\frac{9K\eta_l^2}{N}\sum_i \left(\EE[\|c_i^t-v_i^t\|^2]+\EE[\|v_i^t-\nabla f_i(x^{t-1})\|^2]+L^2\EE[\|x^t-x^{t-1}\|^2]+\EE[\|\nabla f(x^t)\|^2]\right)+\eta_l^2\sigma^2\right),\label{eqn:vndnvdsfsfd}
    \end{align}
    where we use $\eta_l KL\leq \frac{1}{2}$ so that $3K\eta_l^2L^2\leq \frac{1}{K-1}$ in the last inequality.
    Iterating and averaging \eqref{eqn:vndnvdsfsfd} over $k=0,\dots,K-1$, we obtain
    \begin{align}
        U^t\leq &\frac{1}{K}\sum_k \sum_{\ell=0}^{k-1} \left(1+\frac{2}{K-1}\right)^{\ell}\\
        &\qquad  \times \left(\frac{9K\eta_l^2}{N}\sum_i \left(\EE[\|c_i^t-v_i^t\|^2]+\EE[\|v_i^t-\nabla f_i(x^{t-1})\|^2]+L^2\EE[\|x^t-x^{t-1}\|^2]+\EE[\|\nabla f(x^t)\|^2]\right)+\eta_l^2\sigma^2\right)\\
        \leq & \sum_{\ell=0}^{K-2} \left(1+\frac{2}{K-1}\right)^{K-1}\\
        &\qquad \times \left(\frac{9K\eta_l^2}{N}\sum_i \left(\EE[\|c_i^t-v_i^t\|^2]+\EE[\|v_i^t-\nabla f_i(x^{t-1})\|^2]+L^2\EE[\|x^t-x^{t-1}\|^2]+\EE[\|\nabla f(x^t)\|^2]\right)+\eta_l^2\sigma^2\right)\\
        \leq &\frac{9e^2K^2\eta_l^2}{N}\sum_i \left(\EE[\|c_i^t-v_i^t\|^2]+\EE[\|v_i^t-\nabla f_i(x^{t-1})\|^2]+L^2\EE[\|x^t-x^{t-1}\|^2]+\EE[\|\nabla f(x^t)\|^2]\right)+e^2K\eta_l^2\sigma^2
    \end{align}
    where we use the fact $\left(1+\frac{2}{K-1}\right)^2\leq e^2$ in the last inequality.
\end{proof}

\begin{theorem}
    Under Assumptions~\ref{asp:smooth} and~\ref{asp:gd-noise}, supposing clients are associated with $q^2$-contractive compressors, if we initialize $c_i^0=v_i^0=\frac{1}{B}\sum_{b=1}^B\nabla F(x^0;\xi_i^b)$, $c^0=\frac{1}{N}\sum_{i=1}^N c_i^0$ with $\{\xi_i^b\}_{b=1}^B\overset{iid}{\sim}\cD_i$ and $B\gtrsim \frac{\sigma^2}{(1-q)L\Delta}$  ($c_i^0\to \nabla f_i(x^0)$ as $B\to \infty$), set
    \begin{equation}\label{eqn:scafcom-para}
        \begin{aligned}
            &\eta_l KL\leq \sqrt{\frac{\beta(1-q)^2}{36e^2 N(189(1-q)^2+306\beta^2)}} ,\quad \eta_g \eta_lKL =\left(\frac{20N}{\beta S}+\frac{28N}{(1-q)S}\right)^{-1},\\
            &\beta =\left(1+\left(\frac{TS\sigma^2}{N^2KL\Delta}\right)^{1/2}+\left(\frac{TS\sigma^2}{NK(1-q)L\Delta}\right)^{1/3}+\left(\frac{TS\sigma^2}{NK(1-q)^2L\Delta}\right)^{1/4}\right)^{-1},
        \end{aligned}
    \end{equation}
    % $\eta_g \eta_lKL =\left(\frac{20N}{\beta S}+\frac{28N}{(1-q)S}\right)^{-1}$, $\eta_l KL\leq \sqrt{\frac{\beta(1-q)^2}{36e^2 N(189(1-q)^2+306\beta^2)}} $,  
    % \begin{equation}
    %     \beta =\left(1+\left(\frac{TS\sigma^2}{N^2KL\Delta}\right)^{1/2}+\left(\frac{TS\sigma^2}{NK(1-q)L\Delta}\right)^{1/3}+\left(\frac{TS\sigma^2}{NK(1-q)^2L\Delta}\right)^{1/4}\right)^{-1},
    % \end{equation}
    then  \scafcom converges as
    \begin{equation}
    \frac{1}{T}\sum_{t=0}^{T-1}\EE[\|\nabla f(x^t)\|^2]
    \lesssim \sqrt{\frac{L\Delta \sigma^2}{SKT}} +\left(\frac{N^2L^2\Delta^2\sigma^2}{(1-q)S^2KT^2}\right)^{1/3}+\left(\frac{N^3L^3\Delta^3\sigma^2}{(1-q)^2S^3KT^3}\right)^{1/4}+\frac{NL\Delta}{(1-q)ST}
    \end{equation}
    where $\Delta \triangleq f(x^0)-\min f(x)$.
\end{theorem}
\begin{proof}

Adding $\eqref{eqn:vhidhfsv cx}\times \frac{8\gamma N}{\beta S}+\eqref{eqn:vidfnvsdd}\times \frac{13\gamma N}{(1-q)S}$ to \eqref{eqn:descent}, we have
\begin{align}
    &\EE[f(x^{t+1})]+\frac{8\gamma N}{\beta S}\EE[\|v^{t+1}-\nabla f(x^t)\|^2]+\frac{14\gamma N}{(1-q)S}\frac{1}{N}\sum_i\EE[\|v_i^{t+1}-c_i^{t+1}\|^2]\\
      \leq &\EE[f(x^t)]+\frac{8\gamma N}{\beta S}\EE[\|v^{t}-\nabla f(x^{t-1})\|^2]+\gamma\left(\frac{13 N}{(1-q)S}-1\right)\frac{1}{N}\sum_i\EE[\|v_i^{t}-c_i^{t}\|^2]\\
      &\quad -\frac{\gamma}{2}\EE[\|\nabla f(x^t)\|^2]-\left(\frac{1}{2\gamma}-\frac{L}{2}\right)\EE[\|x^{t+1}-x^t\|^2]+\gamma L^2\left(16+\frac{48N^2}{\beta^2S^2}+\frac{52\beta^2q^2}{(1-q)^2}\right)\EE[\|x^t-x^{t-1}\|^2]\\
      &\quad +\gamma\left(8\beta^2+\frac{48\beta}{N}+\frac{26 \beta^2q^2}{1-q}\right)\frac{\sigma^2}{K}+\gamma L^2\left(9\beta^2+48 +\frac{39\beta^2q^2}{(1-q)^2}\right)U^t\\
      &\quad +\gamma\left(12\beta^2+\frac{32\beta}{N}+\frac{52\beta^2q^2}{(1-q)^2}\right)\frac{1}{N}\sum_{i}\EE[\|v_i^{t}-\nabla f_i(x^{t-1})\|^2].
\end{align}
Using $q,\beta \in[0,1]$ and  $1\leq S\leq N$  to simplify coefficients, we obtain 
\begin{align}
    &\EE[f(x^{t+1})]+\frac{8\gamma N}{\beta S}\EE[\|v^{t+1}-\nabla f(x^t)\|^2]+\frac{13\gamma N}{(1-q)S}\frac{1}{N}\sum_i\EE[\|v_i^{t+1}-c_i^{t+1}\|^2]\\
      \leq &\EE[f(x^t)]+\frac{8\gamma N}{\beta S}\EE[\|v^{t}-\nabla f(x^{t-1})\|^2]+\gamma\left(\frac{13 N}{(1-q)S}-1\right)\frac{1}{N}\sum_i\EE[\|v_i^{t}-c_i^{t}\|^2]\\
      &\quad -\frac{\gamma}{2}\EE[\|\nabla f(x^t)\|^2]-\left(\frac{1}{2\gamma}-\frac{L}{2}\right)\EE[\|x^{t+1}-x^t\|^2]+\gamma L^2\left(\frac{64N^2}{\beta^2S^2}+\frac{52\beta^2}{(1-q)^2}\right)\EE[\|x^t-x^{t-1}\|^2]\\
      &\quad +\gamma\left(\frac{48\beta}{N}+\frac{26\beta^2}{1-q}\right)\frac{\sigma^2}{K}+\gamma L^2\left(48+\frac{39\beta^2}{(1-q)^2}\right)U^t\\
      &\quad +\gamma\left(\frac{32\beta}{N}+\frac{64\beta^2}{(1-q)^2}\right)\frac{1}{N}\sum_{i}\EE[\|v_i^{t}-\nabla f_i(x^{t-1})\|^2].\label{eqn:nbifngsdgccvc}
\end{align}
Now adding $\eqref{eqn:vidnvdcxvc} \times66\gamma(\frac{1}{S}+\frac{2\beta N}{(1-q)^2S})$ to \eqref{eqn:nbifngsdgccvc} and defining the Lyapunov function ($x^{-1}:=x^0$)
\begin{align}
    \Psi^t=&\EE[f(x^{t})]+\frac{8\gamma N}{\beta S}\EE[\|v^{t}-\nabla f(x^{t-1})\|^2]\\
    &\quad +\frac{13\gamma N}{(1-q)S}\frac{1}{N}\sum_i\EE[\|v_i^{t}-c_i^{t}\|^2]+66\gamma\left(\frac{1}{S}+\frac{2\beta N}{(1-q)^2S}\right)\frac{1}{N}\sum_i\EE[\|v_i^{t}-\nabla f_i(x^{t-1})\|^2],
\end{align}
we obtain
\begin{align}
    &\Psi^{t+1}-\Psi^t\\
      \leq & -\gamma\left(\frac{1}{N}\sum_i\EE[\|v_i^{t}-c_i^{t}\|^2]+\left(\frac{\beta}{N}+\frac{2\beta^2}{(1-q)^2}\right)\frac{1}{N}\sum_i\EE[\|v_i^{t}-\nabla f_i(x^{t-1})\|^2]\right)\\
      &\quad -\frac{\gamma}{2}\EE[\|\nabla f(x^t)\|^2]-\left(\frac{1}{2\gamma}-\frac{L}{2}\right)\EE[\|x^{t+1}-x^t\|^2]\\
      &\quad +\gamma L^2\left(\frac{64N^2}{\beta^2S^2}+\frac{52\beta^2}{(1-q)^2}+132\left(\frac{N}{\beta S^2}+\frac{2N^2}{(1-q)^2S^2}\right)\right)\EE[\|x^t-x^{t-1}\|^2]\\
      &\quad +\gamma\left(\frac{32\beta}{N}+\frac{64\beta^2}{1-q}+132\left(\frac{\beta^2}{N}+\frac{2\beta^3}{(1-q)^2}\right)\right)\frac{\sigma^2}{K}+\gamma L^2\left(48 +\frac{39\beta^2}{(1-q)^2}+132\left(\frac{\beta}{N}+\frac{2\beta^2}{(1-q)^2}\right)\right)U^t.
\end{align}
Using $q,\beta \in[0,1]$ and  $1\leq S\leq N$ to simplify coefficients, we obtain
\begin{align}
    &\Psi^{t+1}-\Psi^t\\
      \leq & -\gamma\left(\frac{1}{N}\sum_i\EE[\|v_i^{t}-c_i^{t}\|^2]+\left(\frac{\beta}{N}+\frac{2\beta^2}{(1-q)^2}\right)\frac{1}{N}\sum_i\EE[\|v_i^{t}-\nabla f_i(x^{t-1})\|^2]\right)\\
      &\quad -\frac{\gamma}{2}\EE[\|\nabla f(x^t)\|^2]-\left(\frac{1}{2\gamma}-\frac{L}{2}\right)\EE[\|x^{t+1}-x^t\|^2]+\gamma L^2\left(\frac{196N^2}{\beta^2S^2}+\frac{52\beta^2}{(1-q)^2}+\frac{264N^2}{(1-q)^2S^2}\right)\EE[\|x^t-x^{t-1}\|^2]\\
      &\quad +\gamma\left(\frac{164\beta}{N}+\frac{64\beta^2}{1-q}+\frac{264\beta^3}{(1-q)^2}\right)\frac{\sigma^2}{K}+\gamma L^2\left(180+\frac{303\beta^2}{(1-q)^2}\right)U^t.\label{eqn:bvifdbvcxc}
\end{align}
Using $9e^2K^2\eta_l^2L^2\left(180+\frac{303\beta^2}{(1-q)^2}\right)\leq \frac{\beta}{4N}$ and Lemma~\ref{lem:u-bound}, we have
\begin{align}
        \gamma L^2\left(180+\frac{303\beta^2}{(1-q)^2}\right)U^t\leq& \gamma\left( \frac{1}{N}\sum_i \EE[\|c_i^t-v_i^t\|^2]+\left(\frac{\beta}{N}+\frac{2\beta^2}{(1-q)^2}\right)\frac{1}{N}\sum_i\EE[\|v_i^{t}-\nabla f_i(x^{t-1})\|^2]\right)\\
        &\quad +\frac{\gamma}{4}\EE[\|\nabla f(x^t)\|^2]+\frac{\gamma L^2}{4}\EE[\|x^t-x^{t-1}\|^2]+\frac{\gamma \beta\sigma^2}{NK}.\label{eqn:bingdcvcx}
    \end{align}
Plugging \eqref{eqn:bingdcvcx} into \eqref{eqn:bvifdbvcxc}, we reach 
\begin{align}
    &\Psi^{t+1}-\Psi^t\\
      \leq & -\frac{\gamma}{4}\EE[\|\nabla f(x^t)\|^2]-\left(\frac{1}{2\gamma}-\frac{L}{2}\right)\EE[\|x^{t+1}-x^t\|^2]+\gamma L^2\left(\frac{197N^2}{\beta^2S^2}+\frac{316N^2}{(1-q)^2S^2}\right)\EE[\|x^t-x^{t-1}\|^2]\\
      &\quad +\gamma\left(\frac{165\beta}{N}+\frac{64\beta^2}{1-q}+\frac{264\beta^3}{(1-q)^2}\right)\frac{\sigma^2}{K}.\label{eqn:bvifdbvdsdscxc}
\end{align}
Due to the choice of $\gamma=\eta_g \eta_l K$, it holds that
\begin{align}
    \frac{1}{2\gamma}\geq \frac{L}{2}+\gamma L^2\left(\frac{197N^2}{\beta^2S^2}+\frac{264\beta^3}{(1-q)^2}\right),
\end{align}
Averaging \eqref{eqn:bvifdbvdsdscxc} over $k$ and noting $\|x^0-x^{-1}\|^2=0$,
we obtain 
\begin{align}\label{eqn:fgndasifna} 
    \frac{1}{T}\sum_{t=0}^{T-1}\EE[\|\nabla f(x^t)\|^2]
    \lesssim  &\frac{\Psi^0-\Psi^{T+1}}{\gamma T}+\left(\frac{\beta}{N}+\frac{\beta^2}{1-q}+\frac{\beta^3}{(1-q)^2}\right)\frac{\sigma^2}{K}.
\end{align}
Note that, by the definition of $\Psi^t$, it holds that
\begin{align}
    \frac{\Psi^0-\Psi^{T+1}}{\gamma T}\lesssim &\frac{L\Delta}{\gamma T}+\frac{ N}{\beta S}\frac{\EE[\|v^{0}-\nabla f(x^{0})\|^2]}{T}\\
    &\quad +\frac{ N}{(1-q)S}\frac{\frac{1}{N}\sum_i\EE[\|v_i^{0}-c_i^{0}\|^2]}{T}+\left(\frac{1}{S}+\frac{\beta N}{(1-q)^2S}\right)\frac{\frac{1}{N}\sum_i\EE[\|v_i^{0}-\nabla f_i(x^{0})\|^2]}{T}\\
    \lesssim&\frac{L\Delta}{T}\left(\frac{N}{\beta S}+\frac{N}{(1-q)S}\right)+\frac{\sigma^2}{BT}\left(\frac{1}{\beta S}+\frac{ N}{(1-q)S}+\frac{\beta N}{(1-q)^2S}\right)
\end{align}
where we use the choice of $\gamma$ and the initialization of $\{v_i^0\}_{i\in[N]}$, $\{c_i^0\}_{i\in[N]}$, and $c^0$ in the second inequality. Due to the choice of $B$, we have 
\begin{equation}
    \frac{\sigma^2}{B}\left(\frac{1}{\beta S}+\frac{ N}{(1-q)S}+\frac{\beta N}{(1-q)^2S}\right)\lesssim L\Delta\left(\frac{N}{\beta S}+\frac{N}{(1-q)S}\right)
\end{equation}
and consequently
\begin{align}
    \frac{1}{T}\sum_{t=0}^{T-1}\EE[\|\nabla f(x^t)\|^2]
    \lesssim &\frac{L\Delta}{T}\left(\frac{N}{\beta S}+\frac{N}{(1-q)S}\right)+\left(\frac{\beta}{N}+\frac{\beta^2}{1-q}+\frac{\beta^3}{(1-q)^2}\right)\frac{\sigma^2}{K}.
\end{align}
Plugging the choice of $\beta$ into \eqref{eqn:fgndasifna} completes the proof.

\end{proof}

\newpage

\end{document}